\documentclass[12pt,reqno]{amsart}

\usepackage{amsmath,amssymb,ifthen}
\usepackage{fullpage}
\usepackage{graphicx,psfrag,subfigure}
\usepackage{color}
\usepackage{moreverb}
\usepackage{amsthm}
\usepackage{mathrsfs}
\usepackage{esint}  
\usepackage{graphicx}
\usepackage{psfrag}
\usepackage{hhline}
\usepackage{enumerate}
\usepackage{bbm}
\usepackage[foot]{amsaddr}

\def\H{\widetilde{H}}
\def\MM{\mathcal M}

\def\R{{\mathbb R}}
\def\K{{\mathbb K}}
\def\Z{\mathbb {Z}}
\def\N{{\mathbb N}}

\def\KK{{\mathcal K}}
\def\NN{{\mathcal N}}

\def\OO{{\mathcal O}}

\def\TT{{\mathcal T}}
\def\XX{{\mathcal X}}
\def\VV{{\mathfrak V}}
\def\WW{\mathfrak{W}}
\def\WWe{\mathcal{W}}
\def\YY{{\mathcal Y}}
\def\ZZ{\mathcal{Z}}
\def\diam{{\rm diam}}
\def\edual#1#2{\langle\hspace*{-1mm}\langle#1\,,\,#2\rangle\hspace*{-1mm}\rangle_{\WW}}
\def\edualV#1#2{\langle\hspace*{-1mm}\langle#1\,,\,#2\rangle\hspace*{-1mm}\rangle_{\VV}}
\def\norm#1#2{\|#1\|_{#2}}

\def\enorm#1{|\hspace*{-.5mm}|#1|\hspace*{-.5mm}|_{\WW}}
\def\enormV#1{|\hspace*{-.5mm}|#1|\hspace*{-.5mm}|_{\VV}}
\def\set#1#2{\big\{#1\,:\,#2\big\}}

\def\dual#1#2{\langle#1\,,\,#2\rangle}

\def\diag{{\rm diag}}
\def\gal{{\mathcal{G}}}
\def\level{{\rm level}}
\def\QQ{{\mathcal{Q}}}
\def\gen{\mathrm{gen}}

\def\hyp{\mathfrak{W}} 
\def\prec{\mathcal{P}} 
\def\PAS{\widetilde{\mathcal{P}}_{L}^{\WW}}
\def\PASV{\widetilde{\mathcal{P}}_{L}^{\VV}}
\def\PPAS{\widetilde{\mathbf{P}}_{L}^{\WW}}
\def\PPASV{\widetilde{\mathbf{P}}_L^{\VV}}
\def\SSAS{\widetilde{\mathbf{S}}_{L}^{\WW}}
\def\SSASV{\widetilde{\mathbf{S}}_L^{\VV}}
\def\AA{\mathbf{A}}

\def\BB{\mathbf{B}}

\def\evmin{\lambda_{\rm min}^\WW}
\def\evmax{\lambda_{\rm max}^\WW}
\def\evminV{\lambda_{\rm min}^\VV}
\def\evmaxV{\lambda_{\rm max}^\VV}

\def\IW{{\rm id}_{\ell\to L}^\WW}
\def\IV{{\rm id}_{\ell\to L}^\VV}
\def\IIW{\mathbf{{ id}}_{\ell\to L}^\WW}
\def\IIV{\mathbf{{ id}}_{\ell\to L}^\VV}
\def\HW{{{\rm id}}_{\widetilde\ell\to \ell}^\WW}

\def\HHW{{\mathbf{{ id}}}_{\widetilde\ell\to \ell}^\WW}
\def\HHV{{\mathbf{{ id}}}_{\widetilde\ell\to \ell}^\VV}

\def\cond{{\rm cond}}

\def\refine{{\tt ref}}
 
\def\linhull{{\rm span}} 

\def\supp{{\rm supp}}

\def\Cnorm{C_{\rm norm}}

\def\Cscott{C_{\rm sz}}
\def\Cinv{C_{\rm inv}}

\def\Cscale{C_{\rm scale}}

\newcounter{constantsnumber}
\def\namec#1#2{%
  \ifthenelse{\equal{#1}{lipschitz}}{C_{\rm lip}}{%
  \ifthenelse{\equal{#1}{c:unifEquivLevel}}{C_{\rm level}}{%
  \ifthenelse{\equal{#1}{mark}}{C_{\rm mark}}{%
  \ifthenelse{\equal{#1}{basis}}{C_{\rm basis}}{%
  \ifthenelse{\equal{#1}{monotone}}{C_{\rm mon}}{%
  \ifthenelse{\equal{#1}{cea}}{C_{\mbox{\rm\scriptsize C\'ea}}}{%
  \ifthenelse{\equal{#1}{norm}}{C_{\rm norm}}{%
  \ifthenelse{\equal{#1}{mon}}{{C}_{\rm mon}}{
  \ifthenelse{\equal{#1}{lip}}{{C}_{\rm lip}}{
  \ifthenelse{\equal{#1}{monA}}{c_{\rm mon}}{
  \ifthenelse{\equal{#1}{lipA}}{c_{\rm lip}}{
  \ifthenelse{\equal{#1}{normequiv1}}{c_{\rm norm}}{ 
  \ifthenelse{\equal{#1}{inv}}{C_{\rm inv}}{ 
  \ifthenelse{\equal{#1}{inv2}}{\widetilde{C}_{\rm inv}}{ 
  \ifthenelse{\equal{#2}{newcounter}}{\refstepcounter{constantsnumber}\label{const#1}}{}C_{\ref{const#1}}}%
  }}}}}}}}}}}}}}
\def\setc#1{\namec{#1}{newcounter}}

\numberwithin{equation}{section}
\numberwithin{figure}{section}
\newtheorem{theorem}{Theorem}[section]
\newtheorem{proposition}[theorem]{Proposition}
\newtheorem{remark}[theorem]{Remark}
\newtheorem{lemma}[theorem]{Lemma}
\newtheorem{corollary}[theorem]{Corollary}
\newtheorem{algorithm}[theorem]{Algorithm}

\def\subsection#1
{
 \medskip

 \refstepcounter{subsection}
 {\noindent\bf\arabic{section}.\arabic{subsection}.~#1.~}
}

\def\refine{{\tt refine}}
\def\MM{\mathcal M}
\def\II{\widetilde{\mathcal I}}

\begin{document}

\title[Optimal additive Schwarz preconditioning for adaptive 2D IGABEM]{Optimal additive Schwarz preconditioning\\for adaptive 2D IGA boundary element methods}
\date{\today}

\author{Thomas F\"uhrer}
\address[TF]{Facultad de Matem\'aticas, Pontificia Universidad Católica de Chile, Vicku\~na Mackenna 4860, Santiago, Chile}
\email{tofuhrer@mat.uc.cl}
\author{Gregor Gantner}
\email{gregor.gantner@asc.tuwien.ac.at}
\author{Dirk Praetorius}
\email{dirk.praetorius@asc.tuwien.ac.at}
\author{Stefan Schimanko}
\email{stefan.schimanko@asc.tuwien.ac.at}
\address[GG, DP, SS]{Institute for Analysis and Scientific Computing, TU Wien, Wiedner Hauptstr. 8-10, 1040 Wien, Austria}

\thanks{{\bf Acknowledgement.} 
TF was supported by CONICYT through the FONDECYT project \textit{Least-squares methods for obstacle problems} (grant P11170050).
The authors GG, DP, and SS were supported by the Austrian Science Fund (FWF) through the research projects \textit{Optimal isogeometric boundary element method} (grant P29096) and \textit{Optimal adaptivity for BEM and FEM-BEM coupling} (grant P27005), the doctoral school \textit{Dissipation and dispersion in nonlinear PDEs} (grant W1245), and the special research program \textit{Taming complexity in PDE systems} (grant SFB F65).
}

\begin{abstract}
 We define and analyze (local) multilevel diagonal preconditioners for isogeometric boundary elements on locally refined meshes in two dimensions. 
 Hypersingular and weakly-singular integral equations are considered. 
 We prove that the condition number of the preconditioned systems of linear equations is independent of the mesh-size and the refinement level. 
 Therefore, the computational complexity, when using appropriate iterative solvers, is optimal. 
 Our analysis is carried out for closed and open boundaries and numerical examples confirm our theoretical results.
\end{abstract}
\maketitle

\section{Introduction}
%

In the last decade, the isogeometric analysis (IGA) had a strong impact on the field of scientific computing and 
numerical analysis.
 We refer, e.g., to the pioneering work~\cite{pioneer} and to \cite{bible,overview} for an introduction to the field.
The basic idea is to utilize the same ansatz functions for approximations as are used for the description of the geometry by some computer aided design (CAD) program.
Here, we consider the case, where the geometry is represented by rational splines.
For certain problems, where the fundamental solution is known, the boundary element method (BEM) is attractive since CAD programs usually only provide a parametrization of the boundary $\partial\Omega$ and not of the volume $\Omega$ itself. 
Isogeometric BEM (IGABEM) has first been considered for 2D BEM in~\cite{igabem2d} and for 3D BEM in \cite{igabem3d}.
We refer {to~\cite{SBTR,helmholtziga,simpson,tran}} for numerical experiments, to \cite{cad2wave,TM,zechner,dhp16,wolf18,wolf_new} for fast IGABEM based on wavelets, fast multipole, $\mathcal{H}$-matrices resp.\ $\mathcal{H}^2$-matrices, and 
to~\cite{stokesiga,keuchel,sampoli,falini2018study} for some quadrature analysis.

Recently, adaptive IGABEM  has been analyzed in \cite{igafaermann,resigabem,giannelli_new} for rational  resp.\ hierarchical splines in 2D  and 
optimal algebraic convergence rates have been proven in~\cite{optigabem,hypiga}  for rational splines in 2D resp.\ in \cite{diss} for hierarchical splines in 3D.
In~2D, the corresponding adaptive algorithms allow for both $h$-refinement as well as regularity reduction via knot multiplicity increase. 
Usually, it is assumed that the resulting systems of linear equations are solved exactly. 
In practice, however, iterative solvers are used and their effectivity hinges on the condition number of the Galerkin matrices.
It is well-known that the condition number of Galerkin matrices corresponding to the discretization of certain
integral operators depend not only on the number of degrees of freedom but also on the ratio
$h_\mathrm{max}/h_\mathrm{min}$ of the largest and smallest element diameter, which can become arbitrarily large on
locally refined meshes; see, e.g.,~\cite{amt99} for the case of affine boundary elements and lowest-order ansatz functions.
Therefore, the construction of optimal preconditioners is a necessity. We say that  a preconditioner is optimal, if the condition
number of the resulting preconditioned matrices is independent of the mesh-size function $h$, the number of degrees of
freedom and the refinement level.

In this work, we consider simple additive Schwarz methods. 
The central idea of our local multilevel diagonal preconditioners is to use newly created nodes 
and old nodes whose  multiplicity has changed, to define local diagonal scalings on each refinement level.
This allows us to prove optimality of the proposed preconditioner and the computational complexity for applying our
preconditioner is linear with respect to the number of degrees of freedom on the finest mesh.
In particular, this extends our prior works~\cite{ffpsfembem,ffps,optpcg} on local multilevel diagonal preconditioners for hypersingular
integral equations  and weakly-singular integral equations for affine geometries  in 2D and 3D and lowest-order
discretatizations.
Other results on Schwarz methods for BEM with affine boundaries are found in~\cite{cao,transtep96,tsm97}, mainly for
uniform mesh-refinements and in ~\cite{amcl03} for some specially local refined meshes.
For the higher order case, we refer, e.g., to \cite{heuer,fmpr}.
Diagonal preconditioners for BEM are covered in~\cite{amt99,gm06}.
Another preconditioner technique that leads to uniformly bounded condition numbers 
is based on the use of integral operators of opposite order. The case of closed
boundaries is analyzed in~\cite{sw98}, whereas open boundaries are treated in the recent
works~\cite{hjhut2014,HJU16_646,HJU17_709}.
 The recent work \cite{stevenson18} deals with the opposite order operator preconditioning technique in Sobolev spaces of negative order.

To the best of our knowledge, the preconditioning of IGABEM, even on uniform meshes, is still an open problem.
For isogeometric finite elements (IGAFEM), a BPX-type preconditioner  is analyzed in~\cite{bhksBPX} on uniform meshes, where the authors
consider general pseudodifferential operators of positive order.
Recently, a BPX-type preconditioner with local smoothing has been analyzed in~\cite{BPXadapIGA} for locally refined
T-meshes.
Other multilevel preconditioners for IGAFEM have been studied in~\cite{MR3002801,sangalli16,MR3610090,takacs17}  for uniform resp. in \cite{MR3522271} for hierarchical meshes, and domain
decomposition methods can be found in~\cite{MR3037302,MR3216651,MR3612901} for uniform meshes.

\noindent\textbf{Model problem.}
Let $\Omega$ be a  bounded simply connected Lipschitz domain in 
$\R^2$, with piecewise smooth boundary $\partial\Omega$ and let 
$\Gamma\subseteq\partial\Omega$ be a  connected subset with Lipschitz boundary $\partial\Gamma$. 
Neumann screen problems on $\Gamma$ yield the weakly-singular integral equation
\begin{align}
\label{eq:hypsing}
\WW u(x) := -\frac{\partial}{\partial\nu_x}\int_\Gamma
\Big(\frac{\partial}{\partial\nu_y}\,G(x,y)\Big)\,u(y)\,dy
= f(x)
\quad\text{for all }x\in\Gamma
\end{align}
with the hypersingular integral operator $\WW$ and some given right-hand side $f$. 
Here, $\nu_x$ denotes the outer normal unit vector of $\Omega$ at some point 
$x\in\Gamma$, and 
\begin{align}
 G(x,y) :=
 -\frac{1}{2\pi}\,\log|x-y|
\end{align}
is the fundamental solution of the Laplacian. 
Similarly, Dirichlet screen problems lead to the weakly-singular integral equation
\begin{align}
\label{eq:weaksing}
\VV \phi(x) := \int_\Gamma
\,G(x,y)\,\phi(y)\,dy
= g(x)
\end{align}
with the weakly-singular integral operator $\VV$ and some given right-hand side $g$. 

\noindent\textbf{Outline.}
The remainder of the work is organized as follows: 
Section~\ref{section:preliminaries} provides the functional analytic setting of the boundary integral operators, the
definition of the mesh, B-splines and NURBS together with their basic properties.
Auxiliary results that are used in the proof of our main results are stated in Section~\ref{section:aux}.
In Section~\ref{section:precond}, we define our local multilevel diagonal preconditioner for the hypersingular integral
operator on closed and open boundaries and prove its optimality (Theorem~\ref{thm:PAS}).
Then, in Section~\ref{section:weakly precond}, we extend our local multilevel diagonal preconditioner to the
weakly-singular case and give a proof of its optimality (Theorem~\ref{thm:PASV}).
Finally, in Section~\ref{sec:numerics} we restate the abstract results for additive Schwarz operators in matrix
formulation (Corollary~\ref{cor:main}).
Moreover, numerical examples for closed and open boundaries are presented and some aspects of
implementation are discussed.


\section{Preliminaries}
\label{section:preliminaries}
\subsection{Notation}
Throughout and without any ambiguity, $|\cdot|$ denotes the absolute value of scalars, the Euclidean norm of vectors in $\R^2$, the measure of a set in $\R$ (e.g., the length of an interval), or the arclength of a curve in $\R^{2}$.
We write $A\lesssim B$ to abbreviate $A\le cB$ with some generic constant $c>0$, which is clear from the context.
Moreover, $A\simeq B$ abbreviates $A\lesssim B\lesssim A$. 
Throughout, mesh-related quantities have the same index, e.g., $\NN_{\bullet}$ is the set of nodes of the partition $\TT_{\bullet}$, and $h_{\bullet}$ is the corresponding local mesh-width etc. 
The analogous notation is used for partitions $\TT_{\circ}$ resp.\ $\TT_\ell$ etc. We sometimes use  \,$\widehat{\cdot}$\, to transform notation on the boundary to the parameter domain.
The most important symbols are listed in Table~\ref{table}.

\begin{table}
 \caption{Important symbols}
 \label{table}\begin{minipage}{0.49\textwidth}
{\tiny \begin{tabular}{lll}
\hline
Name & Description &First appearance \\
\hline
$\widehat B_{\bullet,i,p}$ & B-spline& Section \ref{subsec:splines}\\
${B}_{\bullet,i,q}$& B-spline on $\Gamma$&Section \ref{section:igabem}\\
$\overline B_{\bullet,i,p}$&continuous B-spline on $\Gamma$&Section \ref{section:igabem}\\
$\widehat B_{\bullet,i,p}^*$&dual B-spline&Section \ref{section:scott}\\
$\gen(\cdot)$& generation fct.&Section \ref{section:multilevel}\\
$\widehat h_{\bullet,T} $&length in $[a,b]$ &Section \ref{section:boundary:discrete}\\
$h_{\bullet,T} $&arclength&Section \ref{section:boundary:discrete}\\
$\widehat h_\bullet(\cdot), h_{\bullet}(\cdot)$&mesh-size functions&Section \ref{section:boundary:discrete}\\
$\overline h_{{\rm uni}(m)}$&uniform mesh-size in $[a,b]$&Section \ref{section:multilevel}\\
$\II_{\ell}$ &index set&Section \ref{section:precond}\\
$J_{\bullet}$&Scott-Zhang operator&Section \ref{section:scott}\\
$\widehat \KK_{\bullet}$&knot vector in $[a,b]$&Section \ref{section:boundary:discrete}\\
${\KK}_{\bullet}$&knot vector&Section \ref{section:boundary:discrete}\\
$\widetilde\NN_{\circ\setminus\bullet}$& set of new knots&Section \ref{section:scott}\\
$\K$&admissible knot vectors&Section \ref{section:mesh-refinement}\\
$\level_{\bullet}(\cdot)$&level fct.&Section \ref{section:multilevel}\\
$N_{\bullet}$&number of knots& Section \ref{section:boundary:discrete}\\
$\NN_{\bullet}$& set of nodes&Section \ref{section:boundary:discrete}\\
$o$& $1$ if $\Gamma=\partial\Omega$, $0$ else&Section~\ref{section:igabem} \\
$p$&positive polynomial order&Section \ref{section:boundary:discrete}\\
$\PASV$&Schwarz operator for $\VV$&Section \ref{section:weakly precond}\\
$\PAS$&Schwarz operator for $\WW$&Section \ref{section:precond}\\
$\refine(\cdot)$&set of refined knot vectors &Section \ref{section:mesh-refinement}\\
\hline
\end{tabular}}
\end{minipage}
\begin{minipage}{0.49\textwidth}
{\tiny\begin{tabular}{lll}
\hline
Name & Description &First appearance\\
\hline
$\widehat R_{\bullet,i,p}$&NURBS&Section \ref{subsec:splines}\\
$R_{\bullet,i,p}$&NURBS on $\Gamma$&Section \ref{section:igabem}\\
$\overline R_{\bullet,i,p}$&continuous NURBS on $\Gamma$&Section \ref{section:igabem}\\
$\widehat R_{\bullet,i,p}^*$&dual NURBS&Section \ref{section:scott}\\
$t_{\bullet,i}$&knot&Section \ref{subsec:splines}\\
$\TT_{\bullet}$&mesh&Section \ref{section:boundary:discrete}\\
$w_{\bullet,i}$&weight&Section \ref{subsec:splines}\\
$\widehat w(\cdot)$&denominator&Section \ref{section:igabem}\\
$w_{\min}, w_{\max}$&bounds for weights&Section \ref{section:igabem}\\
$\WWe_{\bullet}$&weight vector&Section \ref{section:igabem}\\
$\XX_{\bullet}$&NURBS space for $\WW$&Section \ref{section:igabem}\\
$\widetilde\XX_\ell$&space of new NURBS&Section \ref{section:precond}\\
$\XX_{\ell,i}$&one-dim. subspace&Section \ref{section:precond}\\
$\YY_{\bullet}$&spline space for $\VV$&Section \ref{section:weakly precond}\\
$\widetilde\YY_\ell^0$&space of new splines&Section \ref{section:weakly precond} \\
$z_{\bullet,j}$&node&Section \ref{section:boundary:discrete}\\
$\widetilde{z}_{\ell,i}$&new knot&Section \ref{section:precond}\\
${\widehat \kappa}_{\bullet}$& local mesh-ratio in $[a,b]$&Section \ref{section:boundary:discrete}\\
${\widehat \kappa}_{\max}$& bound for local mesh-ratio in $[a,b]$&Section \ref{section:mesh-refinement}\\
$\Pi_{{\rm uni}(m)}$&projection on unif. space&Section \ref{section:multilevel}\\
$\omega_{\bullet}^m(\cdot)$&patch&Section \ref{section:boundary:discrete}\\
$\#_{\bullet}$&multiplicity&Section \ref{section:boundary:discrete}\\
\hline
 \end{tabular}}
 \end{minipage}\end{table}

\def\Cgamma{C_\Gamma}
\subsection{Sobolev spaces}
\label{section:sobolev}
The usual Lebesgue and Sobolev spaces on $\Gamma$ are denoted by $L^2(\Gamma)=H^0(\Gamma)$ and  
$H^1(\Gamma)$. 
We  introduce the corresponding seminorm on any measurable subset $\Gamma_0\subseteq\Gamma$ via
 \begin{align}
 |v|_{H^{1}({\Gamma_0})} := \norm{\partial_\Gamma v}{L^2(\Gamma_0)}\quad \text{for all }v\in H^1(\Gamma),
 \end{align}
with the arclength derivative $\partial_\Gamma$.
We have that
\begin{align}
\norm{v}{H^1(\Gamma)}^2= \norm{v}{L^2(\Gamma)}^2+|v|_{H^1(\Gamma)}^2\quad\text{for all }v\in H^1(\Gamma),
\end{align}

Moreover, $\H^1(\Gamma)$ is the space of 
$H^1(\Gamma)$ functions, which have a vanishing trace on the relative boundary 
$\partial\Gamma$ equipped with the same norm.
On $\Gamma$, Sobolev spaces of fractional order $0<\sigma<1$ are defined by
the 
$K$-method of interpolation~\cite[Appendix B]{mclean}: 
For $0<\sigma<1$, we let
$H^\sigma(\Gamma) := [L^2(\Gamma),H^1(\Gamma)]_\sigma$ 
and $\H^\sigma(\Gamma):=[L^2(\Gamma),\H^1(\Gamma)]_\sigma$.
We also introduce the Sobolev-Slobodeckij seminorm 
\begin{align}\label{eq:ss seminorm}
|v|_{H^{\sigma}(\Gamma_0)}:=\Big(\int_{\Gamma_0}\int_{\Gamma_0}\frac{|v(x)-v(y)|^2}{|x-y|^{1+2\sigma}}\,dx\,dy\Big)^{1/2} \quad\text{for all }v\in H^\sigma(\Gamma).
\end{align}

For $0<\sigma\le1$, Sobolev spaces of negative order are defined by duality
$H^{-\sigma}(\Gamma) := \H^{\sigma}(\Gamma)^*$ and 
$\H^{-\sigma}(\Gamma) := H^{\sigma}(\Gamma)^*$, where duality is understood with
respect to the extended $L^2(\Gamma)$-scalar
product $\dual\cdot\cdot_\Gamma$. 
In general, there holds the continuous inclusion $\H^{\pm \sigma}(\Gamma)\subseteq H^{\pm \sigma}(\Gamma)$ with $\norm{v}{H^{\pm \sigma}(\Gamma)}\lesssim\norm{v}{\H^{\pm \sigma}(\Gamma)}$ for all $v\in \H^{\pm\sigma}(\Gamma)$.
We note that $\H^{\pm \sigma}(\Gamma)
= H^{\pm \sigma}(\Gamma)$ for $0<\sigma<1/2$ with equivalent norms. Moreover,
it holds for $\Gamma=\partial\Omega$ that $\H^{\pm \sigma}(\partial\Omega) = H^{\pm \sigma}(\partial\Omega)$ 
even with equal norms for all $0<\sigma\le1$.
Finally, the treatment of the closed boundary $\Gamma=\partial\Omega$ 
requires the definition of $H^{\pm \sigma}_0(\partial\Omega) = \set{v\in H^{\pm \sigma}(\partial\Omega)}{\dual{v}{1}_{\partial\Omega} = 0}$
for all $0\le \sigma\le1$.

Details and equivalent definitions of the Sobolev spaces 
are, found, e.g., in \cite{mclean,ss}.

\subsection{Hypersingular integral equation}
\label{subsec:hypsing}
For $0\le \sigma\le1$, the hypersingular integral operator 
$\WW:\H^{\sigma}(\Gamma)\to H^{\sigma-1}(\Gamma)$ is well-defined, linear, and continuous. Recall that $\Gamma$ and $\partial\Omega$ are supposed to be connected.

For $\Gamma\subsetneqq\partial\Omega$ and $\sigma=1/2$, 
$\WW:\H^{1/2}(\Gamma)\to H^{-1/2}(\Gamma)$ is symmetric and 
elliptic. 
Hence,
\begin{align}
\edual{u}{v}:=\dual{\WW u}{v}_\Gamma\quad\text{for all }u,v\in \H^{1/2}(\Gamma),
\end{align}
 defines an equivalent scalar product on  
$\H^{1/2}(\Gamma)$ with corresponding norm $\enorm{\cdot}$.

For  $\Gamma=\partial\Omega$, the 
operator $\WW$ is symmetric and elliptic up to the constant functions, i.e., 
$\WW:H^{1/2}_0(\partial\Omega)\to H^{-1/2}_0(\partial\Omega)$ is 
elliptic. In particular,
\begin{align}
\edual{u}{v}:=\dual{\WW u}{v}_{\partial\Omega} + \dual{u}{1}_{\partial\Omega}\dual{v}{1}_{\partial\Omega}\quad \text{for all }u,v\in \H^{1/2}(\Gamma),
\end{align}
defines an equivalent scalar product on $H^{1/2}(\partial\Omega)=\H^{1/2}(\partial\Omega)$  with  norm $\enorm{\cdot}$.

With this notation and provided that $f\in H^{-1/2}_0(\Gamma)$ in
case of $\Gamma = \partial\Omega$, the strong form~\eqref{eq:hypsing} is equivalently stated in variational form: Find $u\in \widetilde H^{1/2}(\Gamma)$ such that
\begin{align}
\label{eq:weakform2}
 \edual{u}{v} = \dual{f}{v}_\Gamma
 \quad\text{for all }v\in \H^{1/2}(\Gamma).
\end{align}
Therefore, the  Lax-Milgram lemma applies and hence \eqref{eq:hypsing}  admits a unique solution $u\in \H^{1/2}(\Gamma)$.
More details and proofs are found, e.g., in \cite{mclean,ss,s}.
\subsection{Weakly-singular integral equation}
\label{subsec:hypsing}
For $0\le \sigma\le1$, the weakly-singular integral operator 
$\VV:\H^{\sigma-1}(\Gamma)\to H^{\sigma}(\Gamma)$ is well-defined, linear, and continuous. 
For $\Gamma=\partial\Omega$, we suppose $\diam(\Omega)<1$.

For $\sigma=1/2$, 
$\VV:\H^{-1/2}(\Gamma)\to H^{1/2}(\Gamma)$ is symmetric and 
elliptic. 
In particular,
\begin{align}
\edualV{\phi}{\psi}:=\dual{\VV \phi}{\psi}_\Gamma\quad \text{for all }\phi,\psi\in \H^{-1/2}(\Gamma),
\end{align}
 defines an equivalent scalar product on 
$\H^{-1/2}(\Gamma)$ with corresponding norm $\enormV{\cdot}$.
With this notation, the strong form~\eqref{eq:weaksing} with data $g\in H^{1/2}(\Gamma)$ is equivalently stated by
\begin{align}
\label{eq:weakform2}
 \edualV{\phi}{\psi} = \dual{g}{\psi}_\Gamma
 \quad\text{for all }\psi\in \H^{-1/2}(\Gamma).
\end{align}
Therefore, the Lax-Milgram lemma applies and hence  \eqref{eq:weaksing} admits a unique solution $\phi\in \H^{-1/2}(\Gamma)$.
More details and proofs are found, e.g.,  in  \cite{mclean,ss,s}.

\subsection{Boundary parametrization}
\label{subsec:boundary parametrization}
We assume that either $\Gamma=\partial\Omega$ is parametrized by a closed continuous and
piecewise continuously differentiable path $\gamma:[a,b]\to\Gamma$ with $a<b$ such
that the restriction $\gamma|_{[a,b)}$ is even bijective, or that $\Gamma\subsetneqq\partial\Omega$ is parametrized by a bijective continuous and piecewise continuously differentiable path $\gamma:[a,b]\to\Gamma$.  In the first case, we speak of \textit{closed} $\Gamma=\partial\Omega$, whereas the second case is referred to as \textit{open} $\Gamma\subsetneqq\partial\Omega$.
For closed $\Gamma=\partial\Omega$, we denote the $(b-a)$-periodic extension to $\R$ also by $\gamma$.

For the left and right derivative of $\gamma$, we assume that {$\gamma^{\prime_\ell}(t)\neq 0$ for $t\in(a,b]$ and $\gamma^{\prime_r}(t)\neq 0$  for $t\in [a,b)$.}
Moreover, we assume that $\gamma^{\prime_\ell}(t)
+c\gamma^{\prime_r}(t)\neq0$ for all $c>0$ {and $t\in[a,b]$ resp.\ $t\in(a,b)$.} 
Finally, let $\gamma_{\rm arc}:[0,|\Gamma|]\to\Gamma$ denote the arclength parametrization, i.e.,
$|\gamma_{\rm arc}^{\prime_\ell}(t)| = 1 = |\gamma_{\rm arc}^{\prime_r}(t)|$, and its periodic extension. Elementary
differential geometry yields bi-Lipschitz continuity
\begin{align}\label{eq:bi-Lipschitz}
 \Cgamma^{-1} \le \frac{|\gamma_{\rm arc}(s)-\gamma_{\rm arc}(t)|}{|s-t|}\le\Cgamma
 \quad\text{for }s,t\in\R, {\text{ with }\begin{cases}
 |s-t|\le \frac{3}{4}\,|\Gamma|, \text{ for closed }\Gamma,\\
  s\neq t\in [0,|\Gamma|], \text{ for open }\Gamma,
\end{cases}}
\end{align}
where $\Cgamma>0$ depends only on $\Gamma$.
A proof is given in \cite[Lemma 2.1]{diplarbeit} for closed $\Gamma=\partial\Omega$. 
For open $\Gamma\subsetneqq\partial\Omega$, the proof is even simpler. 

\subsection{Boundary discretization}
\label{section:boundary:discrete}
In the following, we describe the different quantities, which define the discretization.

{\bf Nodes $\boldsymbol{z_{\bullet,j}=\gamma(\widehat{z}_{\bullet,j})\in\mathcal{N}_{\bullet}}$ and number of nodes $\boldsymbol{n_{\bullet}}$.}\quad Let $\mathcal{N}_{\bullet}:=\set{z_{\bullet,j}}{j=1,\dots,n_{\bullet}}$ and $z_{\bullet,0}:=z_{\bullet,n_{\bullet}}$ for closed $\Gamma=\partial\Omega$ resp.\ $\mathcal{N}_{\bullet}:=\set{z_{\bullet,j}}{j=0,\dots,n_{\bullet}}$ for open $\Gamma\subsetneqq \partial\Omega$ be a set of nodes. We suppose that $z_{\bullet,j}=\gamma(\widehat{z}_{\bullet,j})$ for some $\widehat{z}_{\bullet,j}\in[a,b]$ with
$a=\widehat{z}_{\bullet,0}<\widehat{z}_{\bullet,1}<\widehat{z}_{\bullet,2}<\dots<\widehat{z}_{\bullet,n_{\bullet}}=b$ such that 
$\gamma|_{[\widehat{z}_{\bullet,j-1},\widehat{z}_{\bullet,j}]}\in C^1([\widehat{z}_{\bullet,j-1},\widehat{z}_{\bullet,j}])$.

{\bf Multiplicity $\boldsymbol{\#_{\bullet} z_{\bullet,j}}$, knot vector $\boldsymbol{\KK_{\bullet}}$ and number of knots $\boldsymbol{N_{\bullet}}$.}\quad
Let $p\in\N$ be some fixed positive polynomial order.
{Each  interior node $z_{\bullet,j}$ has a multiplicity $\#_{\bullet} z_{\bullet,j}\in\{1,2\dots, p\}$ and $\#_{\bullet} {z}_{\bullet,0}=\#_{\bullet} z_{n_{\bullet}}=p+1$}.
This induces knots 
\begin{align}
\KK_{\bullet}=(\underbrace{z_{\bullet,k},\dots,z_{\bullet,k}}_{\#_{\bullet} z_{\bullet,k}-\text{times}},\dots,\underbrace{z_{\bullet,n_{\bullet}},\dots,z_{\bullet,n_{\bullet}}}_{\#_{\bullet} z_{\bullet,n_{\bullet}}-\text{times}}),
\end{align} 
with $k=1$ for $\Gamma=\partial\Omega$ resp.  $k=0$ for $\Gamma\subsetneqq\partial\Omega$.
We define the number of knots in $\gamma((a,b])$ as
\begin{align}
N_{\bullet}:=\sum_{j=1}^{n_{\bullet}} \#_{\bullet} z_{\bullet,j}.
\end{align}

{\bf Elements $\boldsymbol{T_{\bullet,j}}$, partition $\boldsymbol{\mathcal{T}_{\bullet}}$.} \quad
Let $\mathcal{T}_{\bullet}=\{T_{\bullet,1},\dots,T_{\bullet,n_{\bullet}}\}$ be a partition of $\Gamma$ into compact and connected segments $T_{\bullet,j}=\gamma(\widehat{T}_{\bullet,j})$ with $\widehat{T}_{\bullet,j}=[\widehat{z}_{\bullet,j-1},\widehat{z}_{\bullet,j}]$.

{\bf Local mesh-sizes $\boldsymbol{\widehat h_{\bullet,T}}, \boldsymbol{h_{\bullet,T}} $ and $\boldsymbol{\widehat h_{\bullet}}, \boldsymbol{h_{\bullet}}$.}\quad
For $T\in\TT_\bullet$, we define $\widehat h_{\bullet,T}:=|\gamma^{-1}(T)|$ as its length in the parameter domain, and $h_{\bullet,T}:=|T|$ as its arclength.
We define the local mesh-width functions $\widehat h_{\bullet}, h_{\bullet}\in L^\infty(\Gamma)$ by $\widehat h_{\bullet}|_T=\widehat h_{\bullet,T}$ and $h_{\bullet}|_T=h_{\bullet,T}$.

{\bf Local mesh-ratio $ \boldsymbol{{\widehat \kappa}_{\bullet}}$.}\quad
We define the {local mesh-ratio} by
\begin{align}\label{eq:meshratio}
 \widehat \kappa_{\bullet}&:=\max\set{\widehat h_{\bullet,T}/\widehat h_{\bullet,T'}}{{T},{T}'\in\TT_{\bullet} \text{ with }  T\cap T'\neq \emptyset}.
\end{align}

{\bf Patches $\boldsymbol{\omega_{\bullet}^m(z)}$ and $\boldsymbol{\omega_{\bullet}^m(\Gamma_0)}$.}
For each set $\Gamma_0\subseteq\Gamma$, we  {inductively define for $m\in\N_0$}
{\begin{align*}
 \omega_{\bullet}^m(\Gamma_0) :=\begin{cases} \Gamma_0\quad&\text{if }m=0,\\
 \omega_{\bullet}(\Gamma_0):= \bigcup\set{T\in \TT_{\bullet}}{T\cap \Gamma_0\neq \emptyset}\quad&\text{if }m=1,\\
 \omega_{\bullet}(\omega_{\bullet}^{m-1}(\Gamma_0)) \quad&\text{if }m>1.\end{cases}
\end{align*}}
For points $z\in\Gamma$, we abbreviate $\omega_{\bullet}(z):=\omega_{\bullet}(\{z\})$ and $\omega_\bullet^m(z):=\omega_\bullet^m(\{z\})$.

\subsection{Admissible knot vectors}\label{section:mesh-refinement}
Throughout, we consider  families of knot vectors $\KK_{\bullet}$ as in Section~\ref{section:boundary:discrete} with uniformly bounded local mesh-ratio, i.e., we suppose the existence of $\widehat \kappa_{\rm max}\ge 1$ with 
\begin{align}\label{eq:kappamax}
{\widehat \kappa}_{\bullet}\le {\widehat \kappa}_{\max}.
\end{align}
Let $\KK_{\bullet}$ and $\KK_{\circ}$ be knot vectors  \eqref{eq:kappamax}.
We say that  $\KK_{\circ}$ is finer than $\KK_{\bullet}$ and write $\KK_{\circ}\in\refine(\KK_{\bullet})$ if $\KK_{\bullet}$ is a subsequence  of $\KK_{\circ}$ such that $\KK_{\circ}$ is obtained from $\KK_{\bullet}$  via iterative dyadic bisections in the parameter domain and multiplicity increases.
Formally, this means that $\NN_{\bullet}\subseteq\NN_{\circ}$ with  $\#_{\bullet} z\le\#_{\circ} z$ for all $z\in\NN_{\bullet}\cap \NN_{\circ}$, and that for all $T\in\TT_{\circ}$ there exists $T'\in\TT_{\bullet}$ and $j\in\N_0$  with $T\subseteq T'$ and $|\gamma^{-1}(T')|=2^{-j}|\gamma^{-1}(T)|$.
Throughout, we suppose that all considered knot vectors $\KK_{\bullet}$ with \eqref{eq:kappamax} are finer than some  fixed initial knot vector $\KK_0$.
We call such a knot vector admissible. 
The set of all these knot vectors is abbreviated by~$\mathbb{K}$.

\subsection{B-splines and NURBS}
\label{subsec:splines}
Throughout this subsection, we consider \textit{knots} {$\widehat{\mathcal{K}}_{\bullet}:=(t_{\bullet,i})_{i\in\Z}$} on $\R$ with multiplicity $\#_{\bullet} t_{\bullet,i}$, which satisfy that
$t_{\bullet,i-1}\leq t_{\bullet,i}$ for $i\in \Z$ and $\lim_{i\to \pm\infty}t_{\bullet,i}=\pm \infty$.
Let $\widehat{\mathcal{N}}_{\bullet}:=\set{t_{\bullet,i}}{i\in\Z}=\set{\widehat{{z}}_{\bullet,j}}{j\in \Z}$ denote the corresponding set of nodes with $\widehat{{z}}_{\bullet,j-1}<\widehat{{z}}_{\bullet,j}$ for $j\in\Z$.
For $i\in\Z$, the $i$-th \textit{B-spline} of degree $q$ is defined inductively by
\begin{align}
\begin{split}
\widehat B_{\bullet,i,0}&:=\chi_{[t_{\bullet,i-1},t_{\bullet,i})},\\
\widehat B_{\bullet,i,q}&:=\beta_{\bullet,i-1,q} \widehat B_{\bullet,i,q-1}+(1-\beta_{\bullet,i,q}) \widehat B_{\bullet,i+1,q-1} \quad \text{for } q\in \N,
\end{split}
\end{align}
where, for $t\in \R$,
\begin{align}
\begin{split}
\chi_{[t_{\bullet,i-1},t_{\bullet,i})}(t):=
\begin{cases}
1 \quad &\text{if }t\in {[t_{\bullet,i-1},t_{\bullet,i})},\\
0 \quad & \text{if } t\not\in {[t_{\bullet,i-1},t_{\bullet,i})},
\end{cases}
\quad
\beta_{\bullet,i,q}(t):=
\begin{cases}
\frac{t-t_{\bullet,i}}{t_{\bullet,i+q}-t_{\bullet,i}} \quad &\text{if } t_{\bullet,i}\neq t_{\bullet,i+q},\\
0 \quad &\text{if } t_{\bullet,i}= t_{\bullet,i+q}.
\end{cases}
\end{split}
\end{align}
The following lemma collects 
basic properties of B-splines. 
Proves are found, e.g., in \cite{Boor-SplineBasics}.

\begin{lemma}\label{lem:properties for B-splines}
For an interval $I=[a,b)$ and $q\in \N_0$, the following assertions  {\rm(i)--(vii)} hold:
\begin{enumerate}[\rm(i)]
\item \label{item:spline basis}
The set $\set{\widehat B_{\bullet,i,q}|_I}{i\in \Z, \widehat B_{\bullet,i,q}|_I\neq 0}$ is a basis for the space of all right-continuous $\widehat{\mathcal{N}_{\bullet}}$-piecewise polynomials of degree lower or equal $q$ on $I$, which are, at each knot $t_{\bullet,i}$, $q-\#_{\bullet} t_{\bullet,i}$ times continuously differentiable if $q-\#_{\bullet} t_{\bullet,i}\ge0$. 
\item \label{item:B-splines local} For $i\in\Z$,  $\widehat B_{\bullet,i,q}$ vanishes outside the interval $[t_{\bullet,i-1},t_{\bullet,i+q})$. 
It is positive on the open interval $(t_{\bullet,i-1},t_{\bullet,i+q})$  and a polynomial of degree $q$ on each interval $(\widehat z_{j-1},\widehat z_j)\subseteq(t_{\bullet,i-1},t_{\bullet,i+q})$ for $j\in\Z$.
\item \label{item:B-splines determined} For $i\in \Z$,  $\widehat B_{\bullet,i,q}$ is completely determined by the $q+2$ knots $t_{\bullet,i-1},\dots,t_{\bullet,i+q}$, wherefore we also write
\begin{align}
\widehat B(\cdot|t_{\bullet,i-1},\dots,t_{\bullet,i+q}):=\widehat B_{\bullet,i,q}
\end{align}

\item\label{item:B-splines partition} The  B-splines of degree $q$ form a (locally finite) partition of unity, i.e.,
\begin{equation}
\sum_{i \in\Z} \widehat B_{\bullet,i,q}=1\quad \text{on }\R.
\end{equation}
\item\label{item:interpolatoric} For $i\in \Z$ with $t_{\bullet,i-1}<t_{\bullet,i}=\dots=t_{\bullet,i+q}<t_{\bullet,i+q+1}$, it holds that
\begin{align}
\widehat B_{\bullet,i,q}(t_{\bullet,i}-)=1\quad\text{and} \quad \widehat B_{\bullet,i+1,q}(t_{\bullet,i})=1.
\end{align}
\item \label{item:derivative of splines}
Suppose the convention $q/0:=0$.
For $q\ge 1 $ and $i\in\Z$, it holds for the right derivative
\begin{equation}\label{eq:derivative of splines}
\widehat B_{\bullet,i,q}'^{_r}=\frac{q}{t_{\bullet,i+q-1}-t_{\bullet,i-1}} \widehat B_{\bullet,i,q-1}-\frac{q}{t_{\bullet,i+q}-t_{\bullet,i}}\widehat B_{\bullet,i+1,q-1}.
\end{equation}
\item\label{item:knot insertion}
Let $t'\in(t_{\ell-1},t_\ell]$ 
for some $\ell\in\Z$  and let $\widehat\KK_{\circ}$ be the refinement of $\widehat{\mathcal{K}}_{\bullet}$, obtained by adding $t'$.
{Then, for all coefficients $(a_{\bullet,i})_{i\in\Z}$, there exists $(a_{\circ,i})_{i\in\Z}$ such that}
\begin{equation}
\sum_{i\in \Z} a_{\bullet,i} \widehat B_{\bullet,i,q}=\sum_{i\in \Z} a_{\circ,i} \widehat B_{\circ,i,q}
\end{equation}
With the multiplicity $\#_{\circ}t'$ of $t'$ in the knots $\widehat\KK_{\circ}$, the new coefficients can be chosen as 
\begin{align}
\begin{split}
a_{\circ,i}=\begin{cases}
a_{\bullet,i} \quad &\text{if } i\leq \ell-q+\#_{\circ} t'-1,\\
(1-\beta_{\bullet,i-1,q}(t')) a_{\bullet,i-1} + \beta_{\bullet,i-1,q}(t') a_{\bullet,i} \quad &\text{if } \ell-q+\#_{\circ} t'\leq i\leq \ell,\\
a_{\bullet,i-1}\quad &\text{if } \ell+1 \leq i.
\end{cases}
\end{split}
\end{align}
If one assumes $\#_{\bullet} t_i\le q+1$ for all $i\in\Z$, these coefficients are unique.
{Note that these three cases are equivalent to $t_{\bullet,i+q-1}\le t'$, $t_{\bullet,i-1}<t'<t_{\bullet,i+q-1}$, resp. $t'\le t_{\bullet,i-1}$.}\hfill$\square$
\end{enumerate}
\end{lemma}

\begin{remark}\label{rem:B lincomb}
Let $j\in\Z$ and $(\delta_{ij})_{i\in\Z}$ be the corresponding Kronecker sequence. 
Choosing $(a_{\bullet,i})_{i\in\Z}=(\delta_{ij})_{i\in\Z}$ in Lemma~\ref{lem:properties for B-splines} \eqref{item:knot insertion}, one sees that $\widehat B_{\bullet,j,q}$ is a linear combination of $\widehat B_{\circ,j,q}$ and $\widehat B_{\circ,j+1,q}$, where $\widehat B_{\bullet,j,q}=\widehat B_{\circ,j,q}$ if  $j\leq \ell-q+\#_{\circ} t'-2$ and $\widehat B_{\bullet,j,q}=\widehat B_{\circ,j+1,q}$ if  $\ell+1 \leq j$.
\end{remark}

In addition to the knots $\widehat{\mathcal{K}}_{\bullet}=(t_{\bullet,i})_{i\in\Z}$, we consider fixed positive weights {$\mathcal{W}_{\bullet}:=(w_{\bullet,i})_{i\in\Z}$} with $w_{\bullet,i}>0$.
For $i\in \Z$ and $q\in \N_0$, we define the $i$-th 
NURBS by
\begin{equation}
\widehat R_{\bullet,i,q}:=\frac{w_{\bullet,i} \widehat B_{\bullet,i,q}}{\sum_{k\in\Z}  w_{\bullet,k}\widehat B_{\bullet,k,q}}.
\end{equation}
Note that the denominator is locally finite and positive.

For any $q\in\N_0$, we define the B-spline space
\begin{equation}
\widehat{\mathcal{S}}^q(\widehat{\mathcal{K}}_{\bullet}):=\left\{\sum_{i\in\Z}a_i \widehat B_{\bullet,i,q}:a_i\in\R\right\}
\end{equation}
as well as the NURBS space
\begin{equation}\label{eq:NURBS space defined} 
\widehat{\mathcal{S}}^q(\widehat{\mathcal{K}}_{\bullet},\mathcal{W}_{\bullet}):=\left\{\sum_{i\in\Z}a_i \widehat R_{\bullet,i,q}:a_i\in\R\right\}=\frac{\widehat{\mathcal{S}}^q(\widehat{\mathcal{K}}_{\bullet})}{\sum_{k\in \Z} w_{\bullet,k} \widehat B_{\bullet,k,q}}.
\end{equation}

We define for $0<\sigma<1$, any interval $I$, and ${\widehat v}\in L^2(I)$ the Sobolev-Slobodeckij seminorm  $|\widehat v|_{H^\sigma(I)}$ as in \eqref{eq:ss seminorm} (with $\Gamma_0$ and $v$  replaced by $I$ and $\widehat v$).
\begin{lemma}\label{lem:basis scaling}
Let $q>0$, $0<\sigma<1$, and $K, w_{\min}, w_{\max}>0$.
Suppose that the weights $\WWe_{\bullet}$ are bounded by $w_{\min}$ and $w_{\max}$, i.e.,  
\begin{align}
w_{\min}\le\inf_{i\in\Z} w_{\bullet,i} \le\sup_{i\in\Z} w_{\bullet,i}\le w_{\max},
\end{align}
 and that the local mesh-ratio on $\R$ is bounded by $K$, i.e.,  
\begin{align}
\sup\Big\{\max\Big(\frac{\widehat{z}_j-\widehat{z}_{j-1}}{\widehat{z}_{j-1}-\widehat{z}_{j-2}},\frac{\widehat{z}_j-\widehat{z}_{j-1}}{\widehat{z}_{j+1}-\widehat{z}_{j}}\Big):j\in\Z\Big\}\le K.
\end{align}
Then, there exists a constant $\Cscale>0$, which depends only on $q, w_{\min}, w_{\max}$, and $K$, such that
 for all $i\in\Z$ with $|\supp (\widehat R_{\bullet,i,q})|>0$, it holds that
\begin{align}\label{eq:basis scaling}
|\supp (\widehat R_{\bullet,i,q})|^{1-2\sigma}\le \Cscale |\widehat R_{\bullet,i,q}|_{H^\sigma(\supp( \widehat R_{\bullet,i,q}))}^2.
\end{align}
\end{lemma}
\begin{proof}
The proof is split into two steps.
\\
\noindent
{\bf Step 1:} First, we suppose  that $w_i=1$ for all $i\in\Z$ and hence $\widehat R_{\bullet,i,q}=\widehat B_{\bullet,i,q}$.
The definition of the B-splines implies their invariance with respect to affine transformations of the knots:  
\begin{align*}
\widehat B(t|t_0,\dots,t_{q+1})=\widehat B(ct+s|ct_0+s,\dots,ct_{q+1}+s)\quad\text{for all }t_0\le \dots\le t_{q+1}, s,t\in\R\text{ and } c>0.
\end{align*}
With the abbreviation $S:=\supp (\widehat B_{\bullet,i,q})=[t_{\bullet,i-1},t_{\bullet,i+q}]$, it hence holds that
\begin{align}\label{eqpr:basis scaling}
\begin{split}
&|\widehat B_{\bullet,i,q}|_{H^\sigma(S)}^2=\int_S\int_S\frac{|\widehat B_{\bullet,i,q}(r)-\widehat B_{\bullet,i,q}(s)|^2}{|r-s|^{1+2\sigma}}\,ds\,dr\\
&\quad=|S|^{1-2\sigma}\int_0^1\int_0^1\frac{|\widehat B(r|0,\frac{t_{\bullet,i}-t_{\bullet,i-1}}{|S|},\dots,\frac{t_{\bullet,i+q-1}-t_{\bullet,i-1}}{|S|},1)-\widehat B(s|\dots)|^2}{|r-s|^{1+2\sigma}}\,ds\,dr\\
&\quad\ge |S|^{1-2\sigma} \inf_{0\le t_1\le\dots\le t_q\le 1} \int_0^1\int_0^1 |\widehat B(r|0,t_1,\dots,t_q,1)-\widehat B(s|0,t_1,\dots,t_q,1)|^2\,ds\,dr,
\end{split}
\end{align}
where for the last inequality we have used that $|r-s|\le 1$.
We use a compactness argument to conclude the proof.
Let $\big((t_{k,1},\dots,t_{k,q})\big)_{k\in\N}$ be a convergent minimizing sequence for the infimum in \eqref{eqpr:basis scaling}.
Let $(t_{\infty,1},\dots,t_{\infty,q})$ be the corresponding limit.
With the definition of the B-splines  one easily verifies  that
\begin{align*}
\widehat B(r|0,t_{k,1},\dots,t_{k,q},1)\to \widehat B(r|0,t_{\infty,1},\dots,t_{\infty,q},1)\quad\text{for almost every }r\in\R.
\end{align*}
The dominated convergence theorem implies that the infimum is attained at $(t_{\infty,1},\dots,t_{\infty,q})$. 
Lemma \ref{lem:properties for B-splines} \eqref{item:B-splines local} especially implies that $\widehat B(\cdot|0,t_{\infty,1},\dots,t_{\infty,q},1)$ is not constant.
Therefore the infimum is positive, and we conclude the proof.
\\
\noindent
{\bf Step 2:} 
Recall that $\widehat R_{\bullet,i,q}=\frac{w_{\bullet,i} \widehat B_{\bullet,i,q}}{\sum_{j=i-q}^{i+q}w_{\bullet,j}\widehat B_{\bullet,j,q}}$. As in Step 1, we transform $\supp (\widehat R_{\bullet,i,q})$ onto the interval $[0,1]$.
Hence, it suffices to prove, with the compact interval $I:=[0-K^q,1+K^q]$, that the infimum
\begin{align*}
\inf_{t_{-q},\dots,t_{-1},t_1,\dots,t_q,t_{q+2},\dots,t_{1+2q}\in I\atop w_{1-q},\dots,w_{1+q}\inÊ[w_{\min},w_{\max}]}\int_0^1\int_0^1 \Big|&\frac{w_1\widehat B(r|0,t_1,\dots,t_q,1)}{\sum_{j=1-q}^{1+q}w_{j}\widehat B(r|t_{j-1},\dots,t_{j+q})}
\\
&\quad-\frac{w_1\widehat B(s|0,t_1,\dots,t_q,1)}{\sum_{j=1-q}^{1+q}w_{j}\widehat B(s|t_{j-1},\dots,t_{j+q})}\Big|^2\,ds\,dr
\end{align*}
is larger than $0$. 
This can be proved analogously as before.
\end{proof}

\subsection{Ansatz spaces}
\label{section:igabem} 
Throughout this section, we  abbreviate $\gamma|_{[a,b)}^{-1}$ with $\gamma^{-1}$ if $\Gamma\subsetneqq \partial\Omega$ is an open boundary.
Additionally to the initial knots $\KK_0\in\K$, suppose that $\mathcal{W}_0=(w_{0,i})_{i=1-p}^{N_0-p}$ are given initial weights with $w_{0,1-p}=w_{0,N_0-p}$, where $N_0=|\KK_0|$ for closed $\Gamma=\partial\Omega$ resp.\ $N_0=|\KK_0|-(p+1)$ for open $\Gamma\subsetneqq\partial\Omega$.
In the weakly-singular case we  assume $w_{0,i}=1$ for $i=1-p,\dots,N_0-p$. 
We extend the corresponding knot vector in the parameter domain, $\widehat{\KK}_0=(t_{0,i})_{i=1}^{N_0}$ if $\Gamma=\partial\Omega$ is closed resp. $\widehat{\KK}_0=(t_{0,i})_{i=-p}^{N_0}$ if $\Gamma\subsetneqq \partial\Omega$ is open, arbitrarily to $(t_{0,i})_{i\in \Z}$ with $t_{0,-p}=\dots=t_{0,0}=a$, $t_{0,i}\le t_{0,i+1}$, $\lim_{i\to \pm\infty}t_{0,i}=\pm \infty$.
For the extended sequence we also write  $\widehat{\mathcal{K}}_0$.
We define the weight function 
\begin{align}
\widehat w:=\sum_{k=1-p}^{N_0-p} w_{0,k}\widehat B_{0,k,p}|_{[a,b]}.
\end{align}

Let $\KK_{\bullet}\in\mathbb{K}$ be an admissible knot vector.
Outside of the interval $(a,b]$, we extend  the corresponding knot sequence $\widehat\KK_{\bullet}$ in the parameter domain exactly as before and write again $\widehat{\KK}_{\bullet}$ for the extension as well.
This guarantees that $\widehat\KK_0$ forms a subsequence of $\widehat\KK_{\bullet}$.
Via {knot insertion} from $\widehat\KK_0$ to $\widehat\KK_{\bullet}$,  Lemma~\ref{lem:properties for B-splines} \eqref{item:spline basis} proves the existence and uniqueness of   weights $\WWe_{\bullet}=(w_{\bullet,i})_{i=1-p}^{N_{\bullet}-p}$ such that 
\begin{align}\label{eq:w}
\widehat w=\sum_{k=1-p}^{N_0-p} w_{0,k}\widehat B_{0,k,p}|_{[a,b]}=\sum_{k=1-p}^{N_{\bullet}-p} w_{\bullet,k} \widehat B_{\bullet,k,p}|_{[a,b]}.
\end{align}
By choosing these weights, we ensure that the denominator of the considered rational splines  does not change.
Lemma \ref{lem:properties for B-splines} \eqref{item:interpolatoric} states  that  $\widehat B_{\bullet,1-p,p}(a)=1=\widehat B_{\bullet,N_{\bullet}-p,p}(b-)$, which implies  that  $w_{\bullet,1-p}=w_{\bullet,N_{\bullet}-p}$.
Further, Lemma \ref{lem:properties for B-splines} \eqref{item:B-splines partition} and \eqref{item:knot insertion} show  that
\begin{align}\label{eq:weights}
w_{\min}:=\min(\WWe_0)\le \min(\WWe_{\bullet})\le \widehat w\le \max(\WWe_{\bullet})\le \max(\WWe_0):=w_{\max}.
\end{align}
In the weakly-singular case there even holds  that $w_{\bullet,i}=1$ for $i=1-p,\dots,N_{\bullet}-p$, and $\widehat w=1$.
Finally, we extend $\mathcal{W}_{\bullet}$ arbitrarily to $(w_{\bullet,i})_{i\in\Z}$ with $w_{\bullet,i}>0$, identify the extension with $\mathcal{W}_{\bullet}$ and set for the hypersingular case
\begin{equation}
\mathcal{S}^p({\mathcal{K}}_{\bullet},\mathcal{W}_{\bullet}):=\set{\widehat V_{\bullet}\circ \gamma^{-1}}{\widehat V_{\bullet}\in\widehat{\mathcal{S}}^p(\widehat{\mathcal{K}}_{\bullet},\mathcal{W}_{\bullet})}
\end{equation}
and for the weakly-singular case
\begin{equation}
\mathcal{S}^{p-1}({\mathcal{K}}_{\bullet}):=\set{\widehat V_{\bullet}\circ \gamma^{-1}}{\widehat V_{\bullet}\in\widehat{\mathcal{S}}^{p-1}(\widehat{\mathcal{K}}_{\bullet})}.
\end{equation}
Lemma \ref{lem:properties for B-splines} \eqref{item:B-splines determined} shows that the definition does not depend on how the sequences are extended.
We define the transformed basis functions
\begin{align}\label{eq:basis def}
{R}_{\bullet,i,p}:=
\widehat R_{\bullet,i,p}\circ\gamma^{-1}\quad\text{and}\quad
 B_{\bullet,i,p-1}:=\widehat B_{\bullet,i,p-1}\circ\gamma^{-1}.
 \end{align}
Later, we will also need the notation $B_{\bullet,i,p}$, which we define analogously.

We introduce the  ansatz space for the hypersingular case
\begin{align}\label{eq:hypsing X0}
\XX_{\bullet}:=\begin{cases}\set{V_{\bullet}\in\mathcal{S}^p({\mathcal{K}}_{\bullet},\mathcal{W}_{\bullet})}{ V_{\bullet}(\gamma(a))=V_{\bullet}(\gamma(b-))}\subset H^{1/2}(\Gamma)&\text{ if }\Gamma=\partial\Omega,\\
\set{V_{\bullet}\in\mathcal{S}^p({\mathcal{K}}_{\bullet},\mathcal{W}_{\bullet})}{V_{\bullet}(\gamma(a))=0=V_{\bullet}(\gamma(b-))}\subset \H^{1/2}(\Gamma)&\text{ if }\Gamma\subsetneqq\partial\Omega,
\end{cases}
\end{align}
and for the weakly-singular case
\begin{equation}\label{eq:weaksing X0}
\YY_{\bullet}:=\mathcal{S}^{p-1}({\mathcal{K}}_{\bullet})\subset \H^{-1/2}(\Gamma).
\end{equation}
Note that, in contrast to the hypersingular case, we only allow for non-rational splines in the weakly-singular case.
We exploit this restriction in Lemma~\ref{lem:Y*=X'} below, which states that $\partial_\Gamma\XX_\bullet=\YY_\bullet$.
For rational splines, this assertion is in general false. 
We abbreviate 
\begin{align}
\overline{R}_{\bullet,i,p}:=\begin{cases}
{R}_{\bullet,1-p,p}+{R}_{\bullet,N_\bullet-p,p}\quad&\text{for }i=1-p,\\
{R}_{\bullet,i,p} \quad&\text{for }i\neq 1-p.
\end{cases}
\end{align}
We define $\overline{B}_{\bullet,i,p}$ analogously.
Further, we set 
\begin{align}
o:= 
\begin{cases}
0\quad&\text{if }
\Gamma= \partial\Omega,\\
1\quad&\text{if }
\Gamma\subsetneqq \partial\Omega.
\end{cases}
\end{align}
Lemma \ref{lem:properties for B-splines}  \eqref{item:spline basis} and  \eqref{item:interpolatoric}  show that 
\begin{align}\label{eq:hypsing basis}
\XX_{\bullet}={\rm span}{\set{\overline R_{\bullet,i,p}}{i=1-p+o,\dots,N_{\bullet}-p-1}},
\end{align}
as well as 
\begin{align}\label{eq:weaksing basis}
\YY_{\bullet}=\linhull\set{ B_{\bullet,i,p-1}}{i=1-(p-1),\dots,N_{\bullet}-1-(p-1)}.
\end{align}
In both cases, the corresponding sets  form a basis of $\XX_{\bullet}$ resp. $\YY_{\bullet}$.
 Note that the spaces $\XX_{\bullet}$ and $\YY_{\bullet}$  satisfy  that
$\XX_{\bullet}\subset \H^1(\Gamma)$ and $\YY_{\bullet}\subset L^2(\Gamma)$.
Lemma~\ref{lem:properties for B-splines} \eqref{item:spline basis} implies nestedness
\begin{align}
\XX_{\bullet}\subseteq \XX_{\circ}\text{ and }\YY_{\bullet}\subseteq\YY_{\circ} \quad\text{for all }\KK_{\bullet},\KK_{\circ}\in\mathbb{K}\text{ with } \KK_{\circ}\in\refine(\KK_{\bullet}).
\end{align}


\section{Auxiliary results for hypersingular case}
\label{section:aux}
\subsection{Scott-Zhang-type projection}
\label{section:scott} 
\noindent
In this section, we  use the Scott-Zhang-type operator from \cite[Section~4.3]{hypiga}.
Let $\KK_{\bullet}\in\K$.
In \cite[Section 2.1.5]{overview}, it is shown that, for $i\in \{1-p,\dots,N_{\bullet}-p\}$, there exist dual basis {functions $\widehat B_{\bullet,i,p}^*\in L^2(a,b)$ such that
\begin{align}
&\supp (\widehat B_{\bullet,i,p}^*)=\supp( \widehat B_{\bullet,i,p})=[t_{\bullet,i-1},t_{\bullet,i+p}],\\
&\int_a^b  \widehat B_{\bullet,i,p}^*(t) \widehat B_{\bullet,j,p}(t) dt=\delta_{ij}=\begin{cases} 1,\quad \text{if }i=j,\\ 0,\quad \text{else,}\end{cases}
\end{align}
and
\begin{align}\label{eq:dual inequality}
\norm{\widehat B_{\bullet,i,p}^*}{L^2(a,b)}\le 9^p(2p+3) {|\supp( \widehat B_{\bullet,i,p})|^{-1/2}}.
\end{align}
Each dual basis function depends only on the knots $t_{\bullet,i-1},\dots,t_{\bullet,i+p}$. 
Therefore, we also write
\begin{align}\label{eq:dual basis knots}
\widehat B_{\bullet,i,p}^*=\widehat B^*(\cdot | t_{\bullet,i-1},\dots,t_{\bullet,i+p}).
\end{align}
With the denominator $\widehat w$ from \eqref{eq:w},  define  
\begin{align}\label{eq:ripstar}
\widehat R_{\bullet,i,p}^*:=\widehat B_{\bullet,i,p}^*\widehat  w/w_{\bullet,i}.
\end{align}
This immediately proves  that
\begin{align}
\label{eq:duality}&\int_a^b \widehat  R_{\bullet,i,p}^*(t) \widehat R_{\bullet,j,p}(t) dt=\delta_{ij}
\end{align}
and
\begin{align}\label{eq:dual inequality2}
\norm{\widehat R_{\bullet,i,p}^*}{L^2(a,b)}\lesssim 9^p(2p+3) {|\supp (\widehat R_{\bullet,i,p})|^{-1/2}},
\end{align}
where the hidden constant depends only on $w_{\min}$ and  $w_{\max}$. 
We define the Scott-Zhang-type operator $J_{\bullet}:L^2(\Gamma)\to\XX_{\bullet}$
\begin{align}\label{eq:scotty}
J_{\bullet} v:=
\sum_{i=1-p+o}^{N_{\bullet}-p-1}\alpha_{\bullet,i}(v)\overline R_{\bullet,i,p} \text{ with }\alpha_{\bullet,i}(v):=\begin{cases}
 \int_a^b \frac{\widehat  R_{\bullet,1-p,p}^*+\widehat  R_{\bullet,N_\bullet-p,p}^*}{2} \,v\circ\gamma \,dt \,&\text{if }i=1-p,\\
  \int_a^b \widehat  R_{\bullet,i,p}^*\,v\circ\gamma \,dt \quad&\text{if }i\neq 1-p.
\end{cases}
\end{align}

A similar operator, namely $I_{\bullet} := \sum_{i=1-p}^{N_{\bullet}-p} \Big(\int_a^b  \widehat R_{\bullet,i,p}^*\, v \circ\gamma \,dt \Big){R}_{\bullet,i,p}$, has been analyzed in \cite[Section 3.1.2]{overview}.
However, $I_{\bullet}$ is not applicable here for two reasons:
First, for $\Gamma=\partial\Omega$, it does not guarantee that  $I_{\bullet} v$ is continuous at $\gamma(a)=\gamma(b)$.
Second, for $\Gamma \subsetneqq \partial\Omega$, it does not guarantee that $I_{\bullet} v(\gamma(a))=0=I_{\bullet}v(\gamma(b))$.

\begin{lemma}\label{lem:Scott difference}
Let  $\KK_{\bullet},\KK_{\circ}\in\K$ with $\KK_{\circ}\in\refine(\KK_{\bullet})$. 
Then, each $v\in L^2(\Gamma)$ satisfies  that
\begin{align}\label{eq:diffJ}
(J_{\circ} - J_{\bullet}) v \in {\rm{span}} \set{\overline R_{\circ,i,p}}{i\in\{1-p+o,\dots,N_{\circ}-p-1\}, \supp  (\overline R_{\circ,i,p})\cap \widetilde \NN_{\circ\setminus\bullet}\neq \emptyset},
\end{align}
where 
\begin{align}
\widetilde \NN_{\circ\setminus\bullet}:=\NN_{\circ}\setminus\NN_{\bullet} \cup \set{z\in\NN_{\circ}\cap\NN_{\bullet}}{\#_{\circ}z>\#_{\bullet} z}.
\end{align}
\end{lemma}

\begin{proof} 
We only prove the lemma for closed  $\Gamma=\partial\Omega$. 
For open $\Gamma\subsetneqq \partial\Omega$, the proof even simplifies. 
We split the proof into two steps.

\noindent\textbf{Step 1:}
We consider  the   case where $\widehat\KK_\circ$ is obtained from $\widehat\KK_\bullet$  by insertion of a single knot $t'\in [a,b]$ in the parameter domain.
Let  $b\neq t'\in(t_{\bullet,\ell-1},t_{\bullet,\ell}]$ with some $\ell\in \{1,\dots,N_{\bullet}-p\}$.
Note that  $N_{\circ}=N_{\bullet}+1$.
It holds that
\begin{align}\label{eqpr:Scott difference}
(J_{\circ}-J_{\bullet})v&=\alpha_{\circ,1-p}(v)  \overline R_{\circ,1-p,p}-\alpha_{\bullet,1-p}(v) \overline R_{\bullet,1-p,p}+\sum_{i=2-p}^{N_{\bullet}-p}\alpha_{\circ,i}(v)  R_{\circ,i,p}-\sum_{i=2-p}^{N_{\bullet}-p-1}\alpha_{\bullet,i}(v)  R_{\bullet,i,p}.
\end{align}
Remark~\ref{rem:B lincomb} and the choice $(a_{\bullet,j})_{j\in\Z}:=(w_{\bullet,j})_{j\in\Z}$ in Lemma \ref{lem:properties for B-splines} \eqref{item:knot insertion}  show the following: 
 first, for $1-p\le i\le\ell-p+\#_{\circ} t'-2$, it holds that $\widehat B_{\bullet,i,p}=\widehat B_{\circ,i,p}$ and $w_{\bullet,i}=w_{\circ,i}$, whence $ R_{\bullet,i,p}= R_{\circ,i,p}$; second,  for $\ell+1\le i\le N_{\bullet}-p$, it holds  that 
$\widehat B_{\bullet,i,p}=\widehat B_{\circ,i+1,p}$ and $w_{\bullet,i}=w_{\circ,i+1}$,  whence $ R_{\bullet,i,p}= R_{\circ,i+1,p}$.
Moreover, for $1-p\le i\le \ell-p+\#_{\circ} t'-2$, it holds  that $\widehat R_{\bullet,i,p}^*=\widehat R_{\circ,i,p}^*$
and for $\ell+1\le i\le  N_{\bullet}-p$  that $\widehat R_{\bullet,i,p}^*=\widehat R_{\circ,i+1,p}^*$
Hence, \eqref{eqpr:Scott difference} simplifies to
\begin{align}\label{eqpr:Scott difference 1}
&(J_{\circ}-J_{\bullet})v=\alpha_{\circ,1-p}(v)  \overline R_{\circ,1-p,p}-\alpha_{\bullet,1-p}(v) \overline R_{\bullet,1-p,p}\\\
&+\sum_{i=\max(2-p,\ell-p+\#_{\circ} t'-1)}^{\ell+1}\alpha_{\circ,i}(v) R_{\circ,i,p}-\sum_{i=\max(2-p,\ell-p+\#_{\circ} t'-1)}^{\ell}\alpha_{\bullet,i}(v) R_{\bullet,i,p}.\notag
\end{align}
Remark~\ref{rem:B lincomb} and  Lemma~\ref{lem:properties for B-splines} \eqref{item:B-splines local} imply  that
\begin{align}\label{eqpr:new hats}
\begin{split}
&\set{R_{\circ,i,p}}{i=\max(2-p,\ell-p+\#_{\circ}t'-1)\dots,\ell+1}\\
&\quad\cup \set{R_{\bullet,i,p}}{i=\max(2-p,\ell-p+\#_{\circ} t'-1),\dots,\ell}\\
&\subseteq{\rm span}\set{ \overline R_{\circ,i,p}}{i=\max(2-p,\ell-p+\#_{\circ} t'-1),\dots,\ell+1}\\
&\subseteq{\rm span}\set{ \overline R_{\circ,i,p}}{
\gamma(t')\in \supp( \overline R_{\circ,i,p})}.
\end{split}
\end{align}
We have already seen  that $ R_{\bullet,N_{\bullet}-p,p}= R_{\circ,N_{\bullet}-p+1,p}$ and $ R_{\bullet,N_{\bullet}-p,p}^*= R_{\circ,N_{\bullet}-p+1,p}^*$ if $N_{\bullet}-p\ge \ell+1$, and that $\widehat R_{\bullet,1-p,p}=\widehat R_{\circ,1-p,p}$ and $\widehat R_{\bullet,1-p,p}^*=\widehat R_{\circ,1-p,p}^*$ if $3\le \ell+\#_\circ t'$.
This shows that the first  summands in \eqref{eqpr:Scott difference 1} cancel each other if $N_{\bullet}-p\ge \ell+1$ and $3\le \ell+\#_\circ t'$.
Otherwise there holds  that $\ell=1$ or $\ell=N_\bullet-p$ and the functions $\overline R_{\circ,1-p,p}, R_{\circ,2-p,p}$ and $R_{\circ,N_\bullet-p,p}$ are in the last set of \eqref{eqpr:new hats}.
Since $\overline R_{\bullet,1-p,p}$ is a linear combination of these functions, we conclude~that
\begin{align*}
(J_{\circ} - J_{\bullet}) v \in {\rm{span}} \set{\overline R_{\circ,i,p}}{\gamma(t')\in\supp (\overline R_{\circ,i,p})}.
\end{align*}

\noindent \textbf{Step 2:} Let $\KK_{\circ}\in\K$  be an arbitrary refinement of $\KK_\bullet$ and let $\KK_{\circ}=\KK_{(M)},$  $\KK_{(M-1)},$ $\dots,$ $\KK_{(1)},$ $\KK_{(0)}=\KK_{\bullet}$ be a sequence of knot vectors such that each $ \KK_{(k)}$ is obtained by insertion of one single knot $\gamma(t_{(k)})$ in $\KK_{(k-1)}$.
Note that these meshes do not necessarily belong to $\K$, as the $\widehat \kappa$-mesh property~\eqref{eq:kappamax} can be violated.
However, the corresponding Scott-Zhang operator $J_{(k)}$ for $\KK_{(k)}$ can be defined just as above and Step 1 holds analogously.
There holds that
\begin{align}\label{eqpr:Scott difference 2}
(J_{\circ}-J_{\bullet})v=\sum_{k=1}^M (J_{(k)}-J_{(k-1)})v.
\end{align}
This and Step 1 imply  that
\begin{align}\label{eq:sum of spans}
(J_{\circ}-J_{\bullet})v\in \sum_{k=1}^M  {\rm{span}} \set{\overline R_{(k),i,p}}{(\gamma(t_{(k)})\in \supp (\overline R_{(k),i,p})}.
\end{align}
Remark~\ref{rem:B lincomb} shows that any basis function $\widehat B_{(k),i,p}$ with $k<M$ is the linear combination of $\widehat B_{(k+1),i,p}$ and $\widehat B_{(k+1),i+1,p}$. 
Moreover,  Lemma~\ref{lem:properties for B-splines} \eqref{item:B-splines local} shows that $\supp (\widehat B_{(k+1),i,p})\cup \supp (\widehat B_{(k+1),{i+1},p})\subseteq\supp (\widehat B_{(k),i,p})$.
We conclude  that
\begin{align*}
\overline R_{(k),i,p}\in{\rm span}\set{\overline R_{(k+1),j,p}}{\supp (\overline R_{(k+1),j,p})\subseteq \supp( \overline R_{(k),i,p})}.
\end{align*}
Together with \eqref{eq:sum of spans}, this shows  that
\begin{align*}
(J_{\circ}-J_{\bullet})v&\in \sum_{k=1}^M{\rm span}\set{\overline R_{(M),i,p}}{\gamma(t_{(k)})\in \supp(\overline R_{(M),i,p}})\\
&= {\rm{span}} \set{\overline R_{\circ,i,p}}{\supp( \overline R_{\circ,i,p})\cap\widetilde\NN_{\circ\setminus\bullet}\neq\emptyset},
\end{align*}
and concludes the proof.
\end{proof}

The following proposition  is taken from \cite[Proposition~4.3]{hypiga}.
\begin{proposition}\label{lem:Scott properties}
For $\KK_{\bullet}\in\K$, the corresponding Scott-Zhang operator $J_{\bullet}$ satisfies the following properties:
\begin{enumerate}[{\rm(i)}]
\item Local projection property: For  all $v\in L^2(\Gamma)$ and all $T\in\TT_{\bullet}$ it holds  that
\begin{align}
(J_{\bullet} v)|_T= v|_T \quad\text{if }v|_{\omega_{\bullet}^p(T)}\in \XX_{\bullet}|_{\omega_{\bullet}^p(T)}:=\set{V_{\bullet}|_{\omega_{\bullet}^p(T)}}{V_{\bullet}\in\XX_{\bullet}}.
\end{align}
\item Local $L^2$-stability: For all $v\in L^2(\Gamma)$ and all $T\in \TT_{\bullet}$, it holds  that
\begin{align}
\norm{J_{\bullet} v}{L^2(T)}\le \Cscott \norm{v}{L^2(\omega_{\bullet}^{p}(T))}.
\end{align}
\item Local $\H^1$-stability: For all $v\in \H^1(\Gamma)$ and all $T\in \TT_{\bullet}$, it holds  that
\begin{align}
|J_{\bullet} v|_{H^1(T)}\le \Cscott |v|_{H^1(\omega_{\bullet}^{p}(T))}.
\end{align}
\item Local approximation property: For all $v\in \H^1(\Gamma)$ and all $T\in\TT_{\bullet}$, it holds  that
\begin{align}\label{eq:local approx}
\norm{h_{\bullet}^{-1}(1-J_{\bullet}) v}{L^2(T)}\le \Cscott |v|_{H^1(\omega_{\bullet}^{p}(T))}.
\end{align}
\end{enumerate}
The constant $\Cscott>0$ depends only on ${\widehat \kappa}_{\max}, p, w_{\min}, w_{\max}$, and $\gamma$.\hfill$\square$
\end{proposition}

\subsection{Inverse inequalities}
In this section, we state some inverse estimates for NURBS from \cite[Section~4.3 and Appendix~C.1]{hypiga},  which are well-known for piecewise polynomials \cite{inverse2,hypsing3d}.
\begin{proposition}\label{lem:inverseIneq}
Let $\KK_{\bullet}\in\K$ and $0\le\sigma\le 1$.
Then, there hold the inverse inequalities
\begin{align}\label{eq:inverseIneq}
\norm{V_{\bullet}}{\H^\sigma(\Gamma)}\le \Cinv \norm{h_{\bullet}^{-\sigma}V_{\bullet}}{L^2(\Gamma)} \quad\text{for all } V_{\bullet}\in \XX_{\bullet},
\end{align}
and 
\begin{align}\label{eq:inverseIneq2}
\norm{h_{\bullet}^{1-\sigma}\partial_\Gamma{V_{\bullet}}}{L^2(\Gamma)}\le \Cinv \norm{V_{\bullet}}{\H^\sigma(\Gamma)} \quad\text{for all } V_{\bullet}\in \XX_{\bullet}.
\end{align}
The constant $\Cinv>0$ depends only on ${\widehat \kappa}_{\max}, p, w_{\min}, w_{\max}$, $\gamma$, and $\sigma$.\hfill$\square$
\end{proposition}

\subsection{Uniform meshes}
\label{section:multilevel} 
We consider a sequence  $\KK_{{\rm uni}(m)}\in\K$ of uniform knot vectors:
Let $\KK_{{\rm uni}(0)} :=\KK_0$ and let $\KK_{{\rm uni}(m+1)}$ be obtained from
$\KK_{{\rm uni}(m)}$ by uniform refinement, i.e., all elements of $\TT_{{\rm uni}(m)}$ are
bisected {in the parameter domain} into son elements with half length, where each new knot has multiplicity one. 
Define $\overline h_{{\rm uni}(0)} := \max_{T\in\TT_0} |\gamma^{-1}(T)|$ as well as 
\begin{align}\label{def:hat-hell}
  \overline h_{{\rm uni}(m)} := 2^{-m} \overline h_{{\rm uni}(0)} \quad\text{for each }m \geq 1.
\end{align}%
Note that $\overline h_{{\rm uni}(m)}$ is equivalent to the usual local mesh-size function
on $\TT_{{\rm uni}(m)}$, i.e., $\overline h_{{\rm uni}(m)} \simeq |T|$ for all
$T\in\TT_{{\rm uni}(m)}$ and all $m\ge0$, {where the hidden constants depend only on $\TT_0$ and $\gamma$}.
Moreover, let $\XX_{{\rm uni}(m)}$ denote the associated discrete space with corresponding $L^2$-orthogonal projection $\Pi_{{\rm uni}(m)} : L^2(\Gamma)\to \XX_{{\rm uni}(m)}$.  Note that the discrete spaces $\XX_{{\rm uni}(m)}$ are nested, i.e., $\XX_{{\rm uni}(m)}\subseteq\XX_{{\rm uni}(m+1)}$ for all $m\geq 0$. 

The next result follows by the approximation property of   Proposition~\ref{lem:Scott properties} {\rm(iv)} and the inverse inequality \eqref{eq:inverseIneq} of  Proposition~\ref{lem:inverseIneq} in combination with \cite{bornemann}.

\begin{lemma}\label{lem:normEquiv}
  {Let $0< \sigma< 1$.} Then, 
  \begin{align}\label{eq:normEquiv}
 \sum_{m=0}^\infty \overline h_{{\rm uni}(m)}^{-2\sigma} \norm{(1 -\Pi_{{\rm uni}(m)})v}{L^2(\Gamma)}^2\le \Cnorm \norm{v}{\H^\sigma(\Gamma)}^2\quad \text{for all } v\in \H^\sigma(\Gamma),
\end{align}
where the constant $\Cnorm>0$ depends only on $\TT_0$,  $\widehat \kappa_{\max}, p, w_{\min}, w_{\max}$, $\gamma$, and $\sigma$.
\end{lemma}
\begin{proof}
The $L^2$-best approximation property of $\Pi_{{\rm uni}(m)}$ yields  that $\norm{v-{\Pi}_{{\rm uni}(m)} v}{L^2(\Gamma)}\le \norm{v- J_{{\rm uni}(m)} v}{L^2(\Gamma)}$.
With Proposition~\ref{lem:Scott properties}, we have  that
  \begin{align}\label{eq:approxPropL2}
    \norm{v-\Pi_{{\rm uni}(m)} v}{L^2(\Gamma)} \stackrel{\eqref{eq:local approx}}\lesssim \overline h_{{\rm uni}(m)} \norm{v}{\H^1(\Gamma)} = \overline h_{{\rm uni}(0)} 2^{-m} \norm{v}{\H^1(\Gamma)} \quad\text{for all } v\in \H^1(\Gamma).
  \end{align}
The approximation property \eqref{eq:approxPropL2} (also called Jackson inequality)  together with the inverse inequality \eqref{eq:inverseIneq} (also called Bernstein inequality) from Proposition~\ref{lem:inverseIneq} allow to apply~\cite[Theorem~1 and Corollary~1]{bornemann} with $X=\H^1(\Gamma)$ and $\alpha=1$.
The latter proves~that
  \begin{align*}
  \norm{v}{\H^\sigma(\Gamma)}^2\simeq  \overline h_{{\rm uni}(0)}^{-2\sigma}\norm{v}{L^2(\Gamma)}^2 + \sum_{m=0}^\infty \overline h_{{\rm uni}(m)}^{-2\sigma} \norm{(1-\Pi_{{\rm uni}(m)})v}{L^2(\Gamma)}^2\quad \text{for all }v\in \H^\sigma(\Gamma).
  \end{align*}
{This concludes the proof.}
\end{proof}

\subsection{Level function}
Let $\KK_{\bullet}\in\K$.
For given $T\in\TT_{\bullet}$, let $T_0\in\TT_0$ denote its unique ancestor such that $T\subseteq T_0$ and define 
with the corresponding elements $\widehat T=\gamma^{-1}(T),\widehat T_0=\gamma^{-1}(T_0)$ in the parameter domain the
generation of $T$ by
\begin{align*}
  \gen(T) := \frac{\log(|\widehat T|/|\widehat T_0|)}{\log(1/2)} \in \N_0,
\end{align*}
i.e., $\gen(T)$ denotes the number of bisections of $T_0\in\TT_0$ {in the parameter domain} needed to obtain the element $T\in\TT_{\bullet}$.
To each node $z\in\NN_{\bullet}$, we associate 
\begin{align}\label{eq:def:level}
  \level_{\bullet}(z) := \max\set{\gen(T)}{T\in \TT_{\bullet} \text{ and } z\in T} \quad\text{for all } z\in\NN_{\bullet}.
\end{align}
The function $\level_{\bullet}(\cdot)$  provides a link between the mesh $\TT_{\bullet}$ and the sequence of
uniformly refined meshes $\TT_{{\rm uni}(m)}$.  A simple proof  of the following result is found in~\cite[Lemma~6.11]{dissfuehrer}.
\begin{lemma}\label{lem:levelProperties}
Let $\KK_{\bullet}\in\K$ and    $z\in\NN_{\bullet}$ and $m:=\level_{\bullet}(z)$. Then, it holds  that $z\in \NN_{{\rm uni}(m)}$~and
  \begin{align}
C_{\rm level}^{-1} \overline h_{{\rm uni}(m)} \leq |T| \leq C_{\rm level}\overline h_{{\rm uni}(m)} \quad\text{for all }
    T\in\TT_{\bullet} \text{ with } z\in T.
  \end{align}
  The constant $\setc{c:unifEquivLevel}>0$ depends only on {$\TT_0,{\widehat \kappa}_{\max}$, and $\gamma$}.
  \qed
\end{lemma}


\def\ZZ{\mathcal{Z}}

\section{Local multilevel diagonal preconditioner for the hypersingular case}
\label{section:precond}
Throughout this section, let $(\KK_\ell)_{\ell\in\N_0}$ be a sequence of refined knot vectors, i.e., $\KK_\ell,\KK_{\ell+1}\in\K$ with $\KK_{\ell+1}\in\refine(\KK_\ell)$,  
and let $L\in\N_0$. We set 
\begin{align}
\widetilde\NN_{0\setminus -1}:=\NN_0\quad\text{and}\quad\omega_{-1}(\cdot):=\omega_0(\cdot).
\end{align}
For $\ell\in\N_0$, abbreviate the corresponding index set from \eqref{eq:diffJ}
\begin{align}
  \II_\ell:= \set{i\in\{1-p+o,\dots,N_\ell-p-1\}}{\supp(  \overline R_{\ell,i,p})\cap \widetilde \NN_{\ell\setminus\ell-1}\neq \emptyset},
\end{align}
and define the spaces
\begin{align}\label{eq:reddecomp}
  \widetilde\XX_\ell := \linhull\set{\overline R_{\ell,i,p}}{i\in\II_\ell}=\sum_{i\in\II_\ell} \XX_{\ell,i}
 \quad \text{with}\quad
\XX_{\ell,i} :=\linhull\{\overline R_{\ell,i,p}\}.
\end{align}
Note that  $\II_0=\{1-p+o,\dots,N_0-p-1\}$ and $\widetilde{\XX}_0=\XX_0$.
For all $\ell\in\N_0$ and $i\in\II_\ell$, fix a node 
\begin{align}\label{eq:zli}
\widetilde z_{\ell,i} \in\widetilde\NN_{\ell\setminus\ell-1}\quad \text{with}\quad\widetilde z_{\ell,i} \in\supp( \overline R_{\ell,i,p}).
\end{align}
For $V_L\in\XX_L$ and $\ell=0,1,\dots,L$, we define (see Lemma \ref{lem:Scott difference})
\begin{align}
  \widetilde{V}_L^{\ell} := (J_\ell-J_{\ell-1})V_L \in \widetilde\XX_\ell, \quad\text{where } J_{-1}:=0.
\end{align}
For all $i\in \II_\ell$, we set with  the abbreviation $\alpha_{\ell,i}(\widetilde V_L^\ell)$ from \eqref{eq:scotty}
\begin{align}\label{eq:alphas}\begin{split}
V_L^{\ell,i}:=\alpha_{\ell,i}(\widetilde V_L^\ell)\,\overline R_{\ell,i,p}.
\end{split}
\end{align}}
By the duality property \eqref{eq:duality} and the decomposition \eqref{eq:reddecomp}, we have the decompositions
\begin{align}\label{eq:decomposition of v}
  \widetilde{V}_L^{\ell} = \sum_{i\in\II_\ell} V_L^{\ell,i} \quad\text{and}\quad
  V_L = \sum_{\ell=0}^L \widetilde{V}_L^{\ell} = \sum_{\ell=0}^L \sum_{i\in\II_\ell} V_L^{\ell,i}
\end{align}
and hence
\begin{align}\label{eq:space decomp}
  \XX_L = \sum_{\ell=0}^L \sum_{i\in \II_\ell} \XX_{\ell,i}.
\end{align}
With the one-dimensional $\edual{\cdot}{\cdot}$-orthogonal projections $\prec_{\ell,i}$ onto $\XX_{\ell,i}$ defined by
\begin{align}\label{eq:Pelli}
  \edual{\prec_{\ell,i} u}{V_{\ell,i}} = \edual{u}{V_{\ell,i}} \quad\text{for all } u\in \H^{1/2}(\Gamma),
  V_{\ell,i}\in\XX_{\ell,i},
\end{align}
the space decomposition \eqref{eq:space decomp} gives rise to the additive Schwarz operator
\begin{align}\label{eq:PAS}
  \PAS = \sum_{\ell=0}^L \sum_{i\in\II_\ell} \prec_{\ell,i}.
\end{align}

A similar operator for continuous  piecewise affine ansatz functions on affine geometries has been investigated in \cite{ffps}.
Indeed, the proof of the following main result (for the hypersingular case) is essentially inspired by the corresponding proof of \cite[Theorem~1]{ffps}.

\begin{theorem}\label{thm:PAS}
  The additive Schwarz operator $\PAS : \H^{1/2}(\Gamma)\to \XX_L$ satisifies  that
  \begin{align}\label{eq:prop:PAS}
    \evmin \enorm{V_L}^2 \leq \edual{\PAS V_L}{V_L} \leq \evmax \enorm{V_L}^2 \quad\text{for all } V_L\in \XX_L,
  \end{align}
where the constants $\evmin,\evmax>0$ depend only on $\TT_0$, ${\widehat \kappa}_{\max}, p, w_{\min}, w_{\max}$, and $\gamma$. 
\end{theorem}
We split the proof into two parts. In Section~\ref{sec:PAS:lb}, we show the lower bound. 
The upper bound is proved in Section~\ref{sec:PAS:ub}.

\subsection{Proof of Theorem~\ref{thm:PAS} (lower bound)}\label{sec:PAS:lb}
In the remainder of this section, we will show that the decomposition \eqref{eq:decomposition of v} of $V_L$ is stable, i.e.,
\begin{align}\label{eq:lower decomp}
  \sum_{\ell=0}^L \sum_{i\in\II_\ell} \enorm{V_L^{\ell,i}}^2 \lesssim \enorm{V_L}^2.
\end{align}
It is well known from additive Schwarz theory~\cite{lions88,wid89,zhang92,ToselliWidlund} that this proves the lower bound in Theorem~\ref{thm:PAS}; see, e.g.,  \cite[Lemma~3.1]{zhang92}.
We start with two auxiliary lemmas. To ease readability, the proofs of Lemma~\ref{lem:localProj} and Lemma~\ref{lem:li set} are postponed to the end of this section after the proof of the lower bound in~\eqref{eq:prop:PAS}.
In the following, we set $\XX_{{\rm uni}(m)} := \XX_{{\rm uni}(0)}$ if $m<0$.

\begin{lemma}\label{lem:localProj}
{Let $\ell\in \N$ and $q\in\N$. There exists a constant $C_1(q)\in\N_0$ such that for all {$z\in\NN_\ell$} with $m=\level_\ell(z)$, it holds  that
  \begin{align}
\set{V|_{\omega_{\ell-1}^q(z)}}{V\in\XX_{{\rm uni}(m-{C_1}(q))}}\subseteq \set{V|_{\omega_{\ell-1}^q(z)}}{V\in\XX_{\ell-1}}
  \end{align}
  The constant ${C_1}(q)$ depends only on ${\widehat \kappa}_{\max}, \gamma$ and $q$.
  We abbreviate ${C_1}:={C_1}(2p+1)$.}
\end{lemma}

\begin{lemma}\label{lem:li set}
For each $m\in\N_0$ and $z\in\NN_L$, it holds   that $|\ZZ_m(z)|\le {C_2}$,  where 
\begin{align}\label{eq:li set}
\ZZ_m(z):=\set{(\ell,i)}{\ell\in\{0,\dots,L\},i\in\II_\ell,\level_\ell({\widetilde z_{\ell,i})=m,z=\widetilde z_{\ell,i}}}.
\end{align}
The   constant $C_2>0$  depends only on $p$.
\end{lemma}

\begin{proof}[Proof of lower bound in~\eqref{eq:prop:PAS}]
The proof is split into two steps.
\\
\noindent
{\bf Step 1:}
We show \eqref{eq:vliw2}.
The norm equivalence $\enorm{\cdot}\simeq\norm{\cdot}{H^{1/2}(\Gamma)}$, the inverse inequality \eqref{eq:inverseIneq} for NURBS and $\norm{h_\ell^{-1/2}\overline R_{\ell,i,p}}{L^2(\Gamma)}\lesssim 1$  prove for the functions $V_L^{\ell,i}$ of \eqref{eq:alphas}  that
\begin{align*}
\enorm{V_L^{\ell,i}}^2 \lesssim\big|\alpha_{\ell,i}\big((J_\ell-J_{\ell-1})V_L\big)\big|^2.
\end{align*}
The Cauchy-Schwarz inequality, and the property~\eqref{eq:dual inequality2} of the dual
basis functions imply~that
\begin{align}\label{eq:vliw}
  \enorm{V_L^{\ell,i}}^2 
  \lesssim |\supp( \overline R_{\ell,i,p})|^{-1} \norm{(J_\ell-J_{\ell-1})V_L}{L^2(\supp( \overline R_{\ell,i,p}))}^2.
\end{align}
We abbreviate $m=\level_\ell(\widetilde z_{\ell,i})$.
Proposition~\ref{lem:Scott properties} {\rm (i)} and Lemma~\ref{lem:localProj} together with nestedness $\XX_{\ell-1}\subseteq\XX_{\ell}$ imply for almost every $x\in\omega_{\ell-1}^{p+1}(\widetilde{z}_{\ell,i})$  that
\begin{align*}
(J_\ell \Pi_{{\rm uni}(m-{C_1})} V_L) (x)= (\Pi_{\rm uni(m-{C_1}}) V_L)(x)=
(J_{\ell-1} \Pi_{{\rm uni}(m-{C_1})} V_L)(x).
\end{align*}
This together with \eqref {eq:zli} and  local $L^2$-stability of $J_\ell$ and $J_{\ell-1}$ (Proposition \ref{lem:Scott properties} (ii))
  shows~that
\begin{align}\label{eq:vtil}
\begin{split}  \norm{(J_\ell-J_{\ell-1})V_L}{L^2(\supp(\overline R_{\ell,i,p}))}^2 
  &= \norm{(J_\ell-J_{\ell-1})(1-\Pi_{{\rm uni}(m-{C_1})})V_L}{L^2(\supp(\overline R_{\ell,i,p}))}^2
  \\
  &\leq \norm{(J_\ell-J_{\ell-1})(1-\Pi_{{\rm uni}(m-{C_1})})V_L}{L^2(\omega_{\ell-1}^{p+1}(\widetilde z_{\ell,i}))}^2
  \\
  &\lesssim \norm{(1-\Pi_{{\rm uni}(m-{C_1})})V_L}{L^2(\omega_{\ell-1}^{2p+1}(\widetilde z_{\ell,i}))}^2.
  \end{split}
\end{align}
Further, Lemma~\ref{lem:levelProperties} shows  that $\overline h_{\rm uni(m)}\simeq|\supp(\overline {R}_{\ell,i,p})|$.
Hence, \eqref{eq:vliw} and \eqref{eq:vtil} prove  that 
\begin{align}\label{eq:vliw2}
  \enorm{V_L^{\ell,i}}^2 \lesssim \overline h_{{\rm uni}(m)}^{-1} 
  \norm{(1-\Pi_{{\rm uni}(m-{C_1})})V_L}{L^2(\omega_{\ell-1}^{2p+1}(\widetilde z_{\ell,i}))}^2.
\end{align}
\\
\noindent
{\bf Step 2:} 
We show \eqref{eq:lower decomp}, which concludes  the lower bound in \eqref{eq:prop:PAS}.
Step~1 gives that
\begin{align*}
  \sum_{\ell=0}^L \sum_{i\in\II_\ell} \enorm{V_L^{\ell,i}}^2 &= \sum_{m=0}^\infty \sum_{\ell=0}^L \sum_{\substack{i\in\II_\ell\\
  \level_\ell(\widetilde z_{\ell,i})=m}} \enorm{V_L^{\ell,i}}^2 \\
&  \lesssim \sum_{m=0}^\infty \sum_{\ell=0}^L \sum_{\substack{i\in\II_\ell\\ \level_\ell(\widetilde z_{\ell,i})=m}}
  \overline h_{{\rm uni}(m)}^{-1} 
  \norm{(1-\Pi_{{\rm uni}(m-{C_1})})V_L}{L^2(\omega_{\ell-1}^{2p+1}(\widetilde z_{\ell,i}))}^2.
\end{align*}
There exists a constant $C_3\in\N$,  which depends only on $p,\widehat \kappa_{\max},\gamma$,  and $\TT_0$, such that for 
$z\in\NN_\ell$ with $\level_{\ell}(z)=m$, it holds  that
\begin{align}\label{unifPatch}
\omega_{\ell-1}^{2p+1}(z) \subseteq
\omega_{{\rm uni}(m)}^{{C_3}}(z).
\end{align}
Hence,
\begin{align*}
  \sum_{\ell=0}^L \sum_{i\in\II_\ell} \enorm{V_L^{i,\ell}}^2 & \lesssim 
  \sum_{m=0}^\infty \sum_{\ell=0}^L \sum_{\substack{i\in\II_\ell\\ \level_\ell(\widetilde z_{\ell,i})=m}}
  \overline h_{{\rm uni}(m)}^{-1} 
  \norm{(1-\Pi_{{\rm uni}(m-{C_1})})V_L}{L^2(\omega_{{\rm uni}(m)}^{{C_3}}(\widetilde z_{\ell,i}))}^2 \\
  &\stackrel{\eqref{eq:li set}}{=} \sum_{m=0}^\infty \sum_{z\in\NN_L} \sum_{(\ell,i)\in \ZZ_m(z)}
  \overline h_{{\rm uni}(m)}^{-1} \norm{(1-\Pi_{{\rm uni}(m-{C_1})})V_L}{L^2(\omega_{{\rm uni}(m)}^{{C_3}}(z))}^2.
\end{align*}
If $z\in\NN_L$ and $(\ell,i)\in\ZZ_m(z)$, it follows  that $z\in\NN_\ell$ with $\level_\ell(z) = m$ by definition.
Lemma~\ref{lem:levelProperties} implies  that  $z\in\NN_{{\rm uni}(m)}$.
This and Lemma~\ref{lem:li set} give  that
\begin{align*}
  &\sum_{m=0}^\infty \sum_{z\in\NN_L} \sum_{(\ell,i)\in \ZZ_m(z)}
  \overline h_{{\rm uni}(m)}^{-1} \norm{(1-\Pi_{{\rm uni}(m-{C_1})})V_L}{L^2(\omega_{{\rm uni}(m)}^{{C_3}}(z))}^2 \\
  &\qquad\qquad= \sum_{m=0}^\infty \sum_{z\in\NN_L\cap \NN_{{\rm uni}(m)}}
  \sum_{(\ell,i)\in \ZZ_m(z)} 
  \overline h_{{\rm uni}(m)}^{-1} \norm{(1-\Pi_{{\rm uni}(m-{C_1})})V_L}{L^2(\omega_{{\rm uni}(m)}^{{C_3}}(z))}^2
  \\ 
  &\qquad\qquad \stackrel{\eqref{eq:li set}}{\lesssim} \sum_{m=0}^\infty \sum_{z\in\NN_L\cap \NN_{{\rm uni}(m)}}
  \overline h_{{\rm uni}(m)}^{-1} \norm{(1-\Pi_{{\rm uni}(m-{C_1})})V_L}{L^2(\omega_{{\rm uni}(m)}^{{C_3}}(z))}^2
  \\ 
  &\qquad\qquad\leq \sum_{m=0}^\infty \sum_{z\in\NN_{{\rm uni}(m)}}
  \overline h_{{\rm uni}(m)}^{-1}\norm{(1-\Pi_{{\rm uni}(m-{C_1})})V_L}{L^2(\omega_{{\rm uni}(m)}^{{C_3}}(z))}^2\\
  &\qquad\qquad\lesssim \sum_{m=0}^\infty\overline h_{{\rm uni}(m)}^{-1}\norm{(1-\Pi_{{\rm uni}(m-{C_1})})V_L}{L^2(\Gamma)}^2.
\end{align*}
The definition $\Pi_{{\rm uni}(m)}=\Pi_{{\rm uni}(0)}$ 
for $m<0$ yields  that
\begin{align*}
\sum_{m=0}^\infty\overline h_{{\rm uni}(m)}^{-1}\norm{(1-\Pi_{{\rm uni}(m-{C_1})})V_L}{L^2(\Gamma)}^2\lesssim
\sum_{m=0}^\infty\overline h_{{\rm uni}(m)}^{-1}\norm{(1-\Pi_{{\rm uni}(m)})V_L}{L^2(\Gamma)}^2.
\end{align*}
Combining the latter three estimates, 
Lemma~\ref{lem:normEquiv} leads us to
\begin{align*}
 \sum_{\ell=0}^L \sum_{i\in\II_\ell} \enorm{V_L^{\ell,i}}^2
 \lesssim \sum_{m=0}^\infty\overline h_{{\rm uni}(m)}^{-1}\norm{(1-\Pi_{{\rm uni}(m)})V_L}{L^2(\Gamma)}^2
\stackrel{\eqref{eq:normEquiv}} \lesssim \norm{V_L}{H^{1/2}(\Gamma)}^2
 \simeq\enorm{V_L}^2.
\end{align*}
This proves \eqref{eq:lower decomp} and yields the lower bound in \eqref{eq:prop:PAS}.
\end{proof}

{%
\begin{proof}[Proof of Lemma~\ref{lem:localProj}]
  We show that 
  \begin{align}\label{eq:nested node patches}
\NN_{{\rm uni}(m-{C_1}(q))} \cap \omega_{\ell-1}^q(z)\subseteq \NN_{\ell-1}\cap\omega_{\ell-1}^q(z).
  \end{align}
  Let $\tau\in\TT_\ell$ such that $z\in \tau$. 
  Let $T\in\TT_{\ell-1}$ be the father element of $\tau$, i.e., $\tau\subseteq T$. 
  We note that $\gen(T)=\gen(\tau)$ or $\gen(T)=\gen(\tau)+1$ and hence
  \begin{align*}
    |\gen(\tau)-\gen(T)| \le 1.
  \end{align*}
  Moreover, there exists a constant $C\in\N$, which depends only on $\widehat \kappa_\mathrm{max},\gamma$ and $q$ such that 
  \begin{align*}
    |\gen(T)-\gen(T')| \leq C \quad \text{for all } T'\in\TT_{\ell-1} \text{ with } T'\subseteq
    \omega_{\ell-1}^q(z),
  \end{align*}
  i.e., the difference in the element generations within some  $q$-th order patch is uniformly bounded. This implies  that
  \begin{align*}
    \gen(\tau) \leq \gen(T') + C+1 \quad\text{for all } T'\in\TT_{\ell-1} \text{ with } T'\subseteq
    \omega_{\ell-1}^q(z).
  \end{align*}
By definition of $\level_\ell(z)$, we thus infer that ${C_1}(q):=C+1>0$ yields  that
  \begin{align}\label{eq:levelineq}
    m=\level_\ell(z) \leq \min\set{\gen(T')}{T'\in\TT_{\ell-1} \text{ and } T'\subseteq \omega_{\ell-1}^{q}(z)}
    + {C_1}(q).
  \end{align}
For $m-{C_1}(q)\leq 0$, we have  that $\XX_{{\rm uni}(m-{C_1}(q))} = \XX_0$, and the assertion is
  clear. 
  Therefore, we suppose  that $m-{C_1}(q)\geq1$. 
Let $T'\in\TT_{\ell-1}$ with $T'\subseteq \omega_{\ell-1}^q(z)$.  
According to \eqref{eq:levelineq}, it holds  that  $m-{C_1}(q)\le \gen(T')$.
Therefore, there exists  a father element $Q\in\TT_{{\rm uni}(m-{C_1}(q))}$ with $T'\subseteq Q$.
Suppose that \eqref{eq:nested node patches} does not hold true.
Then there is some $z'\in\NN_{{\rm uni}(m-{C_1}(q))} \cap \omega_{\ell-1}^q(z)$, which is not contained in $\NN_{\ell-1}\cap\omega_{\ell-1}^q(z)$.
Therefore, $z'$ is in the interior of some $T'\in\TT_{\ell-1}$ with $T'\subseteq\omega_{\ell-1}^q(z)$ and hence also in the interior of the father $Q\in\TT_{{\rm uni}(m-{C_1}(q))}$ of $T'$.
This contradicts $z\in\NN_{{\rm uni}(m-{C_1}(q))}$ 
and concludes the proof of \eqref{eq:nested node patches}. 

  By the definition of $\XX_{{\rm uni}(m-{C_1}(q))}$, we even have for the multiplicities   that $\#_{{\rm uni}(m-{C_1}(q))} z'\le \#_{\ell-1} z'$ for $z'\in \NN_{{\rm uni}(m-{C_1}(q))}\cap \omega_{\ell-1}^q(z)$.
  With Lemma \ref{lem:properties for B-splines} \eqref{item:spline basis} and the fact that the NURBS denominator $w$ of \eqref{eq:w} is fixed, this proves the assertion.
  \end{proof}
}

\begin{proof}[Proof of Lemma~\ref{lem:li set}]
As we only use bisection or knot multiplicity increase (with maximal multiplicity $p$), it holds   that
$|\set{\ell\in\{1,\dots,L\}}{z\in\widetilde\NN_{\ell\setminus\ell-1}}|\le p$.
This shows that only a bounded number of different $\ell$ appears in the set of \eqref{eq:li set}.
For fixed $\ell\in\{0,\dots,L\}$, \eqref{eq:zli} and $\omega_{\ell-1}^{p+1}(z)\subseteq \omega_\ell^{2(p+1)}(z)$ yield  that
\begin{align*}
\set{i\in\II_\ell}{\level_\ell(\widetilde z_{\ell,i})=m,z=\widetilde z_{\ell,i}}&\stackrel{\eqref{eq:zli}}{\subseteq}\set{i}{\supp(\overline{R}_{\ell,i,p})\subseteq\omega_{\ell-1}^{p+1}(z)}\\
&\,\,\subseteq\,\,\set{i}{\supp(\overline{R}_{\ell,i,p})\subseteq\omega_\ell^{2(p+1)}(z)}.
\end{align*}
The cardinality of the last set is bounded by a constant ${C_2}>0$ that depends only on $p$.
\end{proof}

\subsection{Proof of Theorem~\ref{thm:PAS} (upper bound)}\label{sec:PAS:ub}
For $m\in\N_0$, let $\KK_{{\rm uni}(m,p)}\in\K$  be the knot vector with $\TT_{{\rm uni}(m,p)}=\TT_{{\rm uni}(m)}$ and $\#z=p$ for all $z\in \NN_{{\rm uni}(m,p)}\setminus\{\gamma(a),\gamma(b)\}$.
By Lemma~\ref{lem:levelProperties}, it holds  that $\NN_L\subseteq \NN_{{\rm uni}(M,p)}$, where 
\begin{align}
M:= \max_{z\in\NN_L} \level_L(z).
\end{align}
The definition of $\XX_{{\rm uni}(m,p)}$ yields that $\XX_L\subseteq \XX_{{\rm uni}(M,p)}$.
Moreover, we can rewrite the
additive Schwarz operator as
\begin{align}\label{eq:QQ}
  \PAS =\sum_{m=0}^M\QQ_m\quad\text{with}\quad \QQ_m :=
\sum_{\ell=0}^L \sum_{\substack{i\in\II_\ell\\ \level_\ell(\widetilde z_{\ell,i}) = m}}
\prec_{\ell,i}.
\end{align}
There holds the following type of strengthened Cauchy-Schwarz inequality.
\begin{lemma}\label{lem:strongCS}
  For all $0\leq m\leq M$, $\edual{\QQ_m (\cdot)}{(\cdot)}$ defines a symmetric positive semi-definite bilinear form on $\H^{1/2}(\Gamma)$.
For $k\in\N_0$, it holds  that
  \begin{align}\label{eq:hypscs}
   \edual{\QQ_m {V_{{\rm uni}(k,p)}}}{{V_{{\rm uni}(k,p)}}} 
    \le C_4
    2^{-(m-k)} \enorm{{V_{{\rm uni}(k,p)}}}^2
    \quad\text{for all }{V_{{\rm uni}(k,p)}}\in\XX_{{\rm uni}(k,p)}.
  \end{align}
  The constant $C_4>0$ depends only on $\TT_0, {\widehat \kappa}_{\max}, p, w_{\min}, w_{\max}$ and
  $\gamma$.
\end{lemma}
\begin{proof}
Symmetry and positive semi-definiteness follow by the symmetry and positive semi-definiteness of the one-dimensional projectors $\prec_{\ell,i}$.
To see \eqref{eq:hypscs}, we only consider closed $\Gamma=\partial\Omega$ and split the proof into two steps. For open $\Gamma\subsetneqq \partial\Omega$ the proof works analogously.  
\\
\noindent
{\bf Step 1:} 
Let $\ell\in\{0,\dots,L\}$ and $i\in\II_\ell$ with  $\level_\ell(\widetilde z_{\ell,i}) = m$. 
We want to estimate $\edual{\prec_{\ell,i} {V_{{\rm uni}(k,p)}}}{{V_{{\rm uni}(k,p)}}}$.
   From the definition \eqref{eq:Pelli} of $\prec_{\ell,i}$,  we infer that
   \begin{align}\label{eq:proof pli}
    \edual{\prec_{\ell,i} {V_{{\rm uni}(k,p)}}}{{V_{{\rm uni}(k,p)}}} =
\frac{\edual{V_{{\rm uni}(k,p)}}{ \overline R_{\ell,i,p}}^2}{\enorm{ \overline R_{\ell,i,p}}^2}.
  \end{align}
 Lipschitz continuity of $\gamma$ gives that $|\widehat R_{\ell,i,p}|_{H^{1/2}(\supp(\widehat R_{\ell,i,p}))}\lesssim|{R}_{\ell,i,p}|_{H^{1/2}(\supp({R}_{\ell,i,p}))}\lesssim|{\overline R}_{\ell,i,p}|_{H^{1/2}(\Gamma)}$.
Hence, Lemma~\ref{lem:basis scaling} with $\sigma=1/2$ shows that $1\lesssim|\overline R_{\ell,i,p}|_{H^{1/2}(\Gamma)}$.
This implies that
\begin{align*}
\norm{ \overline R_{\ell,i,p}}{L^2(\Gamma)}\lesssim  |\supp( \overline R_{\ell,i,p})|^{1/2}\lesssim  |\supp( \overline R_{\ell,i,p})|^{1/2}|\overline{R}_{\ell,i,p}|_{H^{1/2}(\Gamma)}\lesssim|\supp  (\overline R_{\ell,i,p})|^{1/2} \enorm{
    \overline R_{\ell,i,p}}.
    \end{align*} 
With the Cauchy-Schwarz inequality and
  $|\supp(  \overline R_{\ell,i,p})| \simeq \overline h_{{\rm uni}(m)}$ (Lemma~\ref{lem:levelProperties}), this gives that
  \begin{align}\label{eq:proof CSin}
  \begin{split}
&     \edual{\prec_{\ell,i} {V_{{\rm uni}(k,p)}}}{{V_{{\rm uni}(k,p)}}}\lesssim\frac{\dual{\hyp{V_{{\rm uni}(k,p)}}}{ \overline R_{\ell,i,p}}_\Gamma^2}{\enorm{ \overline R_{\ell,i,p}}^2}+\frac{\dual{V_{{\rm uni}(k,p)}}{1}_\Gamma^2\dual{ \overline R_{\ell,i,p}}{1}_\Gamma^2}{\enorm{ \overline R_{\ell,i,p}}^2}
    \\&\quad\lesssim |\supp( \overline R_{\ell,i,p})| \big(\norm{\hyp {V_{{\rm uni}(k,p)}}}{L^2(\supp( \overline R_{\ell,i,p}))}^2+ |\supp( \overline R_{\ell,i,p})|\norm{V_{{\rm uni}(k,p)}}{L^2(\Gamma)}^2\big)\\
    &\quad\lesssim \frac{\overline h_{{\rm uni}(m)}}{\overline h_{{\rm uni}(k)}} \big(
    \norm{\overline h_{{\rm uni}(k)}^{1/2}\hyp {V_{{\rm uni}(k,p)}}}{L^2(\supp( \overline R_{\ell,i,p}))}^2+|\supp( \overline R_{\ell,i,p})|\norm{V_{{\rm uni}(k,p)}}{L^2(\Gamma)}^2\big)\\
    &\quad= 2^{-(m-k)}\big( \norm{\overline h_{{\rm uni}(k)}^{1/2}\hyp {V_{{\rm uni}(k,p)}}}{L^2(\supp( \overline R_{\ell,i,p}))}^2+|\supp( \overline R_{\ell,i,p})|\norm{V_{{\rm uni}(k,p)}}{L^2(\Gamma)}^2\big).
    \end{split}
  \end{align}
\\
\noindent
{\bf Step 2:}
We stress that the choice \eqref{eq:zli} of $\widetilde{z}_{\ell,i}$ and \eqref{unifPatch} show that
  \begin{align*}
  \supp( \overline R_{\ell,i,p}) \stackrel{\eqref{eq:zli}}{\subseteq} \omega_{\ell-1}^{p+1}(\widetilde{z}_{\ell,i})\subseteq\omega_{\ell-1}^{2p+1}(\widetilde{z}_{\ell,i})\stackrel{\eqref{unifPatch}}{\subseteq}
 \omega_{{\rm uni}(m)}^{{C_3}}(\widetilde{z}_{\ell,i}).
  \end{align*}
  Thus, the definition of $\QQ_m$ and Step~1 yield that
  \begin{align*}
    &\edual{\QQ_m {V_{{\rm uni}(k,p)}}}{{V_{{\rm uni}(k,p)}}} = \sum_{\ell=0}^L \sum_{\substack{i\in\II_\ell\\ \level_\ell(\widetilde{z}_{\ell,i}) = m}}
 \edual{\prec_{\ell,i} {V_{{\rm uni}(k,p)}}}{{V_{{\rm uni}(k,p)}}} \\
  & \quad\stackrel{\eqref{eq:proof CSin}}{\lesssim} 2^{-(m-k)} \sum_{\ell=0}^L \sum_{\substack{i\in\II_\ell\\ \level_\ell(\widetilde{z}_{\ell,i}) = m}}\Big(
    \norm{\overline h_{{\rm uni}(k)}^{1/2} \hyp {V_{{\rm uni}(k,p)}}}{L^2(\omega_{{\rm uni}(m)}^{{C_3}}(\widetilde{z}_{\ell,i}))}^2+| \omega_{{\rm uni}(m)}^{{C_3}}(\widetilde{z}_{\ell,i})|\norm{V_{{\rm uni}(k,p)}}{L^2(\Gamma)}^2\Big)\\
    &\quad\stackrel{\eqref{eq:li set}}{=}2^{-(m-k)}\sum_{z\in \NN_L} \sum_{(\ell,i)\in\ZZ_m(z)} \Big(
    \norm{\overline h_{{\rm uni}(k)}^{1/2} \hyp {V_{{\rm uni}(k,p)}}}{L^2(\omega_{{\rm uni}(m)}^{{C_3}}(z))}^2+| \omega_{{\rm uni}(m)}^{{C_3}}(z)|\norm{V_{{\rm uni}(k,p)}}{L^2(\Gamma)}^2\Big).
  \end{align*}
If $z\in\NN_L$ and $(\ell,i)\in\ZZ_m(z)$, it follows $z\in\NN_\ell$ with $\level_\ell(z) = m$.
Lemma~\ref{lem:levelProperties} implies that  $z\in\NN_{{\rm uni}(m)}$. 
 Hence, we can replace in the upper sum $\NN_L$ by $\NN_L\cap \NN_{{\rm uni}(m)}$.
With Lemma \ref{lem:li set}, we further see that
  \begin{align*}
&\sum_{z\in \NN_L\cap \NN_{{\rm uni}(m)}} \sum_{(\ell,i)\in\ZZ_m(z)}\Big( 
    \norm{\overline h_{{\rm uni}(k)}^{1/2} \hyp {V_{{\rm uni}(k,p)}}}{L^2(\omega_{{\rm uni}(m)}^{{C_3}}(z))}^2 +| \omega_{{\rm uni}(m)}^{{C_3}}(z)|\norm{V_{{\rm uni}(k,p)}}{L^2(\Gamma)}^2\Big)\\
    &\quad\lesssim \sum_{z\in\NN_{{\rm uni}(m)}} 
    \Big(\norm{\overline h_{{\rm uni}(k)}^{1/2} \hyp {V_{{\rm uni}(k,p)}}}{L^2(\omega_{{\rm uni}(m)}^{{C_3}}(z))}^2+| \omega_{{\rm uni}(m)}^{{C_3}}(z)|\norm{V_{{\rm uni}(k,p)}}{L^2(\Gamma)}^2\Big)\\
    &\quad \lesssim \norm{\overline h_{{\rm uni}(k)}^{1/2} \hyp {V_{{\rm uni}(k,p)}}}{L^2(\Gamma)}^2+\norm{V_{{\rm uni}(k,p)}}{L^2(\Gamma)}^2.
  \end{align*}
  Note that boundedness $\hyp: H^1(\Gamma)\to L^2(\Gamma)$ as well as the inverse inequality \eqref{eq:inverseIneq2}  prove~that
  \begin{align*}
    \norm{\overline h_{{\rm uni}(k)}^{1/2} \hyp {V_{{\rm uni}(k,p)}}}{L^2(\Gamma)}^2 \lesssim \overline h_{{\rm uni}(k)}\norm{V_{{\rm uni}(k,p)}}{H^1(\Gamma)}^2\lesssim \norm{{V_{{\rm uni}(k,p)}}}{H^{1/2}(\Gamma)}^2.
  \end{align*}
Putting the latter three inequalities together shows that
  \begin{align*}
    \edual{\QQ_m {V_{{\rm uni}(k,p)}}}{{V_{{\rm uni}(k,p)}}} \lesssim 2^{-(m-k)} \big(\norm{{V_{{\rm uni}(k,p)}}}{H^{1/2}(\Gamma)}^2 +\norm{{V_{{\rm uni}(k,p)}}}{L^{2}(\Gamma)}^2\big)
    \simeq 2^{-(m-k)} \enorm{{V_{{\rm uni}(k,p)}}}^2.
  \end{align*}
This finishes the proof.
 \end{proof}

The rest of the proof of the upper bound in Theorem \ref{thm:PAS} follows essentially as in \cite[Lemma~2.8]{transtep96} and is only given for completeness; see also~\cite[Section~4.6]{ffps}.

\begin{proof}[Proof of upper bound in~\eqref{eq:prop:PAS}]
For $k\in\N_0$ let $\gal_{{\rm uni}(k,p)} : \H^{1/2}(\Gamma) \to \XX_{{\rm uni}(k,p)}$ denote the Galerkin projection onto
$\XX_{{\rm uni}(k,p)}$ with respect to the scalar product $\edual\cdot\cdot$, i.e.,
\begin{align}
  \edual{\gal_{{\rm uni}(k,p)} v}{V_{{\rm uni}(k,p)}} = \edual{v}{V_{{\rm uni}(k,p)}} \quad\text{for all }
v\in\H^{1/2}(\Gamma),V_{{\rm uni}(k,p)}\in\XX_{{\rm uni}(k,p)}.
\end{align}
Note that $\gal_{{\rm uni}(k,p)}$ is the orthogonal projection onto $\XX_{{\rm uni}(k,p)}$ with respect to
the energy norm $\enorm\cdot$. 
Moreover, we set $\gal_{{\rm uni}(-1,p)} := 0$.
The proof is split into three steps.
\\
\noindent
{\bf Step 1:} 
Let $V_L\in\XX_L\subseteq\XX_{{\rm uni}(M,p)}$.
Lemma~\ref{lem:levelProperties} and the boundedness of the local mesh-ratio by $\widehat \kappa_{\max}$ yield the existence of a constant $C\in \N_0$, which depends only on $\TT_0,\widehat \kappa_{\max},\gamma$, and $p$, such that $\NN_\ell\cap\omega_{\ell-1}^{p+2}({z})\subseteq \NN_{{\rm uni}(m+C,p)}$ for all nodes $z\in\NN_\ell$ with $\level_\ell(z)=m$.
Lemma~\ref{lem:properties for B-splines}~\eqref{item:spline basis} hence proves that
$\XX_{\ell,i}\subseteq\XX_{{\rm uni}(m+C,p)}$ for all $m\in\{0,\dots,M\},\ell\in\{0,\dots,L\},$ and  $i\in\II_\ell$ with $\level_\ell(\widetilde{z}_{i,\ell})=m$. 
Therefore, the range of  $\QQ_m$ is a subspace of  $\XX_{{\rm uni}(m+C,p)}$. 
This shows that
\begin{align}\label{eq:qmvl}
\edual{\QQ_m V_L}{V_L}
 = \edual{\QQ_m V_L}{\gal_{{\rm uni}(m+C,p)}V_L}.
\end{align}
\\
\noindent
{\bf Step 2:}
In Lemma~\ref{lem:strongCS}, we saw that $\edual{\QQ_m(\cdot)}{(\cdot)}$ defines a symmetric positive semi-definite bilinear form and hence satisfies  a Cauchy-Schwarz inequality.
This and \eqref{eq:qmvl} yield that
\begin{align*}
&  \edual{\QQ_m V_L}{V_L}
 = \edual{\QQ_m V_L}{\gal_{{\rm uni}(m+C,p)}V_L}
 = \sum_{k=0}^{m+C}\edual{\QQ_mV_L}{(\gal_{{\rm uni}(k,p)}-\gal_{{\rm uni}(k-1,p)})V_L}
 \\&\quad\le
 \sum_{k=0}^{m+C}\edual{\QQ_m V_L}{V_L}^{1/2}\edual{\QQ_m(\gal_{{\rm uni}(k,p)}-\gal_{{\rm uni}(k-1,p)})V_L}{(\gal_{{\rm uni}(k,p)}-\gal_{{\rm uni}(k-1,p)})V_L}^{1/2}.
\end{align*}
For the second scalar product, we apply Lemma~\ref{lem:strongCS} and obtain that
\begin{align*}
 &\edual{\QQ_m(\gal_{{\rm uni}(k,p)}-\gal_{{\rm uni}(k-1,p)})V_L}{(\gal_{{\rm uni}(k,p)}-\gal_{{\rm uni}(k-1,p)})V_L}\\
 &\quad \lesssim 
2^{-(m-k)}\,\enorm{(\gal_{{\rm uni}(k,p)}-\gal_{{\rm uni}(k-1,p)})V_L}^2
 = 
 2^{-(m-k)}\edual{(\gal_{{\rm uni}(k,p)}-\gal_{{\rm uni}(k-1,p)})^2 V_L}{V_L}.
\end{align*}
Note that $(\gal_{{\rm uni}(k,p)}-\gal_{{\rm uni}(k-1,p)})^2=\gal_{{\rm uni}(k,p)}-\gal_{{\rm uni}(k-1,p)}$, since $\gal_{{\rm uni}(k,p)}-\gal_{{\rm uni}(k-1,p)}$ is again an orthogonal projection.
\\
\noindent
{\bf Step 3:}
With the  representation~\eqref{eq:QQ} of $\PAS$, the two inequalities from Step~2, and the
Young inequality, we infer for all $\delta>0$  that
\begin{align*}
  \edual{\PAS V_L}{V_L} &\stackrel{\eqref{eq:QQ}}{=} \sum_{m=0}^M \edual{\QQ_m V_L}{V_L} 
  \lesssim  \frac\delta{2} \sum_{m=0}^M \sum_{k=0}^{m+C}
  2^{-(m-k)/2} \edual{\QQ_m  V_L}{V_L}\\
& \hspace{35mm}+ \frac{\delta^{-1}}2 \sum_{m=0}^M\sum_{k=0}^{m+C}
  2^{-(m-k)/2} \edual{(\gal_{{\rm uni}(k,p)}-\gal_{{\rm uni}(k-1,p)})V_L}{V_L}.
\end{align*}
We abbreviate  $\sum_{k=-C}^\infty
2^{-k/2} =: K<\infty$. Changing the summation
indices in the second sum, we see with $V_L\in\XX_L\subseteq\XX_{{\rm uni}(M,p)}\subseteq \XX_{{\rm uni}(M+C,p)}$ that
\begin{align*}
  \edual{\PAS V_L}{V_L}
 &\lesssim K\frac\delta{2} \, \sum_{m=0}^M \edual{\QQ_m V_L}{V_L} \\&\quad+ 
 \frac{\delta^{-1}}2 \, \sum_{k=0}^{M+C} \sum_{m=\max(k-C,0)}^{M} 2^{-(m-k)/2} 
 \edual{(\gal_{{\rm uni}(k,p)}-\gal_{{\rm uni}(k-1,p)})V_L}{V_L} \\
  &\leq K\frac\delta{2} \, \sum_{m=0}^M  \edual{\QQ_m V_L}{V_L} + 
  K \frac{\delta^{-1}}2 \,\sum_{k=0}^{M+C}
  \edual{(\gal_{{\rm uni}(k,p)}-\gal_{{\rm uni}(k-1,p)})V_L}{V_L} \\
  &= K\frac\delta{2} \, \sum_{m=0}^M  \edual{\QQ_m V_L}{V_L} + K \frac{\delta^{-1}}2
\edual{\gal_{{\rm uni}(M+C,p)}V_L}{V_L}\\
&\stackrel{\eqref{eq:QQ}}= K\frac\delta{2} \edual{\PAS V_L}{V_L} + K \frac{\delta^{-1}}2 \edual{V_L}{V_L}.
\end{align*}
Choosing $\delta>0$ sufficiently small and absorbing the first-term on the right-hand 
side on the left, we prove the upper bound in~\eqref{eq:prop:PAS}.
\end{proof}


\section{Local multilevel diagonal preconditioner for the weakly-singular case}
\label{section:weakly precond} 
Finally, we generalize the results of the previous sections to the weakly-singular integral equation.
The main tool in the following is  Maue's formula (see, e.g. \cite{hilbert})
\begin{align}\label{eq:Maue}
\dual{\WW u}{v}_\Gamma=\dual{\VV\partial_\Gamma u}{\partial_\Gamma v}_\Gamma\quad\text{for all}\quad u,v\in\H^1(\Gamma).
\end{align}
For similar proofs in the case of piecewise constant ansatz functions, we refer to \cite{transtep96} (uniform meshes) resp.\ \cite{ffpsfembem} (adaptive meshes).
Throughout this section, let $(\KK_\ell)_{\ell\in\N_0}$ be a sequence of refined knot vectors, i.e., $\KK_\ell,\KK_{\ell+1}\in\K$ with $\KK_{\ell+1}\in\refine(\KK_\ell)$, 
and let $L\in\N_0$.
For each $\Psi_L\in\YY_L$, we consider the unique decomposition
\begin{align}
\Psi_L=\Psi_L^{00}+\Psi_L^{0},
\quad\text{where}\quad
\Psi_L^{00}:=\dual{\Psi_L}{1}_\Gamma/|\Gamma|\quad\text{and}\quad\Psi_L^0:=\Psi_L-\Psi_L^{00}.
\end{align}

Note that 
\begin{align}
\Psi_L^{00}\in\YY^{00}:=\linhull\{1\}\quad\text{and}\quad \Psi_L^0\in\YY_L^0:=\set{\Psi_L\in\YY_L}{\dual{\Psi_L}{1}_\Gamma=0}.
\end{align}
With hidden constants, which depend only on $\Gamma$, it holds that
 \begin{align}\label{eq:00* stable}
 \norm{\Psi_L^{00}}{\VV}\lesssim \norm{\Psi_L}{\VV}\quad\text{and}\quad \norm{\Psi_L^0}{\VV}\lesssim\norm{\Psi_L}{\VV}.
 \end{align}
Recall the spaces $\XX_L, \widetilde \XX_\ell$ and  $\XX_{\ell,i}$ from Section \ref{section:precond}. For $\ell\in\{0,\dots,L\}$, set
\begin{align}\label{eq:tilY}
 \widetilde\YY_{\ell}^0:=\partial_\Gamma\widetilde\XX_{\ell}\stackrel{\eqref{eq:reddecomp}}=\sum_{i\in\II_\ell}\YY_{\ell,i}^0\quad\text{with}\quad\YY_{\ell,i}^0:=\partial_\Gamma\XX_{\ell,i}.
\end{align}
Recall that we only consider non-rational splines in the weakly-singular case; see Section~\ref{section:igabem}. 
For rational splines, the following lemma is in general false.
\begin{lemma}\label{lem:Y*=X'}
It holds that
\begin{align}\label{eq:Y*=X'}
\YY_L^0=\partial_\Gamma\XX_L.
\end{align}
For closed  $\Gamma=\partial\Omega$, we even have that
\begin{align}\label{eq:Y*=X' 2}
\YY_L^0=\partial_\Gamma \XX_L^0,\quad \text{where}\quad \XX_L^0:=\set{V_L\in\XX_L}{\dual{V_L}{1}_\Gamma=0}.
\end{align}
\end{lemma}
 \begin{proof}
Since $\Gamma$ is connected, the kernel of $\partial_\Gamma:\XX_L\to L^2(\Gamma)$ is one-dimensional for $\Gamma=\partial\Omega$ and zero-dimensional for  $\Gamma\subsetneqq\partial\Omega$.
Together with \eqref{eq:hypsing basis} and \eqref{eq:weaksing basis}, linear algebra yields  for $\Gamma=\partial\Omega$ that $N_L-1=\dim\XX_L=\dim \YY_L$, $N_L-2=\dim\partial_\Gamma\XX_L=\dim \XX_L^0=\dim\YY_L^0$ and for $\Gamma\subsetneqq\partial\Omega$  that $N_L-1=\dim\YY_L$, $N_L-2=\dim\XX_L=\dim\partial_\Gamma\XX_L=\dim\YY_L^0$.
To see \eqref{eq:Y*=X'}, it thus remains to prove that $\partial_\Gamma\XX_L\subseteq \YY_L^0$.
The inclusion $\partial_\Gamma\XX_L\subseteq\YY_L$ follows directly from Lemma~\ref{lem:properties for B-splines} \eqref{item:derivative of splines}.
We stress that for any $v\in\H^{1/2}(\Gamma)$, it holds that $\dual{\partial_\Gamma v}{1}_\Gamma=0$. Thus, any function in $\partial_\Gamma\XX_L$ has vanishing integral mean, which  concludes the proof of \eqref{eq:Y*=X'}. Finally, \eqref{eq:Y*=X' 2} follows immediately from the fact that $\XX_L=\linhull\{1\}+\XX_L^0$. 
 \end{proof}
Define the orthogonal projections on $\YY^{00}$ resp. $\YY_{\ell,i}$ via
\begin{align}
\begin{split}
\edualV{\prec^{00} \chi}{\Psi^{00}}=\edualV{\chi}{\Psi^{00}}\quad\text{for all }\chi\in\H^{-1/2}(\Gamma), \Psi^{00}\in\YY^{00},\\
  \edualV{\prec_{\ell,i}^0 \chi}{\Psi_{\ell,i}^0} = \edualV{\chi}{\Psi_{\ell,i}^0} \quad\text{for all } \chi\in \H^{-1/2}(\Gamma),
  \Psi_{\ell,i}^0\in\YY_{\ell,i}^0.
  \end{split}
\end{align}
With Lemma \ref{lem:Y*=X'}, we see the decomposition
 \begin{align}
 \YY_L=\YY^{00}+\YY^0_L\stackrel{\eqref{eq:Y*=X'}}{=}\YY^{00}+\partial_\Gamma\XX_L\stackrel{\eqref{eq:space decomp}}{=}\YY^{00}+\sum_{\ell=0}^L\widetilde\YY_{\ell}^0\stackrel{\eqref{eq:tilY}}{=}\YY^{00}+\sum_{\ell=0}^L\sum_{i\in\II_\ell} \YY_{\ell,i}^0
 \end{align}
with the corresponding additive Schwarz operator 
\begin{align}\label{eq:PASV}
\PASV:=\prec^{00}+\sum_{\ell=0}^L\sum_{i\in\II_\ell}\prec_{\ell,i}^0.
\end{align}
\begin{theorem}\label{thm:PASV}
  The additive Schwarz operator $\PASV : \H^{-1/2}(\Gamma)\to \YY_L$ satisfies that
  \begin{align}\label{eq:prop:PASV}
    \evminV \enormV{\Psi_L}^2 \leq \edualV{\PASV \Psi_L}{\Psi_L} \leq \evmaxV \enormV{\Psi_L}^2 \quad\text{for all } \Psi_L\in \YY_L,
  \end{align}
where the constants $\evminV,\evmaxV>0$ depend only on $\TT_0$, $\widehat{\kappa}_{\max}, p, w_{\min}, w_{\max}$, and $\gamma$. 
\end{theorem}

 \begin{proof}
We only prove the assertion for closed  $\Gamma=\partial\Omega$. For open  $\Gamma\subsetneqq\partial\Omega$, the proof works analogously with  $\norm{\cdot}{\WW}^2=\dual{\WW(\cdot)}{(\cdot)}_\Gamma$. 

\noindent
{\bf{Step 1:} }
First, we prove the lower bound of \eqref{eq:prop:PASV}.
 We have to find a stable decomposition for any $\Psi_L\in\YY_L$.
 Due to Lemma \ref{lem:Y*=X'}, there exists $V_L^0\in\XX_L^0$ with $\partial_\Gamma V_L^0=\Psi_L^0$.
In Section~\ref{sec:PAS:lb}, we provided a decomposition $V_L^0=\sum_{\ell=0}^L\sum_{i\in\II_\ell} {V_{\ell,i}}$ such that
$ \sum_{\ell=0}^L\sum_{i\in\II_\ell} \norm{V_{\ell,i}}{\WW}^2\lesssim \norm{V_L^0}{\WW}^2.$
 This provides us with a decomposition  
 \begin{align*}
 \Psi_L^0=\partial_\Gamma V_L^0=\sum_{\ell=0}^L\sum_{i\in\II_\ell} \Psi_{\ell,i}^0, \quad\text{where}\quad \Psi_{\ell,i}^0:=\partial_{\Gamma}V_{\ell,i}\in\YY_{\ell,i}^0.
 \end{align*}
 Maue's formula \eqref{eq:Maue} and $\dual{V_L^0}{1}_\Gamma=0$ hence show that
 \begin{align*}
 \sum_{\ell=0}^L\sum_{i\in\II_\ell} \norm{\Psi_{\ell,i}^{0}}{\VV}^2=\sum_{\ell=0}^L\sum_{i\in\II_\ell}\dual{\WW V_{\ell,i}}{V_{\ell,i}}_\Gamma\le \sum_{\ell=0}^L\sum_{i\in\II_\ell}\norm{V_{\ell,i}}{\WW}^2\lesssim\norm{V_L^0}{\WW}^2=\norm{\Psi_L^0}{\VV}^2.
 \end{align*}
With this and \eqref{eq:00* stable}, we  finally conclude that
\begin{align*}
\norm{\Psi_L^{00}}{\VV}^2+\sum_{\ell=0}^L\sum_{i\in\II_\ell} \norm{\Psi_{\ell,i}^0}{\VV}^2\lesssim\norm{\Psi_L}{\VV}^2.
\end{align*}
As in Section~\ref{sec:PAS:lb}, this proves the lower bound.

\noindent
{\bf Step 2:} For the upper bound of \eqref{eq:prop:PASV}, let   $\Psi_L=\Psi_L^0+\Psi_L^{00}\in\YY_L$ with an arbitrary decomposition $\sum_{\ell=0}^L\sum_{i\in\II_\ell}\Psi_{\ell,i}^0=\Psi_L^0$, where $\Psi_{\ell,i}^0\in\YY_{\ell,i}^0$.
In particular, it  holds that $\Psi_{\ell,i}^0=\alpha_{\ell,i} \partial_\Gamma \overline B_{\ell,i,p}$ 
with some $\alpha_{\ell,i}\in\R$.
We define
\begin{align*}
V_L:=\sum_{\ell=0}^L\sum_{i\in \II_\ell}V_{\ell,i}\quad\text{with}\quad V_{\ell,i}:=\alpha_{\ell,i} \overline B_{\ell,i,p}.
 \end{align*}
It is well known from additive Schwarz theory
that the existence of a uniform upper bound in Theorem~\ref{thm:PAS} is equivalent to  
$ \norm{V_L}{\WW}^2\lesssim\sum_{\ell=0}^L\sum_{i\in\II_\ell}\norm{V_{\ell,i}}{\WW}^2$ for all decompositions $V_L=\sum_{\ell=0}^L\sum_{i\in\II_\ell}V_{\ell,i}$; see, e.g.,  \cite[Lemma~3.1]{zhang92}.
Maue's formula~\eqref{eq:Maue} yields that
 \begin{align}\label{eq:psi to v}
 \norm{\Psi_L^0}{\VV}^2=\dual{\WW V_L}{V_L}_{\Gamma}\le \norm{V_{L}}{\WW}^2 \lesssim\sum_{\ell=0}^L\sum_{i\in\II_\ell}\norm{V_{\ell,i}}{\WW}^2.
 \end{align}
With $|\supp( V_{\ell,i})|=|\supp(\overline B_{\ell,i,p})|\lesssim1$, Lemma \ref{lem:basis scaling} shows  that
\begin{align}\label{eq:vli 1}
\dual{V_{\ell,i}}{1}_\Gamma^2\lesssim \alpha_{\ell,i}^2|\supp( V_{\ell,i})|^2\lesssim\alpha_{\ell,i}^2\stackrel{\eqref{eq:basis scaling}}\lesssim\alpha_{\ell,i}^2 |\widehat B_{\ell,i,p}|_{H^{1/2}(\supp( \widehat B_{\ell,i,p}))}^2\lesssim\alpha_{\ell,i}^2 |\widehat B_{\ell,i,p}|_{H^{1/2}(a,b)}^2.
\end{align}
Lipschitz continuity of $\gamma$ proves $|\widehat B_{\ell,i,p}|^2_{H^{1/2}(a,b)}\lesssim|  B_{\ell,i,p}|_{H^{1/2}(\Gamma)}^2 \lesssim| \overline B_{\ell,i,p}|_{H^{1/2}(\Gamma)}^2$.
With the equivalence $|\cdot|^2_{H^{1/2}(\Gamma)}\simeq \dual{\WW(\cdot)}{(\cdot)}_\Gamma$ on $H^{1/2}(\Gamma)$,  \eqref{eq:vli 1} becomes
\begin{align*}
\dual{V_{\ell,i}}{1}^2\lesssim \alpha_{\ell,i}^2| \overline B_{\ell,i,p}|_{H^{1/2}(\Gamma)}^2\simeq\dual{\WW V_{\ell,i}}{V_{\ell,i}}_\Gamma.
\end{align*}
By definition of the  norm $\enorm{\cdot}$, we infer $\enorm{V_{\ell,i}}^2\lesssim \dual{\WW V_{\ell,i}}{V_{\ell,i}}_\Gamma$.
Hence, \eqref{eq:psi to v} yields that
 \begin{align*}
 \norm{\Psi_L^0}{\VV}^2\stackrel{\eqref{eq:psi to v}} \lesssim \sum_{\ell=0}^L\sum_{i\in\II_\ell} \dual{\WW V_{\ell,i}}{V_{\ell,i}}_\Gamma\stackrel{\eqref{eq:Maue}}=\sum_{\ell=0}^L\sum_{i\in\II_\ell} \norm{\Psi_{\ell,i}}{\VV}^2.
 \end{align*}
We conclude that
\begin{align*}
\norm{\Psi_L}{\VV}^2\lesssim\norm{\Psi_L^{00}}{\VV}^2+\norm{\Psi_L^0}{\VV}^2\lesssim\norm{\Psi_L^{00}}{\VV}^2+\sum_{\ell=0}^L\sum_{i\in\II_\ell}\norm{\Psi_{\ell,i}}{\VV}^2.
\end{align*}
Since $\Psi_L^{00}+\sum_{\ell=0}^L\sum_{i\in\II_\ell}\Psi_{\ell,i}=\Psi_L$ was an arbitrary decomposition, standard additive Schwarz theory proves the upper bound.
 \end{proof}


\section{Numerical experiments}\label{sec:numerics}

In this section, we present a matrix version of Theorem \ref{thm:PAS} and Theorem \ref{thm:PASV}. 
We apply these theorems to define preconditioners for some numerical examples.
Throughout this section, let $(\KK_\ell)_{\ell\in\N_0}$ be a sequence of refined knot vectors, i.e., $\KK_\ell,\KK_{\ell+1}\in\K$ with $\KK_{\ell+1}\in\refine(\KK_\ell)$, 
and let $L\in\N_0$.
For the hypersingular equation \eqref{eq:hypsing}, we allow for arbitrary positive initial weights $\mathcal{W}_0$. 
Whereas, whenever we consider the weakly-singular integral equation \eqref{eq:weaksing}, we suppose that all weights in $\mathcal{W}_0$ are equal to one, wherefore the denominator  satisfies that $w=1$. 
The Galerkin approximations $U_\ell\in\XX_\ell$ for the hypersingular case resp. $\Phi_\ell\in\YY_\ell$ for the weakly-singular case satisfy that
\begin{align}
\edual{U_\ell}{V_\ell}=\dual{f}{V_\ell}_\Gamma\quad\text{resp.}\quad\edualV{\Phi_\ell}{\Psi_\ell}=\dual{g}{\Psi_\ell}_\Gamma\quad \text{for all }V_\ell\in\XX_\ell, \Psi_\ell\in\YY_\ell.
\end{align}
The discrete solutions $U_\ell,\Phi_\ell$ are obtained by solving a linear system of equations 
\begin{align}
\mathbf{W}_\ell\mathbf{x}_\ell=\mathbf{f}_\ell \quad \text{resp. }\quad \mathbf{V}_\ell\mathbf{y}_\ell=\mathbf{g}_\ell,
\end{align}
where
\begin{align}
U_\ell=\sum_{k=1}^{N_\ell-1-o} (\mathbf{x}_\ell)_k \overline R_{\ell,k-p+o,p} \quad\text{resp.}\quad\Phi_\ell=\sum_{k=1}^{N_\ell-1} (\mathbf{y}_\ell)_k B_{\ell,k-(p-1),p-1},
\end{align}
and 
\begin{align}
\mathbf W_\ell=\Big(\edual{\overline R_{\ell,k-p+o,p}}{\overline R_{\ell,j-p+o,p}}\Big)_{j,k=1}^{N_\ell-1-o}, \quad\mathbf{f}_\ell=\Big(\dual{f}{\overline R_{\ell,j-p+o,p}}_\Gamma\Big)_{j=1}^{N_\ell-1-o},
\end{align}
resp.\
\begin{align}
\mathbf{V}_\ell=\Big(\edualV{B_{\ell,k-(p-1),p-1}}{B_{\ell,j-(p-1),p-1}}\Big)_{j,k=1}^{N_\ell-1},\quad \mathbf{g}_\ell=\Big(\dual{g}{B_{\ell,j-(p-1),p-1}}_\Gamma\Big)_{j=1}^{N_\ell-1}.
\end{align}

For any $L\in \N_0$, we aim to derive preconditioners $(\SSAS)^{-1}$ resp.\ $(\SSASV)^{-1}$ for the Galerkin matrices $\mathbf{W}_L$ resp.\ $\mathbf{V}_L$.
For their definition, we first have to introduce the following transformation matrices.
For $0\le\ell\le L$, let $\IW:\XX_\ell\to\XX_L$, $\HW:\widetilde{\XX}_\ell\to\XX_\ell$ and $\IV:\YY_\ell\to\YY_L$ be the canonical embeddings, i.e., the formal identities, with matrix representations $\IIW\in \R^{(N_L-1-o)\times (N_\ell-1-o)}$, $\HHW\in\R^{(N_\ell-1-o)\times(\# \II_\ell)}$ and $\IIV\in \R^{(N_L-1)\times (N_\ell-1)}$.
Further, let $\HHV\in\R^{(N_\ell-1)\times (\#\II_\ell)}$ be the matrix that represents the B-spline derivatives in $\widetilde\YY_\ell^0$ as B-splines in $\YY_\ell$, i.e., 
\begin{align}
\partial_\Gamma \overline{B}_{\ell,i(k),p}=\sum_{j=1}^{N_\ell-1} (\HHV)_{jk} B_{\ell,j-(p-1),p-1} \quad \text{for }k=1,\dots,\# \II_\ell,
\end{align}
with the monotonuously increasing bijection $i(\cdot):\{1,\dots,\#\II_\ell\}\to\II_\ell$.
All these matrices can be computed with the help of Lemma \ref{lem:properties for B-splines}.
Finally, let $\mathbf{1}\in\R^{(N_L-1)\times (N_L-1)}$ be the constant one matrix, i.e., 
\begin{align}(\mathbf{1})_{jk}=1\quad\text{for }j,k=1,\dots,N_L-1.
\end{align}
For any quadratic matrix $\mathbf{A}$, we define the corresponding diagonal matrix ${\rm diag}(\mathbf{A})=(\mathbf{A}_{jk}\cdot \delta_{jk})_{j,k}$.
We consider
\begin{align}\label{eq:precW}
(\SSAS)^{-1}:=\sum_{\ell=0}^L\IIW \HHW{\rm diag}((\HHW)^T\mathbf{W}_\ell \HHW)^{-1}(\HHW)^T(\IIW)^T,
\end{align}
resp.\
\begin{align}\label{eq:precV}
(\SSASV)^{-1}:=\edualV{1}{1}^{-1}\mathbf{1}+\sum_{\ell=0}^L\IIV \HHV{\rm diag}((\HHV)^T\mathbf{V}_\ell \HHV)^{-1}(\HHV)^T(\IIV)^T.
\end{align}
Note that, by the partition of unity property from Lemma \ref{lem:properties for B-splines} \eqref{item:B-splines partition}, there holds that
\begin{align}\edualV{1}{1}=\sum_{j,k=1}^{N_L-1}(\mathbf{V}_L)_{jk}.\end{align}
Instead of solving $\mathbf{W}_L\mathbf{x}_L=\mathbf{f}_L$ resp.\ $\mathbf{V}_L\mathbf{y}_L=\mathbf{g}_L$, we consider the preconditioned  systems
\begin{align}\label{eq:precondSys}
(\SSAS)^{-1}\mathbf W_L\mathbf x_L=(\SSAS)^{-1}\mathbf f_L\quad\text{resp.}\quad(\SSASV)^{-1}\mathbf V_L\mathbf y_L=(\SSASV)^{-1}\mathbf g_L.
\end{align}
Elementary manipulations verify  that the preconditioned matrices $(\SSAS)^{-1}\mathbf W_L$ resp.\ $(\SSASV)^{-1}\mathbf V_L$ are just the matrix representations of $\PAS|_{\XX_L}:\XX_L\to\XX_L$ resp.\ $\PASV|_{\YY_L}:\YY_L\to\YY_L$.
Theorem~\ref{thm:PAS} resp.\ Theorem~\ref{thm:PASV} then immediately prove the next corollary, which states uniform boundedness of the condition number of the preconditioned systems.

For a symmetric and positive definite matrix $\AA$, we denote $\dual\cdot\cdot_\AA :=
\dual{\AA\cdot}\cdot_2$, and by $\norm\cdot\AA$ the
corresponding norm resp.\ induced matrix norm.
Here, $\dual\cdot\cdot_2$ denotes the Euclidean inner product.
The condition number $\cond_\AA$ of a quadratic matrix $\BB$ of same dimension as $\AA$ reads
\begin{align}
  \cond_\AA(\BB) := \norm{\BB}\AA\norm{\BB^{-1}}\AA.
\end{align}

\begin{corollary}\label{cor:main}
The matrices $(\SSAS)^{-1}, (\SSASV)^{-1}$ are symmetric and positive definite with respect to
$\dual\cdot\cdot_2$, and $\PPAS:
= (\SSAS)^{-1}\mathbf W_L$ resp.\  $\PPASV:=(\SSASV)^{-1}\mathbf V_L$ are  symmetric and positive definite with respect to 
$\dual\cdot\cdot_{\SSAS} $ resp.\ $\dual\cdot\cdot_{\SSASV} $.
Moreover, the minimal and maximal eigenvalues of the matrices
$\PPAS$ resp.\ $\PPASV$ satisfy that
\begin{align}\label{eq:numeric ev}
   \evmin\le\lambda_{\min}(\PPAS) \le \lambda_{\max}(\PPAS)\le\evmax,
\end{align}
resp.\
\begin{align}
   \evminV\le\lambda_{\min}(\PPASV) \le \lambda_{\max}(\PPASV)\le\evmaxV,
\end{align}
with the constants $\evmin,\evmax$ from Theorem \ref{thm:PAS} and $\evminV,\evmaxV$ from Theorem \ref{thm:PASV}.
In particular, the condition number of the additive Schwarz matrices
$\PPAS$ resp.\ $\PPASV$ is bounded by
\begin{align}
  \cond_{\SSAS}(\PPAS) \leq \evmax/\evmin \quad\text{resp.}\quad   \cond_{\SSASV}(\PPASV) \leq \evmaxV/\evminV.
\end{align}
Recall that these eigenvalue bounds depend only on $\TT_0$, ${\widehat \kappa}_{\max}, p, w_{\min}, w_{\max}$ and $\gamma$.
\end{corollary}
\begin{proof}
We only consider the hypersingular case. The weakly-singular case can be treated analoguously. 
Due to \eqref{eq:space decomp} the operator $\PAS|_{\XX_L}$ is positive definite with respect to $\edual{\cdot}{\cdot}$.
This proves for any $V_L\in\XX_L$ with corresponding coefficient vector $\mathbf{z}_L$ that
\begin{align*}
\dual{(\SSAS)^{-1}\mathbf{W}_L \mathbf z_L}{\mathbf{W}_L\mathbf z_L}_2=\edual{\PAS V_L}{V_L}>0.
\end{align*}
Symmetry and positive definiteness of $\PPAS$ with respect to $\dual{\cdot}{\cdot}_{\SSAS}$  follow immediately by symmetry and positive definiteness of $\mathbf{W}_L$.
Theorem \ref{thm:PAS} and the fact that $\PPAS$ is just the matrix representation of $\PAS|_{\XX_L}$ show \eqref{eq:numeric ev}.
Finally, note that the condition number $\cond_{\SSAS}(\PPAS)$ is just the ratio of the maximal and the minimal eigenvalue of $\PPAS$.
\end{proof}
The corollary can be applied for iterative solution methods such as GMRES~\cite{saad}
or CG~\cite{saad03} to solve~\eqref{eq:precondSys}.
Here, the relative residual of the $j$-th residual depends only on the condition number $ \cond_{\SSAS}(\PPAS)$ resp.\ $ \cond_{\SSASV}(\PPASV)$.
Hence, Corollary \ref{cor:main} proves that the iterative scheme together with the preconditioners $(\SSAS)^{-1}$ resp.\ $(\SSASV)^{-1}$ is optimal in the following sense:
The number of iterations to reduce the relative residual under the tolerance $\epsilon>0$ is bounded by a constant, which depends only on $\TT_0$, ${\widehat \kappa}_{\max}, p, w_{\min}, w_{\max}$ and $\gamma$.

\begin{remark}\label{rem:linear}
  The application of the preconditioners $(\SSAS)^{-1}$ resp.\ $(\SSASV)^{-1}$ on a vector $\mathbf z_L$ can be done efficiently in $\OO(N_L)$ operations.
  Furthermore, the storage requirements of the preconditioners, i.e., the memory consumption of all the tranformation matrices $\mathbf{id}$ and
the diagonal matrices $\diag(\cdot)$ in the sum  is $\OO(N_L)$.
  This implies the optimal linear complexity of our preconditioners.
  A detailed description of an algorithm, which implements the matrix-vector multiplication, can be found in our recent
  work~\cite[Algorithm~1]{ffps} for some local multilevel preconditioner for the
  hypersingular integral operator on adaptively refined meshes, resp.\ in~\cite{yserentant} for some hierarchical basis
  preconditioner.
\end{remark}

In the following subsections, we numerically show the optimality of the proposed preconditioners.
In all examples, the exact solution is known and singular, wherefore adaptive methods are preferable.
To steer the mesh refinement, we apply the following  adaptive Algorithm~\ref{the algorithm} proposed in \cite[Algorithm 3.1]{resigabem} for the weakly-singuar case resp.\ in \cite[Algorithm~3.1 (with $\vartheta=0$)]{hypiga}  for the hypersingular case.
In each experiment, we choose $\theta=0.9$ and employ the weighted-residual error indicators $\eta_\ell(z)$ to steer the refinement.

\begin{algorithm}\label{the algorithm}
\textbf{Input:} Adaptivity parameter $0<\theta\le1$, polynomial order $p\in \N$, initial knots $\KK_0$, initial weights $\mathcal{W}_0$.\\
\textbf{Adaptive loop:} For each $\ell=0,1,2,\dots$ iterate the following steps {\rm(i)--(vi)}:
\begin{itemize}
\item[\rm(i)] Compute discrete approximation $U_\ell\in\XX_\ell$ in the hypersingular case resp.\ $\Phi_\ell\in\YY_\ell$ in the weakly-singular case.
\item[\rm(ii)] Compute refinement indicators $\eta_\ell(z)$
for all nodes ${z}\in\NN_\ell$.
\item[\rm(iii)] Determine a minimal set of nodes $\MM_\ell\subseteq\NN_\ell$ such that
\begin{align}\label{eq:Doerfler}
 \theta\,\eta_\ell^2 \le \sum_{{z}\in\MM_\ell}\eta_\ell({z})^2.
\end{align}
\item[\rm(iv)] If both nodes of an element $T\in\mathcal{T}_\ell$ belong to $\mathcal{M}_\ell$, $T$  will be marked.
\item[\rm(v)] For all other\footnote{These are the nodes $z\in\mathcal{M}_\ell$ whose neighbors are both not marked.} nodes $z\in\mathcal{M}_\ell$, the multiplicity will be increased if $z$ satisfies that $z\not\in\{a,b\}$ and $\#_\ell z<p$, otherwise the elements, which contain one of these nodes $z\in\mathcal{M}_\ell$, will be marked.
\item[\rm(vi)] Refine all marked elements $T\in\mathcal{T}_\ell$ by bisection (insertion of a node with multiplicity one) of the corresponding element $\gamma^{-1}(T)$ in the parameter domain.
Use further bisections to guarantee that the new knots $\KK_{\ell+1}$ satisfy that
\begin{align}\label{eq:kappa small}
\widehat \kappa_{\ell+1}\leq 2\widehat \kappa_0.
\end{align}
\end{itemize}
\textbf{Output:} Approximate solutions $U_\ell$ resp.\ $\Phi_\ell$ and error estimators $\eta_\ell$ for all $\ell \in \N_0$.
\end{algorithm}

The resulting linear systems are solved by PCG.
We refer to the recent work \cite{optpcg} for the interplay of PCG solver and optimal adaptivity.
We compare the preconditioners to simple diagonal preconditioning.
In all experiments the initial vector in the PCG-algorithm is set to $0$ and the tolerance parameter $\epsilon>0$ for the relative residual is $\epsilon=10^{-8}$.

\begin{figure}
\psfrag{pacman (Section 5.3)}[c][c]{}
\psfrag{slit (Section 5.4)}[c][c]{}
\psfrag{45}[r][r]{\tiny $45^\circ$}
\begin{center}
\includegraphics[width=0.4\textwidth,clip=true]{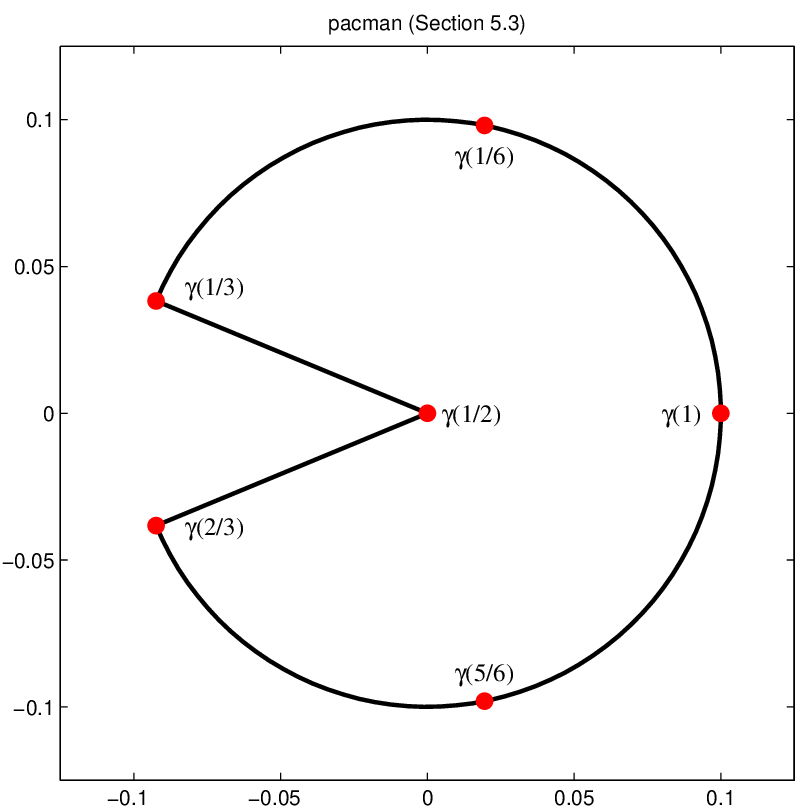}\quad
\includegraphics[width=0.4\textwidth,clip=true]{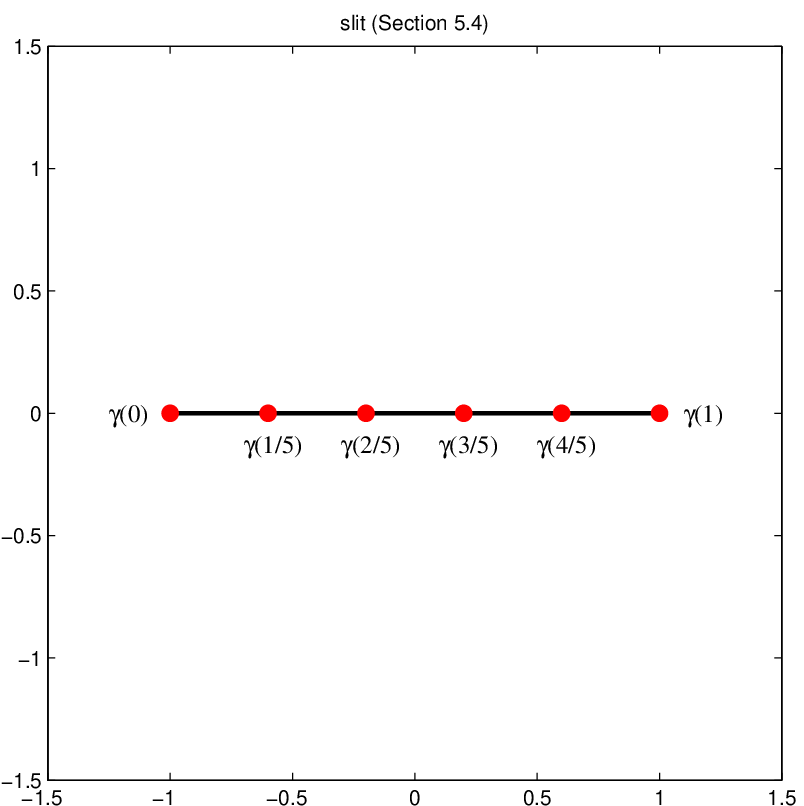}
\end{center}
\caption{Geometries and initial nodes for the experiments from Section \ref{sec:numerics}.} 
\label{fig:geometries}
\end{figure}
\subsection{Adaptive BEM for hypersingular integral equation for Neumann problem on pacman}\label{sec:hyper pacman}
We consider the boundary $\Gamma=\partial\Omega$ of the pacman geometry 
\begin{align}
\Omega:=\left\{r(\cos(\beta),\sin(\beta)):0\le r<\frac{1}{10}, \beta \in \left(-\frac{\pi}{2\tau},\frac{\pi}{2\tau}\right)\right\}
\end{align}
with $\tau=4/7$, sketched in Figure \ref{fig:geometries}.
It can be parametrized by a NURBS curve $\gamma:[0,1]\to\Gamma$ of degree two; see \cite[Section 3.2]{resigabem}.
With the 2D polar coordinates $(r,\beta)$, the function 
\begin{equation*}
P(x,y):=r^{\tau}\cos\left(\tau\beta\right)
\end{equation*}
satisfies that$-\Delta P=0$ and has a generic singularity at the origin. 
With the adjoint double-layer operator $\mathfrak{K}'$, we define  with the normal derivative $\partial_\nu P$
\begin{align*}
f:=(1/2-\mathfrak{K}')\partial_{\nu}P.
\end{align*} 
Up to an additive constant, there holds $u=P|_\Gamma$, where $u$ is the solution of the corresponding hypersingular integral equation.

For Algorithm~\ref{the algorithm}, we choose NURBS of degree two as ansatz space $\XX_\ell$ (i.e., $p=2$) and the same initial knots ${\KK}_{0}$ and weights $\mathcal{W}_0$ as for the geometry representation.

Due to numerical stability reasons, we replace  the right-hand side $f$ in each step  by $f_\ell:= (1/2-\mathfrak{K}')\phi_\ell$.
Here, $\phi_\ell$ is the $L^2(\Gamma)$-orthogonal projection of $\phi:=(\partial_{\nu}P)$ onto the space of transformed piecewise polynomials of degree $p$ on $\TT_\ell$, i.e., $\phi_\ell\circ\gamma$ is  polynomial on all $\gamma^{-1}(T)$ with $T\in\TT_\ell$.
This leads to a perturbed Galerkin approximation $U_\ell^{\rm pert}$.
To steer the algorithm, we use the weighted-residual error indicators $\eta_\ell(z)^2:= \norm{h_\ell^{1/2}(f_\ell-\mathfrak{W}U_{\ell}^{\rm pert})}{L^2(\omega_\ell(z))}^2+\norm{ h_\ell^{1/2}(\phi-\phi_\ell)}{L^2(\omega_\ell(z))}^2$.

In Figure \ref{fig:hyper pacman}, we compare the condition numbers of diagonal preconditioning with our proposed additive Schwarz approach.
Whereas diagonal preconditioning is suboptimal, we observe optimality for our approach, which numerically verifies our theoretical result in Corollary \ref{cor:main}.
This is also reflected by the number of PCG iterations.
Moreover, we plot the time needed to apply the proposed preconditioner $(\SSAS)^{-1}$ from \eqref{eq:precW} to 100 random vectors.
In accordance with Remark~\ref{rem:linear}, we observe linear complexity.

\begin{figure}[h!]
\begin{center}
\psfrag{diag}{\tiny diagonal}
\psfrag{localAS}{\tiny add.\ Schwarz}
\psfrag{O(N)}{\tiny $\mathcal{O}(N)$}
\psfrag{number of knots N}[c][c]{\tiny number of knots $N$}
\psfrag{condition number}[c][c]{\tiny condition number}
\psfrag{time in seconds}[c][c]{\tiny time in seconds}
\psfrag{number of iterations}[c][c]{\tiny number of iterations}
\includegraphics[width=0.325\textwidth,clip=true]{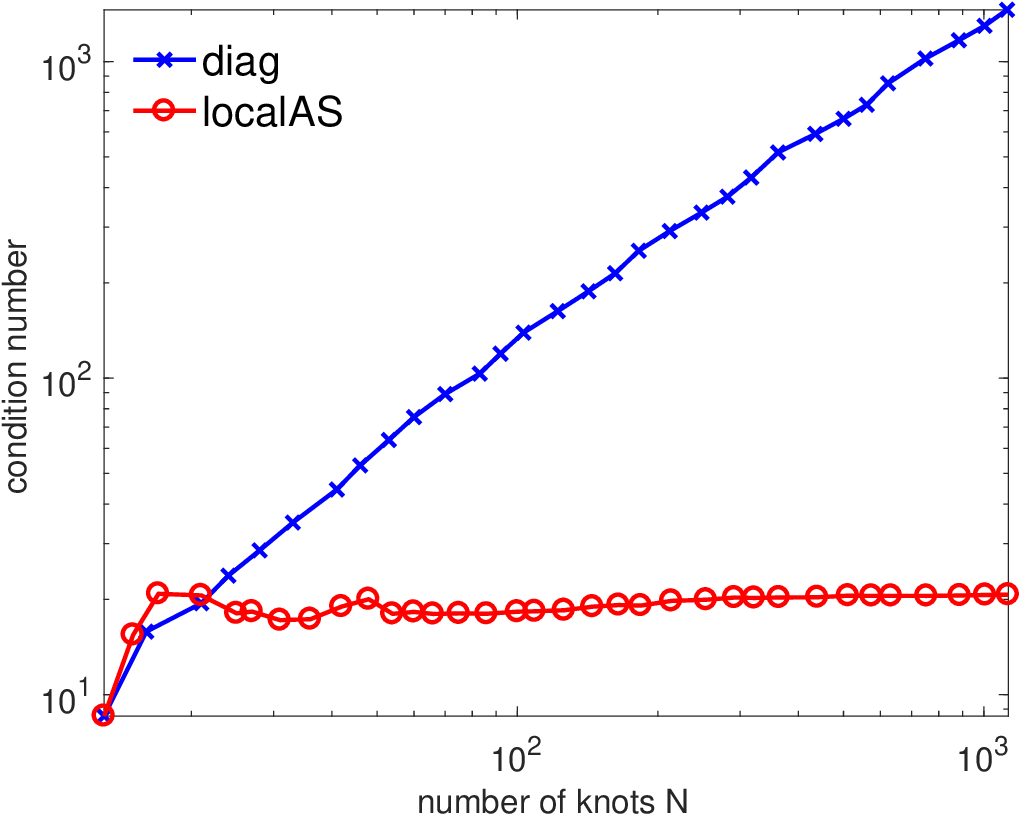}
\includegraphics[width=0.325\textwidth,clip=true]{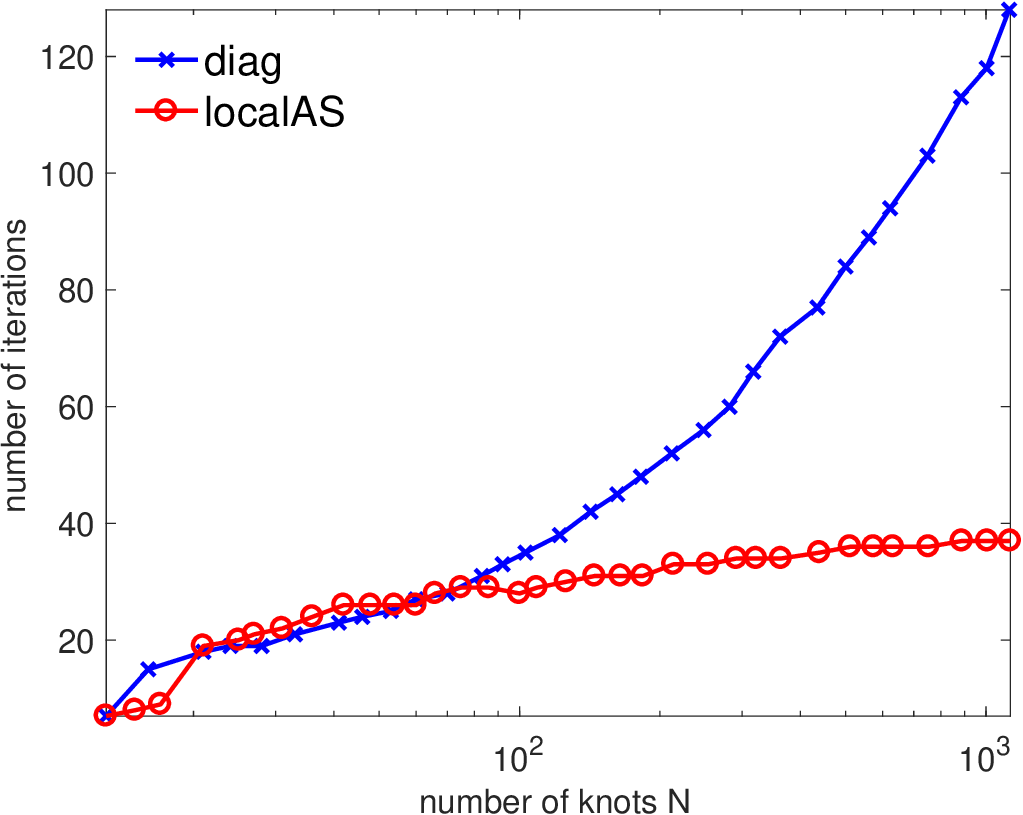}
\includegraphics[width=0.325\textwidth,clip=true]{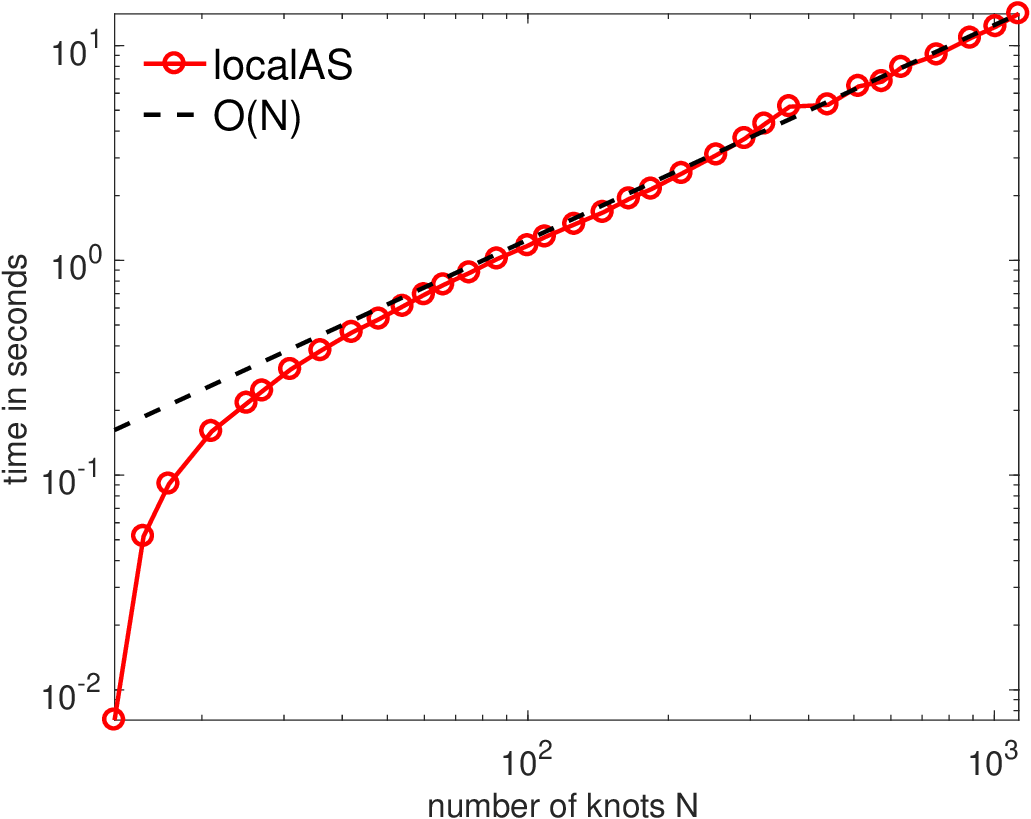}
\end{center}
\caption{Condition numbers $\lambda_{\max}/\lambda_{\min}$ of the diagonal and the additive Schwarz preconditioned Galerkin matrices, number of PCG iterations, and time to apply the additive Schwarz preconditioner to 100 random vectors for the hypersingular equation on the pacman from Section~\ref{sec:hyper pacman}.} 
\label{fig:hyper pacman}
\end{figure}

\subsection{Adaptive BEM for weakly-singular integral equation for Dirichlet problem on pacman}\label{sec:weak pacman}
Let $\Omega$ and $P$ be as in the previous section. 
With the double-layer operator $\mathfrak{K}$ and the right-hand side 
\begin{align}
g:=(1/2+\mathfrak{K}) P|_\Gamma,
\end{align}
the solution of the weakly-singular integral equation \eqref{eq:weaksing} is just the normal derivative of $P$, i.e., $\phi=\partial_{\nu}P$.
For Algorithm~\ref{the algorithm}, we choose splines  of degree two as ansatz space $\YY_\ell$ (i.e., $p=3$ and all weights are equal to one) and the initial knots $\KK_0$  as for the geometry.
To steer the algorithm, we use the weighted-residual error indicators $\eta_\ell(z):= \norm{h_\ell^{1/2} \partial_\Gamma(g-\VV\Phi_\ell)}{L^2(\omega_\ell(z)}$.
Figure~\ref{fig:weak pacman} shows a comparison of the diagonal and the additive Schwarz preconditioner and  the time needed to apply the  preconditioner $(\SSASV)^{-1}$ from \eqref{eq:precV} to 100 random vectors.

\begin{figure}[h!]
\begin{center}
\psfrag{diag}{\tiny diagonal}
\psfrag{localAS}{\tiny add.\ Schwarz}
\psfrag{O(N)}{\tiny $\mathcal{O}(N)$}
\psfrag{number of knots N}[c][c]{\tiny number of knots $N$}
\psfrag{condition number}[c][c]{\tiny condition number}
\psfrag{time in seconds}[c][c]{\tiny time in seconds}
\psfrag{number of iterations}[c][c]{\tiny number of iterations}
\includegraphics[width=0.325\textwidth,clip=true]{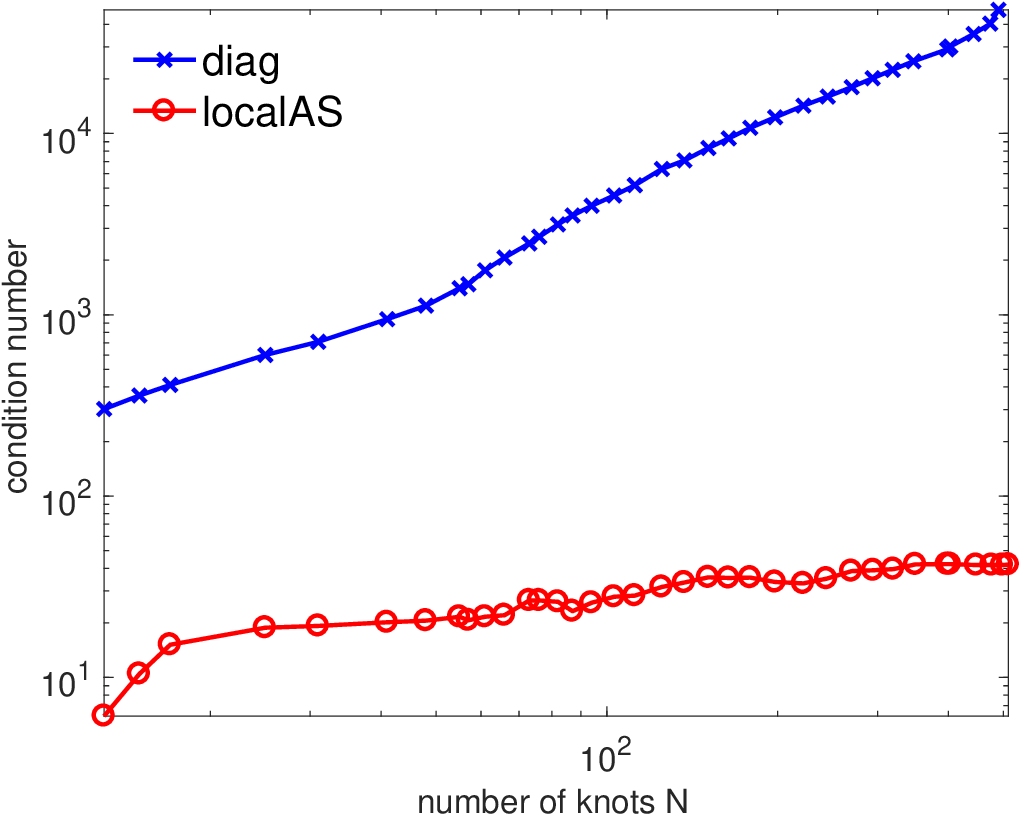}
\includegraphics[width=0.325\textwidth,clip=true]{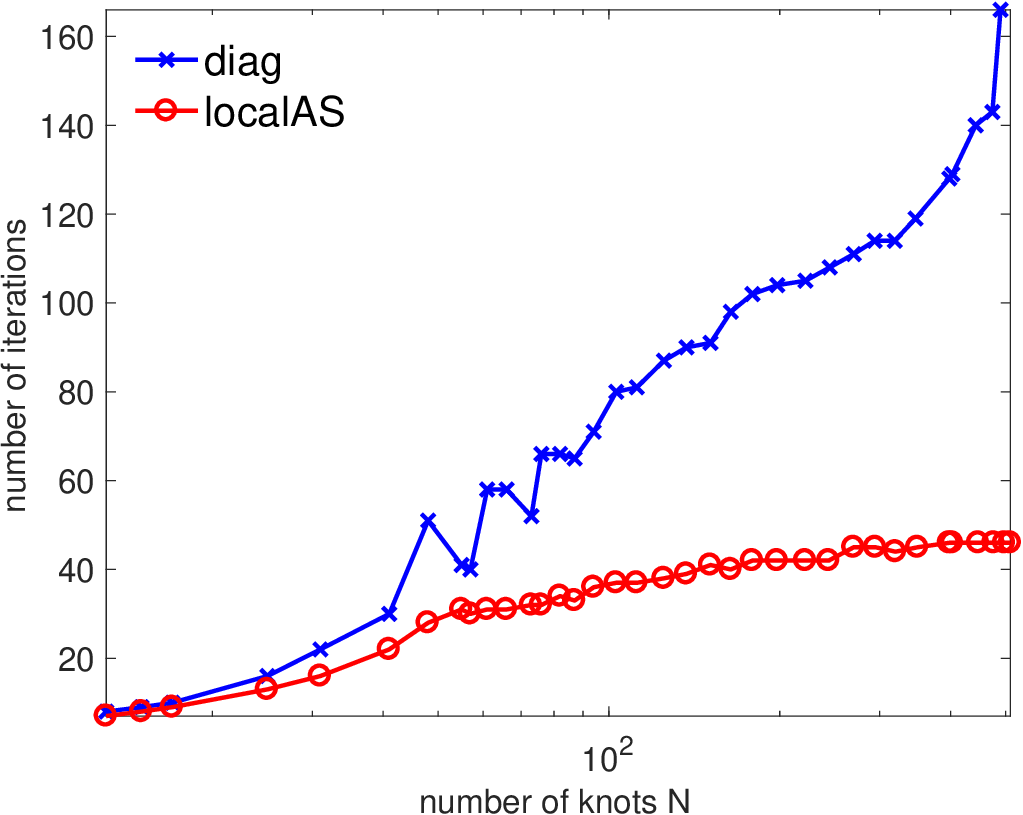}
\includegraphics[width=0.325\textwidth,clip=true]{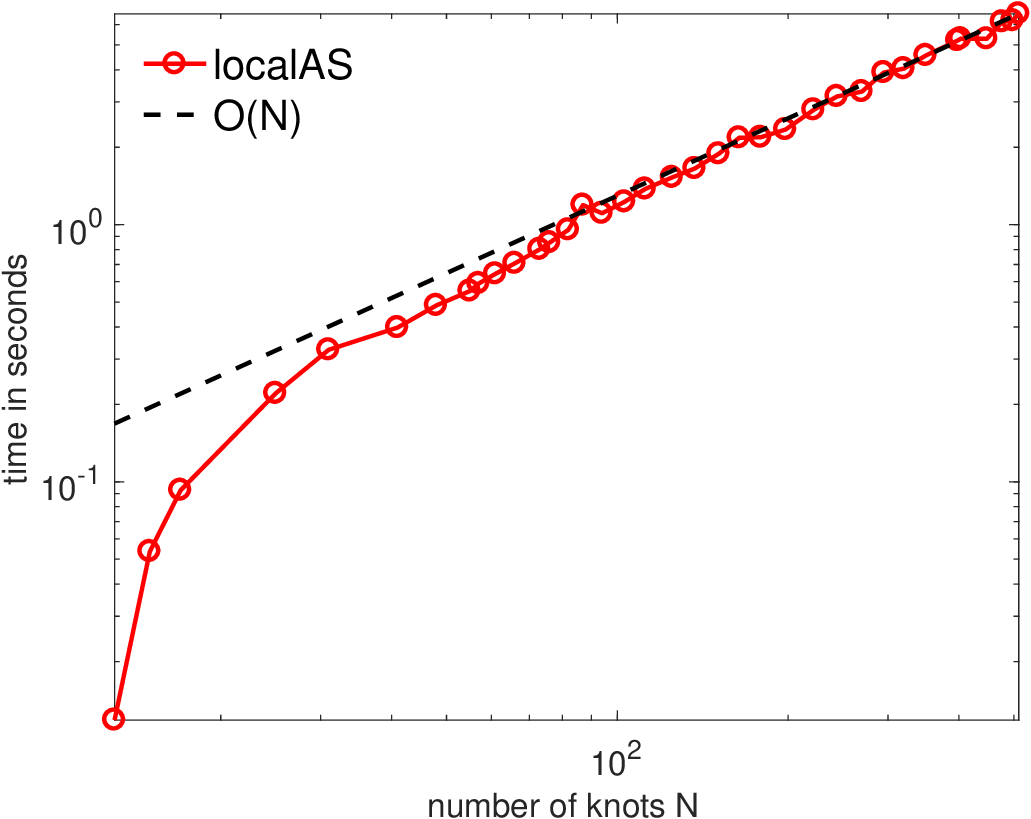}
\end{center}
\caption{Condition numbers $\lambda_{\max}/\lambda_{\min}$ of the diagonal and the additive Schwarz preconditioned Galerkin matrices, number of PCG iterations, and time to apply the additive Schwarz preconditioner to 100 random vectors  for the weakly-singular equation on the pacman from Section~\ref{sec:weak pacman}.} \label{fig:weak pacman}
\end{figure}

\subsection{Adaptive BEM for hypersingular integral equation on slit}\label{sec:hyper slit}
We consider the hypersingular integral equation on the slit $\Gamma=[-1,1]\times \{0\}$, sketched in Figure \ref{fig:geometries}, which is represented as spline curve $\gamma:[0,1]\to\Gamma$ of degree one; see \cite[Section 3.4]{resigabem}.
For $f:=1$, the exact solution is $u(x,0)=2\sqrt{1-x^2}$.
For Algorithm~\ref{the algorithm}, we choose splines of degree one as ansatz space $\XX_\ell$ (i.e., $p=1$ and all weights are equal to one) and the initial knots $\KK_0$   as for the geometry.
To steer the algorithm, we use the weighted-residual error indicators $\eta_\ell(z):= \norm{h_\ell^{1/2}(f-\mathfrak{W}U_\ell)}{L^2(\omega_\ell(z))}$.
Again, we compare diagonal preconditioning and the local multilevel diagonal preconditioner and consider the performance of the latter; see Figure \ref{fig:hyper slit}.

\begin{figure}[h!]
\begin{center}
\psfrag{diag}{\tiny diagonal}
\psfrag{localAS}{\tiny add.\ Schwarz}
\psfrag{O(N)}{\tiny $\mathcal{O}(N)$}
\psfrag{number of knots N}[c][c]{\tiny number of knots $N$}
\psfrag{condition number}[c][c]{\tiny condition number}
\psfrag{time in seconds}[c][c]{\tiny time in seconds}
\psfrag{number of iterations}[c][c]{\tiny number of iterations}
\includegraphics[width=0.325\textwidth,clip=true]{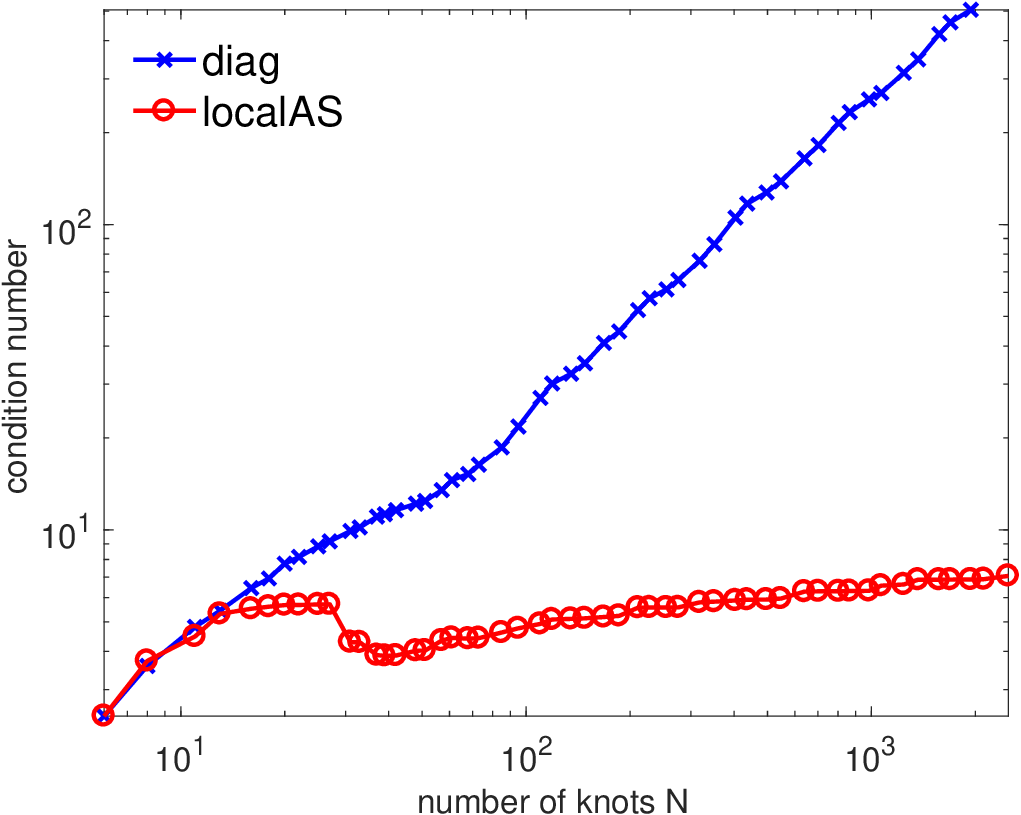}
\includegraphics[width=0.325\textwidth,clip=true]{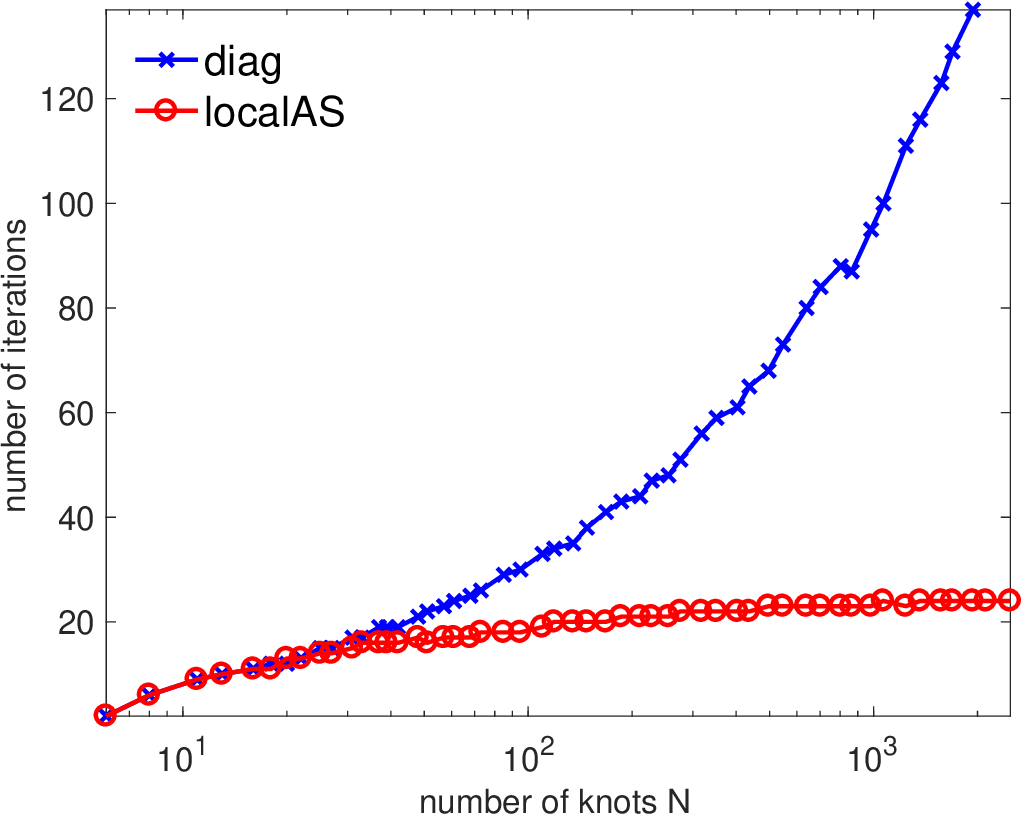}
\includegraphics[width=0.325\textwidth,clip=true]{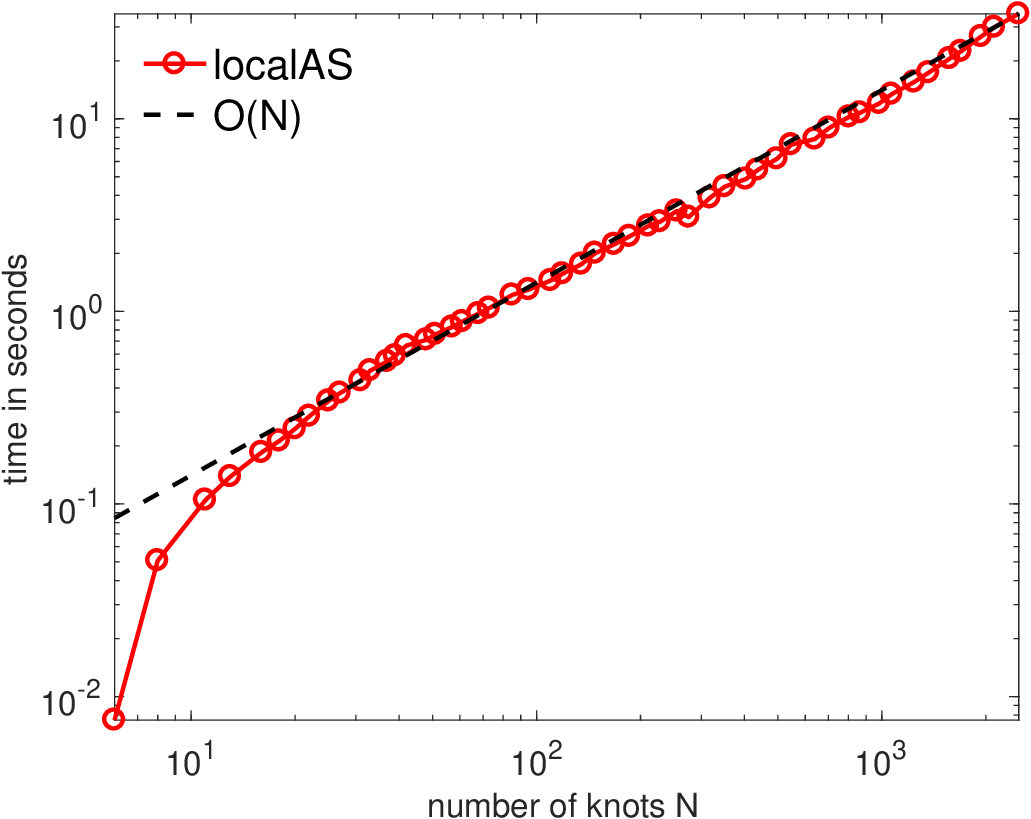}
\end{center}
\caption{Condition numbers $\lambda_{\max}/\lambda_{\min}$ of the diagonal and the additive Schwarz preconditioned Galerkin matrices, number of PCG iterations, and time to apply the additive Schwarz preconditioner to 100 random vectors 
 for the hypersingular equation on the slit from Section~\ref{sec:hyper slit}.} \label{fig:hyper slit}
\end{figure}

\subsection{Adaptive BEM for weakly-singular integral equation on slit}\label{sec:weak slit}
Let $\Gamma$ be again the slit $[-1,1]\times\{0\}$.
For the weakly-singular integral equation with $g:=-x/2$, the corresponding solution reads $\phi(x,0)=-x/\sqrt{1-x^2}$.
For Algorithm~\ref{the algorithm}, we choose splines of degree one  as ansatz space $\YY_\ell$ (i.e., $p=2$ and all weights are equal to one) and the initial knots $\KK_0$  as for the geometry.
To steer the algorithm, we use the weighted-residual error indicators $\eta_\ell(z):= \norm{h_\ell^{1/2} \partial_\Gamma(g-\VV\Phi_\ell)}{L^2(\omega_\ell(z))}$.
Figure \ref{fig:weak slit} compares the diagonal and the additive Schwarz preconditioner.
Moreover, the performance of the latter is  illustrated.

\begin{figure}[h!]
\begin{center}
\psfrag{diag}{\tiny diagonal}
\psfrag{localAS}{\tiny add.\ Schwarz}
\psfrag{O(N)}{\tiny $\mathcal{O}(N)$}
\psfrag{number of knots N}[c][c]{\tiny number of knots $N$}
\psfrag{condition number}[c][c]{\tiny condition number}
\psfrag{time in seconds}[c][c]{\tiny time in seconds}
\psfrag{number of iterations}[c][c]{\tiny number of iterations}
\includegraphics[width=0.325\textwidth,clip=true]{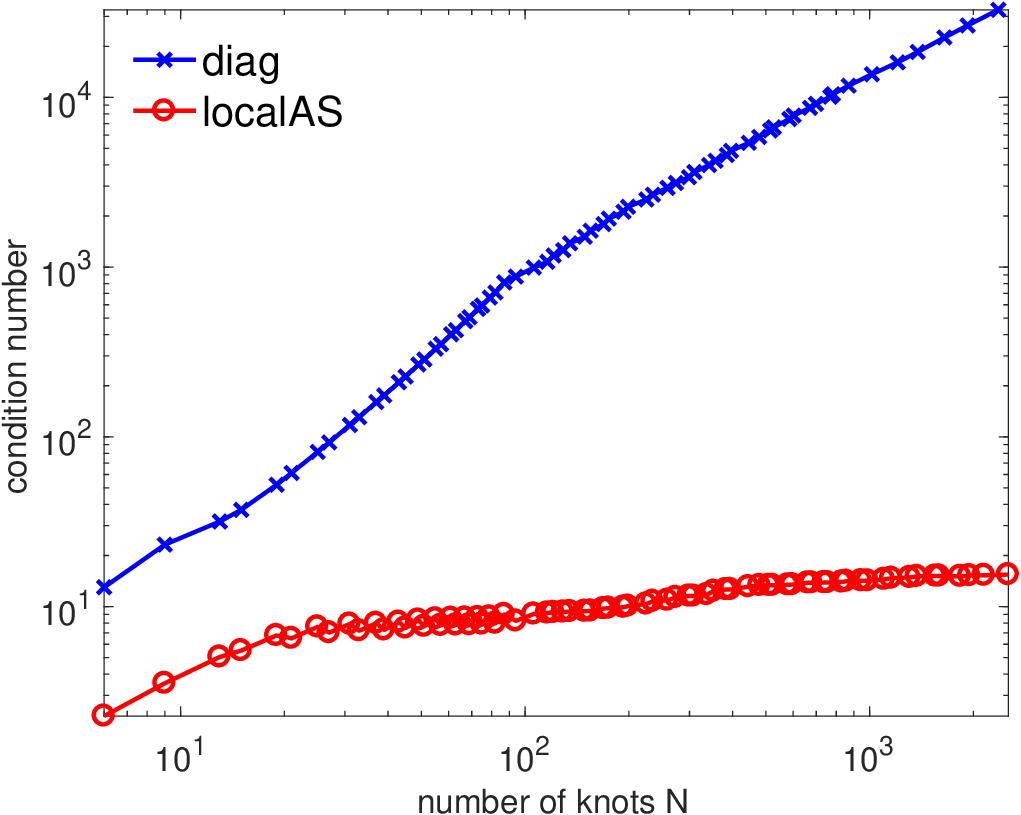}
\includegraphics[width=0.325\textwidth,clip=true]{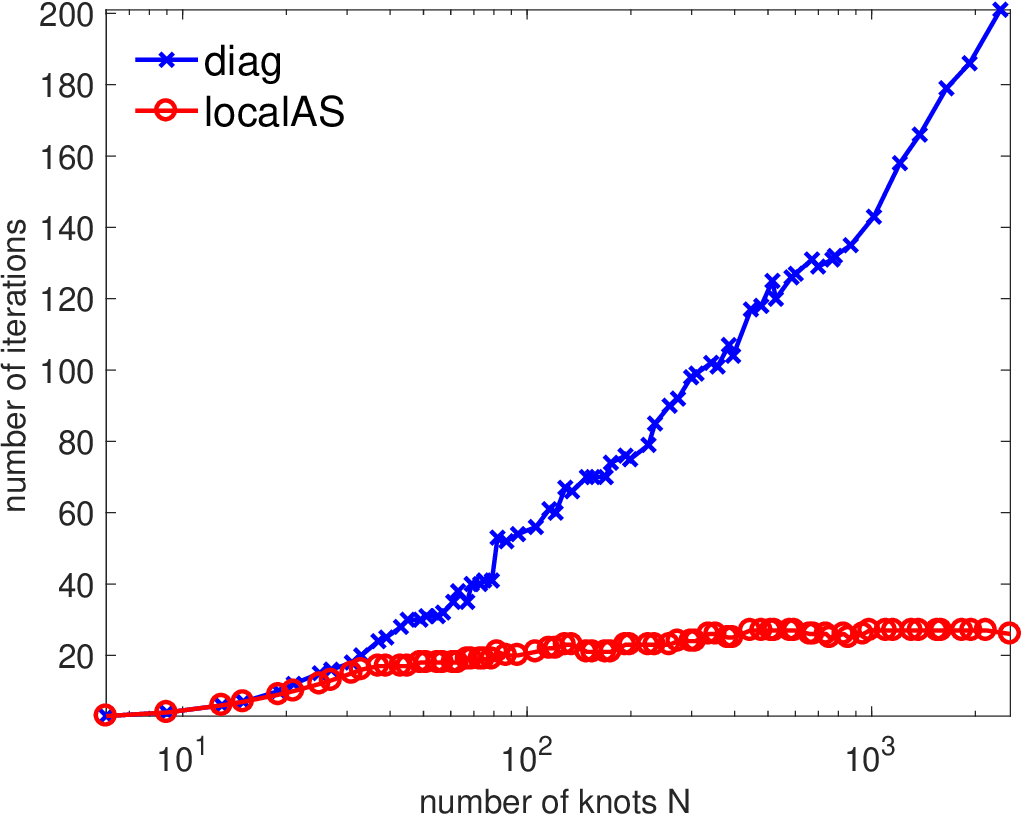}
\includegraphics[width=0.325\textwidth,clip=true]{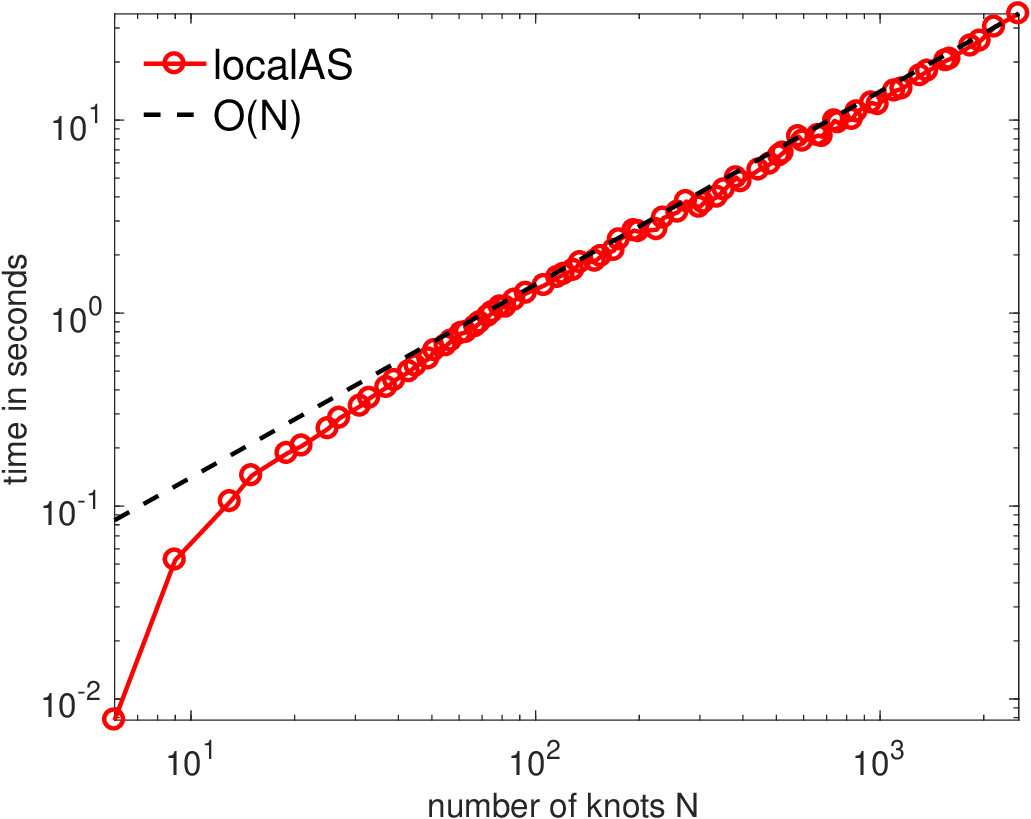}
\end{center}
\caption{Condition numbers $\lambda_{\max}/\lambda_{\min}$ of the diagonal and the additive Schwarz preconditioned Galerkin matrices, number of PCG iterations, and time to apply the additive Schwarz preconditioner to 100 random vectors for the weakly-singular equation on the slit  from Section~\ref{sec:weak slit}.} \label{fig:weak slit}
\end{figure}

\bibliographystyle{alpha}
\bibliography{literature}

\newcommand{\etalchar}[1]{$^{#1}$}
\begin{thebibliography}{BdVPS{\etalchar{+}}17}

\bibitem[ACD{\etalchar{+}}18]{sampoli}
Alessandra Aimi, Francesco Calabr{\`o}, Mauro Diligenti, Maria~L. Sampoli,
  Giancarlo Sangalli, and Alessandra Sestini.
\newblock Efficient assembly based on {B}-spline tailored quadrature rules for
  the {I}g{A}-{SGBEM}.
\newblock {\em Comput. Methods Appl. Mech. Engrg.}, 331:327--342, 2018.

\bibitem[AEF{\etalchar{+}}14]{hilbert}
Markus Aurada, Michael Ebner, Michael Feischl, Samuel Ferraz-Leite, Thomas
  F{\"u}hrer, Petra Goldenits, Michael Karkulik, Markus Mayr, and Dirk
  Praetorius.
\newblock {H}{I}{L}{B}{E}{R}{T} --- a {M}{A}{T}{L}{A}{B} implementation of
  adaptive 2{D}-{B}{E}{M}.
\newblock {\em Numer. Algorithms}, 67(1):1--32, 2014.

\bibitem[AFF{\etalchar{+}}15]{hypsing3d}
Markus Aurada, Michael Feischl, Thomas F\"uhrer, Michael Karkulik, and Dirk
  Praetorius.
\newblock Energy norm based error estimators for adaptive {BEM} for
  hypersingular integral equations.
\newblock {\em Appl. Numer. Math.}, 95:250--270, 2015.

\bibitem[AM03]{amcl03}
Mark Ainsworth and William McLean.
\newblock Multilevel diagonal scaling preconditioners for boundary element
  equations on locally refined meshes.
\newblock {\em Numer. Math.}, 93(3):387--413, 2003.

\bibitem[AMT99]{amt99}
Mark Ainsworth, William McLean, and Thanh Tran.
\newblock The conditioning of boundary element equations on locally refined
  meshes and preconditioning by diagonal scaling.
\newblock {\em SIAM J. Numer. Anal.}, 36(6):1901--1932, 1999.

\bibitem[BdVBSV14]{overview}
Lourenco Beir{\~a}o~da Veiga, Annalisa Buffa, Giancarlo Sangalli, and Rafael
  V{{\'a}}zquez.
\newblock Mathematical analysis of variational isogeometric methods.
\newblock {\em Acta Numer.}, 23:157--287, 2014.

\bibitem[BdVCPS13]{MR3037302}
Lourenco Beir{\~a}o~da Veiga, Durkbin Cho, Luca~F. Pavarino, and Simone
  Scacchi.
\newblock B{DDC} preconditioners for isogeometric analysis.
\newblock {\em Math. Models Methods Appl. Sci.}, 23(6):1099--1142, 2013.

\bibitem[BdVPS{\etalchar{+}}14]{MR3216651}
Lourenco Beir{\~a}o~da Veiga, Luca~F. Pavarino, Simone Scacchi, Olof~B.
  Widlund, and Stefano Zampini.
\newblock Isogeometric {BDDC} preconditioners with deluxe scaling.
\newblock {\em SIAM J. Sci. Comput.}, 36(3):A1118--A1139, 2014.

\bibitem[BdVPS{\etalchar{+}}17]{MR3612901}
Lourenco Beir{\~a}o~da Veiga, Luca~F. Pavarino, Simone Scacchi, Olof~B.
  Widlund, and Stefano Zampini.
\newblock Adaptive selection of primal constraints for isogeometric {BDDC}
  deluxe preconditioners.
\newblock {\em SIAM J. Sci. Comput.}, 39(1):A281--A302, 2017.

\bibitem[BHKS13]{bhksBPX}
Annalisa Buffa, Helmut Harbrecht, Angela Kunoth, and Giancarlo Sangalli.
\newblock B{PX}-preconditioning for isogeometric analysis.
\newblock {\em Comput. Methods Appl. Mech. Engrg.}, 265:63--70, 2013.

\bibitem[Bor94]{bornemann}
Folkmar~A. Bornemann.
\newblock Interpolation spaces and optimal multilevel preconditioners.
\newblock {\em Contemp. Math.}, 180:3--8, 1994.

\bibitem[Cao02]{cao}
Thang Cao.
\newblock Adaptive-additive multilevel methods for hypersingular integral
  equation.
\newblock {\em Appl. Anal.}, 81(3):539--564, 2002.

\bibitem[CHB09]{bible}
J.~Austin Cottrell, Thomas J.~R. Hughes, and Yuri Bazilevs.
\newblock {\em Isogeometric analysis: toward integration of {C}{A}{D} and
  {F}{E}{A}}.
\newblock John Wiley \& Sons, New York, 2009.

\bibitem[CV17]{BPXadapIGA}
Durkbin Cho and Rafael V\'{a}zquez.
\newblock {BPX} preconditioners for isogeometric analysis using
  analysis-suitable {T}-splines.
\newblock Technical Report 17--01, {IMATI} report series, 2017.

\bibitem[dB86]{Boor-SplineBasics}
Carl de~Boor.
\newblock {\em B (asic)-spline basics}.
\newblock Mathematics Research Center, University of Wisconsin-Madison, 1986.

\bibitem[DHK{\etalchar{+}}18]{wolf18}
J{\"u}rgen D{\"o}lz, Helmut Harbrecht, Stefan Kurz, Sebastian Sch{\"o}ps, and
  Felix Wolf.
\newblock A fast isogeometric {BEM} for the three dimensional {L}aplace- and
  {H}elmholtz problems.
\newblock {\em Comput. Methods Appl. Mech. Engrg.}, 330:83--101, 2018.

\bibitem[DHP16]{dhp16}
J{\"u}rgen D{\"o}lz, Helmut Harbrecht, and Michael Peters.
\newblock An interpolation-based fast multipole method for higher-order
  boundary elements on parametric surfaces.
\newblock {\em Internat. J. Numer. Methods Engrg.}, 108(13):1705--1728, 2016.

\bibitem[DKSW18]{wolf_new}
J{\"u}rgen D{\"o}lz, Stefan Kurz, Sebastian Sch{\"o}ps, and Felix Wolf.
\newblock Isogeometric boundary elements in electromagnetism: {R}igorous
  analysis, fast methods, and examples.
\newblock {\em arXiv preprint}, 1807.03097, 2018.

\bibitem[FFPS15]{ffpsfembem}
Michael Feischl, Thomas F{\"u}hrer, Dirk Praetorius, and Ernst~P. Stephan.
\newblock Optimal preconditioning for the symmetric and nonsymmetric coupling
  of adaptive finite elements and boundary elements.
\newblock {\em Numer. Methods Partial Differential Equations}, 2015.

\bibitem[FFPS17]{ffps}
Michael Feischl, Thomas F{\"u}hrer, Dirk Praetorius, and Ernst~P. Stephan.
\newblock Optimal additive {S}chwarz preconditioning for hypersingular integral
  equations on locally refined triangulations.
\newblock {\em Calcolo}, 54(1):367--399, 2017.

\bibitem[FGHP16]{resigabem}
Michael Feischl, Gregor Gantner, Alexander Haberl, and Dirk Praetorius.
\newblock Adaptive 2{D} {I}{G}{A} boundary element methods.
\newblock {\em Eng. Anal. Bound. Elem.}, 62:141--153, 2016.

\bibitem[FGHP17]{optigabem}
Michael Feischl, Gregor Gantner, Alexander Haberl, and Dirk Praetorius.
\newblock Optimal convergence for adaptive {IGA} boundary element methods for
  weakly-singular integral equations.
\newblock {\em Numer. Math.}, 136(1):147--182, 2017.

\bibitem[FGK{\etalchar{+}}18]{giannelli_new}
Antonella Falini, Carlotta Giannelli, Tadej Kandu{\v{c}}, Maria~L. Sampoli, and
  Alessandra Sestini.
\newblock An adaptive {I}g{A}-{BEM} with hierarchical {B}-splines based on
  quasi-interpolation quadrature schemes.
\newblock {\em Internat. J. Numer. Methods Engrg.}, 2018.

\bibitem[FGP15]{igafaermann}
Michael Feischl, Gregor Gantner, and Dirk Praetorius.
\newblock Reliable and efficient a posteriori error estimation for adaptive
  {I}{G}{A} boundary element methods for weakly-singular integral equations.
\newblock {\em Comput. Methods Appl. Mech. Engrg.}, 290:362--386, 2015.

\bibitem[FHPS18]{optpcg}
Thomas F{\"u}hrer, Alexander Haberl, Dirk Praetorius, and Stefan Schimanko.
\newblock Adaptive {BEM} with inexact {PCG} solver yields almost optimal
  computational costs.
\newblock {\em Numer. Math.}, published onine first, 2018.

\bibitem[FK18]{falini2018study}
Antonella Falini and Tadej Kanduc.
\newblock A study on spline quasi-interpolation based quadrature rules for the
  isogeometric {G}alerkin {BEM}.
\newblock {\em arXiv preprint}, 1807.11277, 2018.

\bibitem[FMPR15]{fmpr}
Thomas F{\"u}hrer, J.~Markus Melenk, Dirk Praetorius, and Alexander Rieder.
\newblock Optimal additive {S}chwarz methods for the hp-{BEM}: The
  hypersingular integral operator in 3{D} on locally refined meshes.
\newblock {\em Comput. Math. Appl.}, 70(7):1583--1605, 2015.

\bibitem[F{\"u}h14]{dissfuehrer}
Thomas F{\"u}hrer.
\newblock {\em {Z}ur {K}opplung von finiten {E}lementen und {R}andelementen}.
\newblock PhD thesis, TU Wien, 2014.

\bibitem[Gan14]{diplarbeit}
Gregor Gantner.
\newblock Isogeometric adaptive {B}{E}{M}.
\newblock Master's thesis, TU Wien, 2014.

\bibitem[Gan17]{diss}
Gregor Gantner.
\newblock {\em Optimal adaptivity for splines in finite and boundary element
  methods}.
\newblock PhD thesis, TU Wien, 2017.

\bibitem[GHS05]{inverse2}
Ivan~G. Graham, Wolfgang Hackbusch, and Stefan~A. Sauter.
\newblock Finite elements on degenerate meshes: inverse-type inequalities and
  applications.
\newblock {\em IMA J. Numer. Anal.}, 25(2):379--407, 2005.

\bibitem[GKT13]{MR3002801}
Krishan P.~S. Gahalaut, Johannes~K. Kraus, and Satyendra~K. Tomar.
\newblock Multigrid methods for isogeometric discretization.
\newblock {\em Comput. Methods Appl. Mech. Engrg.}, 253:413--425, 2013.

\bibitem[GM06]{gm06}
Ivan~G. Graham and William McLean.
\newblock Anisotropic mesh refinement: the conditioning of {G}alerkin boundary
  element matrices and simple preconditioners.
\newblock {\em SIAM J. Numer. Anal.}, 44(4):1487--1513, 2006.

\bibitem[GPS19]{hypiga}
Gregor Gantner, Dirk Praetorius, and Stefan Schimanko.
\newblock Adaptive isogeometric boundary element methods with local smoothness
  control.
\newblock {\em arXiv preprint}, 1903.01830, 2019.

\bibitem[HAD14]{stokesiga}
Luca Heltai, Marino Arroyo, and Antonio DeSimone.
\newblock Nonsingular isogeometric boundary element method for {S}tokes flows
  in 3{D}.
\newblock {\em Comput. Methods Appl. Mech. Engrg.}, 268:514--539, 2014.

\bibitem[HCB05]{pioneer}
Thomas J.~R. Hughes, J.~Austin Cottrell, and Yuri Bazilevs.
\newblock Isogeometric analysis: {C}{A}{D}, finite elements, {N}{U}{R}{B}{S},
  exact geometry and mesh refinement.
\newblock {\em Comput. Methods Appl. Mech. Engrg.}, 194(39):4135--4195, 2005.

\bibitem[Heu96]{heuer}
Norbert Heuer.
\newblock Efficient algorithms for the p-version of the boundary element
  method.
\newblock {\em J. Integral Equations Appl.}, 8(3):337--361, 1996.

\bibitem[HJHUT14]{hjhut2014}
Ralf Hiptmair, Carlos Jerez-Hanckes, and Carolina Urz\'ua-Torres.
\newblock Mesh-independent operator preconditioning for boundary elements on
  open curves.
\newblock {\em SIAM J. Numer. Anal.}, 52(5):2295--2314, 2014.

\bibitem[HJHUT16]{HJU16_646}
Ralf Hiptmair, Carlos Jerez-Hanckes, and Carolina Urz{\'u}a-Torres.
\newblock Optimal operator preconditioning for hypersingular operator over 3{D}
  screens.
\newblock {\em Seminar for Applied Mathematics, ETH Z{\"u}rich}, 2016(09),
  2016.

\bibitem[HJHUT17]{HJU17_709}
Ralf Hiptmair, Carlos Jerez-Hanckes, and Carolina Urz{\'u}a-Torres.
\newblock Optimal operator preconditioning for weakly singular operator over
  3{D} screens.
\newblock {\em Seminar for Applied Mathematics, ETH Z{\"u}rich}, 2017(13),
  2017.

\bibitem[HJKZ16]{MR3522271}
Clemens Hofreither, Bert J\"uttler, G\'abor Kiss, and Walter Zulehner.
\newblock Multigrid methods for isogeometric analysis with {THB}-splines.
\newblock {\em Comput. Methods Appl. Mech. Engrg.}, 308:96--112, 2016.

\bibitem[HR10]{cad2wave}
Helmut Harbrecht and Maharavo Randrianarivony.
\newblock From computer aided design to wavelet {BEM}.
\newblock {\em Comput. Vis. Sci.}, 13(2):69--82, 2010.

\bibitem[HTZ17]{MR3610090}
Clemens Hofreither, Stefan Takacs, and Walter Zulehner.
\newblock A robust multigrid method for {I}sogeometric {A}nalysis in two
  dimensions using boundary correction.
\newblock {\em Comput. Methods Appl. Mech. Engrg.}, 316:22--42, 2017.

\bibitem[KHZvE17]{keuchel}
S{\"o}ren Keuchel, Nils~Christian Hagelstein, Olgierd Zaleski, and Otto von
  Estorff.
\newblock Evaluation of hypersingular and nearly singular integrals in the
  isogeometric boundary element method for acoustics.
\newblock {\em Comput. Methods Appl. Mech. Engrg.}, 325:488--504, 2017.

\bibitem[Lio88]{lions88}
Pierre-Louis Lions.
\newblock {\em On the {S}chwarz alternating method. {I}}.
\newblock SIAM, Philadelphia, 1988.

\bibitem[McL00]{mclean}
William McLean.
\newblock {\em Strongly elliptic systems and boundary integral equations}.
\newblock Cambridge University Press, Cambridge, 2000.

\bibitem[MZBF15]{zechner}
Benjamin Marussig, J{\"u}rgen Zechner, Gernot Beer, and Thomas-Peter Fries.
\newblock Fast isogeometric boundary element method based on independent field
  approximation.
\newblock {\em Comput. Methods Appl. Math.}, 284:458--488, 2015.

\bibitem[NZW{\etalchar{+}}17]{tran}
B.~H. Nguyen, Xiaoying Zhuang, Peter Wriggers, Timon Rabczuk, M.~E. Mear, and
  Han~D. Tran.
\newblock Isogeometric symmetric {G}alerkin boundary element method for
  three-dimensional elasticity problems.
\newblock {\em Comput. Methods Appl. Mech. Engrg.}, 323:132--150, 2017.

\bibitem[PGK{\etalchar{+}}09]{igabem2d}
Costas Politis, Alexandros~I. Ginnis, Panagiotis~D. Kaklis, Kostas
  Belibassakis, and Christian Feurer.
\newblock An isogeometric {BEM} for exterior potential-flow problems in the
  plane.
\newblock {\em Proceedings of SIAM/ACM Joint Conference on Geometric and
  Physical Modeling}, pages 349--354, 2009.

\bibitem[PTC13]{helmholtziga}
Michael~J. Peake, Jon Trevelyan, and Graham Coates.
\newblock Extended isogeometric boundary element method ({XIBEM}) for
  two-dimensional {H}elmholtz problems.
\newblock {\em Comput. Methods Appl. Mech. Engrg.}, 259:93--102, 2013.

\bibitem[Saa03]{saad03}
Yousef Saad.
\newblock {\em Iterative methods for sparse linear systems}.
\newblock SIAM, Philadelphia, 2003.

\bibitem[SBLT13]{simpson}
Robert~N. Simpson, St\'ephane P.~A. Bordas, Haojie Lian, and Jon Trevelyan.
\newblock An isogeometric boundary element method for elastostatic analysis:
  2{D} implementation aspects.
\newblock {\em Comput. \& Structures}, 118:2--12, 2013.

\bibitem[SBTR12]{SBTR}
Robert~N. Simpson, St\'ephane P.~A. Bordas, Jon Trevelyan, and Timon Rabczuk.
\newblock A two-dimensional isogeometric boundary element method for
  elastostatic analysis.
\newblock {\em Comput. Methods Appl. Mech. Engrg.}, 209/212:87--100, 2012.

\bibitem[SS86]{saad}
Yousef Saad and Martin~H. Schultz.
\newblock G{MRES}: a generalized minimal residual algorithm for solving
  nonsymmetric linear systems.
\newblock {\em SIAM J. Sci. Statist. Comput.}, 7(3):856--869, 1986.

\bibitem[SS11]{ss}
Stefan~A. Sauter and Christoph Schwab.
\newblock {\em Boundary element methods}.
\newblock Springer, Berlin, 2011.

\bibitem[SSE{\etalchar{+}}13]{igabem3d}
Michael~A. Scott, Robert~N. Simpson, John~A. Evans, Scott Lipton, St\'ephane
  P.~A. Bordas, Thomas J.~R. Hughes, and Thomas~W. Sederberg.
\newblock Isogeometric boundary element analysis using unstructured
  {T}-splines.
\newblock {\em Comput. Methods Appl. Mech. Engrg.}, 254:197--221, 2013.

\bibitem[ST16]{sangalli16}
Giancarlo Sangalli and Mattia Tani.
\newblock Isogeometric preconditioners based on fast solvers for the
  {S}ylvester equation.
\newblock {\em SIAM J. Sci. Comput.}, 38(6):A3644--A3671, 2016.

\bibitem[Ste08]{s}
Olaf Steinbach.
\newblock {\em Numerical approximation methods for elliptic boundary value
  problems}.
\newblock Springer, New York, 2008.

\bibitem[SvV18]{stevenson18}
Rob Stevenson and Raymond van Veneti{\"e}.
\newblock Optimal preconditioning for problems of negative order.
\newblock {\em arXiv preprint}, 1803.05226, 2018.

\bibitem[SW98]{sw98}
Olaf Steinbach and Wolfgang~L. Wendland.
\newblock The construction of some efficient preconditioners in the boundary
  element method.
\newblock {\em Adv. Comput. Math.}, 9(1-2):191--216, 1998.
\newblock Numerical treatment of boundary integral equations.

\bibitem[Tak17]{takacs17}
Stefan Takacs.
\newblock Robust approximation error estimates and multigrid solvers for
  isogeometric multi-patch discretizations.
\newblock {\em Math. Models Methods Appl. Sci.}, 2017.

\bibitem[TM12]{TM}
Toru Takahashi and Toshiro Matsumoto.
\newblock An application of fast multipole method to isogeometric boundary
  element method for {L}aplace equation in two dimensions.
\newblock {\em Eng. Anal. Bound. Elem.}, 36(12):1766--1775, 2012.

\bibitem[TS96]{transtep96}
Thanh Tran and Ernst~P. Stephan.
\newblock Additive {S}chwarz methods for the {$h$}-version boundary element
  method.
\newblock {\em Appl. Anal.}, 60(1-2):63--84, 1996.

\bibitem[TSM97]{tsm97}
Thanh Tran, Ernst~P. Stephan, and Patrick Mund.
\newblock Hierarchical basis preconditioners for first kind integral equations.
\newblock {\em Appl. Anal.}, 65(3-4):353--372, 1997.

\bibitem[TW05]{ToselliWidlund}
Andrea Toselli and Olof Widlund.
\newblock {\em Domain decomposition methods---algorithms and theory}.
\newblock Springer Series in Computational Mathematics. Springer, Berlin, 2005.

\bibitem[Wid89]{wid89}
Olof~B. Widlund.
\newblock Optimal iterative refinement methods.
\newblock SIAM, Philadelphia, 1989.

\bibitem[Yse86]{yserentant}
Harry Yserentant.
\newblock On the multi-level splitting of finite element spaces.
\newblock {\em Numer. Math.}, 49(4):379--412, 1986.

\bibitem[Zha92]{zhang92}
Xuejun Zhang.
\newblock Multilevel {S}chwarz methods.
\newblock {\em Numer. Math.}, 63(4):521--539, 1992.

\end{thebibliography}

\end{document}